\documentclass[12pt]{amsart}

\usepackage{etoolbox}

\makeatletter
\let\old@tocline\@tocline
\let\section@tocline\@tocline
\newcommand{\subsection@dotsep}{4.5}
\newcommand{\subsubsection@dotsep}{4.5}
\patchcmd{\@tocline}
  {\hfil}
  {\nobreak
     \leaders\hbox{$\m@th
        \mkern \subsection@dotsep mu\hbox{.}\mkern \subsection@dotsep mu$}\hfill
     \nobreak}{}{}
\let\subsection@tocline\@tocline
\let\@tocline\old@tocline

\patchcmd{\@tocline}
  {\hfil}
  {\nobreak
     \leaders\hbox{$\m@th
        \mkern \subsubsection@dotsep mu\hbox{.}\mkern \subsubsection@dotsep mu$}\hfill
     \nobreak}{}{}
\let\subsubsection@tocline\@tocline
\let\@tocline\old@tocline

\let\old@l@subsection\l@subsection
\let\old@l@subsubsection\l@subsubsection

\def\@tocwriteb#1#2#3{%
  \begingroup
    \@xp\def\csname #2@tocline\endcsname##1##2##3##4##5##6{%
      \ifnum##1>\c@tocdepth
      \else \sbox\z@{##5\let\indentlabel\@tochangmeasure##6}\fi}%
    \csname l@#2\endcsname{#1{\csname#2name\endcsname}{\@secnumber}{}}%
  \endgroup
  \addcontentsline{toc}{#2}%
    {\protect#1{\csname#2name\endcsname}{\@secnumber}{#3}}}%

\newlength{\@tocsectionindent}
\newlength{\@tocsubsectionindent}
\newlength{\@tocsubsubsectionindent}
\newlength{\@tocsectionnumwidth}
\newlength{\@tocsubsectionnumwidth}
\newlength{\@tocsubsubsectionnumwidth}
\newcommand{\settocsectionnumwidth}[1]{\setlength{\@tocsectionnumwidth}{#1}}
\newcommand{\settocsubsectionnumwidth}[1]{\setlength{\@tocsubsectionnumwidth}{#1}}
\newcommand{\settocsubsubsectionnumwidth}[1]{\setlength{\@tocsubsubsectionnumwidth}{#1}}
\newcommand{\settocsectionindent}[1]{\setlength{\@tocsectionindent}{#1}}
\newcommand{\settocsubsectionindent}[1]{\setlength{\@tocsubsectionindent}{#1}}
\newcommand{\settocsubsubsectionindent}[1]{\setlength{\@tocsubsubsectionindent}{#1}}

\renewcommand{\l@section}{\section@tocline{1}{\@tocsectionvskip}{\@tocsectionindent}{}{\@tocsectionformat}}%
\renewcommand{\l@subsection}{\subsection@tocline{2}{\@tocsubsectionvskip}{\@tocsubsectionindent}{}{\@tocsubsectionformat}}%
\renewcommand{\l@subsubsection}{\subsubsection@tocline{3}{\@tocsubsubsectionvskip}{\@tocsubsubsectionindent}{}{\@tocsubsubsectionformat}}%
\newcommand{\@tocsectionformat}{}
\newcommand{\@tocsubsectionformat}{}
\newcommand{\@tocsubsubsectionformat}{}
\expandafter\def\csname toc@1format\endcsname{\@tocsectionformat}
\expandafter\def\csname toc@2format\endcsname{\@tocsubsectionformat}
\expandafter\def\csname toc@3format\endcsname{\@tocsubsubsectionformat}
\newcommand{\settocsectionformat}[1]{\renewcommand{\@tocsectionformat}{#1}}
\newcommand{\settocsubsectionformat}[1]{\renewcommand{\@tocsubsectionformat}{#1}}
\newcommand{\settocsubsubsectionformat}[1]{\renewcommand{\@tocsubsubsectionformat}{#1}}
\newlength{\@tocsectionvskip}
\newcommand{\settocsectionvskip}[1]{\setlength{\@tocsectionvskip}{#1}}
\newlength{\@tocsubsectionvskip}
\newcommand{\settocsubsectionvskip}[1]{\setlength{\@tocsubsectionvskip}{#1}}
\newlength{\@tocsubsubsectionvskip}
\newcommand{\settocsubsubsectionvskip}[1]{\setlength{\@tocsubsubsectionvskip}{#1}}

\patchcmd{\tocsection}{\indentlabel}{\makebox[\@tocsectionnumwidth][l]}{}{}
\patchcmd{\tocsubsection}{\indentlabel}{\makebox[\@tocsubsectionnumwidth][l]}{}{}
\patchcmd{\tocsubsubsection}{\indentlabel}{\makebox[\@tocsubsubsectionnumwidth][l]}{}{}

\newcommand{\@sectypepnumformat}{}
\renewcommand{\contentsline}[1]{%
  \expandafter\let\expandafter\@sectypepnumformat\csname @toc#1pnumformat\endcsname%
  \csname l@#1\endcsname}
\newcommand{\@tocsectionpnumformat}{}
\newcommand{\@tocsubsectionpnumformat}{}
\newcommand{\@tocsubsubsectionpnumformat}{}
\newcommand{\setsectionpnumformat}[1]{\renewcommand{\@tocsectionpnumformat}{#1}}
\newcommand{\setsubsectionpnumformat}[1]{\renewcommand{\@tocsubsectionpnumformat}{#1}}
\newcommand{\setsubsubsectionpnumformat}[1]{\renewcommand{\@tocsubsubsectionpnumformat}{#1}}
\renewcommand{\@tocpagenum}[1]{%
  \hfill {\mdseries\@sectypepnumformat #1}}

\let\oldappendix\appendix
\renewcommand{\appendix}{%
  \leavevmode\oldappendix%
  \addtocontents{toc}{%
    \protect\settowidth{\protect\@tocsectionnumwidth}{\protect\@tocsectionformat\sectionname\space}%
    \protect\addtolength{\protect\@tocsectionnumwidth}{2em}}%
}
\makeatother

\makeatletter
\settocsectionnumwidth{2em}
\settocsubsectionnumwidth{2.5em}
\settocsubsubsectionnumwidth{3em}
\settocsectionindent{1pc}%
\settocsubsectionindent{\dimexpr\@tocsectionindent+\@tocsectionnumwidth}%
\settocsubsubsectionindent{\dimexpr\@tocsubsectionindent+\@tocsubsectionnumwidth}%
\makeatother

\settocsectionvskip{6pt}
\settocsubsectionvskip{0pt}
\settocsubsubsectionvskip{0pt}
    
\settocsectionformat{\bfseries}
\settocsubsectionformat{\mdseries}
\settocsubsubsectionformat{\mdseries}
\setsectionpnumformat{\bfseries}
\setsubsectionpnumformat{\mdseries}
\setsubsubsectionpnumformat{\mdseries}

\let\oldtableofcontents\tableofcontents
\renewcommand{\tableofcontents}{%
  \vspace*{-\linespacing}
  \oldtableofcontents}

\usepackage{graphicx}
\usepackage{amssymb}
\usepackage{amsmath, amscd}
\usepackage{dbnsymb}

\usepackage[bookmarks=true, bookmarksopen=true,%
bookmarksdepth=3,bookmarksopenlevel=2,%
colorlinks=true,%
linkcolor=blue,%
citecolor=blue,%
filecolor=blue,%
menucolor=blue,%
urlcolor=blue]{hyperref}

\newtheorem{thm}{Theorem}[section]

\newtheorem{theorem}[thm]{Theorem}
\newtheorem{proposition}[thm]{Proposition}
\newtheorem{corollary}[thm]{Corollary}
\newtheorem{lemma}[thm]{Lemma}

\theoremstyle{definition}
\newtheorem{definition}[thm]{Definition}
\newtheorem{dfn}[thm]{Definition}

\theoremstyle{remark}
\newtheorem{remark}[thm]{Remark}
\newtheorem{rmk}[thm]{Remark}
\newtheorem{example}[thm]{Example}

\numberwithin{equation}{section}

\newcommand{\R}{{\mathbb{R}}}
\newcommand{\C}{{\mathbb{C}}}
\newcommand{\Q}{{\mathbb{Q}}}
\newcommand{\Z}{{\mathbb{Z}}}
\renewcommand{\P}{{\mathbb{P}}}

\newcommand{\Oo}{{\mathcal{O}}}

\newcommand{\Jj}{{\mathcal{J}}}
\newcommand{\Mm}{{\mathcal{M}}}

\newcommand{\Ordo}{{\mathbf{O}}}

\newcommand{\M}{{\mathcal{M}}}

\newcommand{\inr}{\operatorname{int}}
\newcommand{\End}{\operatorname{End}}
\newcommand{\Hom}{\operatorname{Hom}}

\newcommand{\ev}{\operatorname{ev}}

\newcommand{\pa}{\partial}

\newcommand{\ind}{\operatorname{index}}

\newcommand{\Sk}{\operatorname{Sk}}

\usepackage[margin=1in,marginparwidth=0.8in, marginparsep=0.1in]{geometry}

\begin{document}

\title{Skeins on branes}
\author{Tobias Ekholm}
\address{Department of mathematics, Uppsala University, Box 480, 751 06 Uppsala, Sweden \and
Institut Mittag-Leffler, Aurav 17, 182 60 Djursholm, Sweden}
\email{tobias.ekholm@math.uu.se}
\author{Vivek Shende}
\address{Center for Quantum Mathematics, Syddansk Univ., Campusvej 55
5230 Odense Denmark \and 
Department of mathematics, UC Berkeley, 970 Evans Hall,
Berkeley CA 94720 USA}
\email{vivek.vijay.shende@gmail.com}

\thanks{TE is supported by the Knut and Alice Wallenberg Foundation, KAW2020.0307 Wallenberg Scholar and by the Swedish Research Council, VR 2022-06593, Centre of Excellence in Geometry and Physics at Uppsala University and VR 2020-04535, project grant. \\ \indent VS is supported by NSF  CAREER DMS-1654545 and Villum Fonden Villum Investigator  37814.}

\begin{abstract}

We study 1-parameter families of holomorphic curves with Lagrangian boundary in Calabi-Yau 3-folds.  We show that the expected codimension one phenomena can be organized to match the HOMFLYPT skein relations from quantum topology.  It follows that counting holomorphic curves by the class of their boundaries in the skein module of the Lagrangian gives a deformation invariant result. 
This is a mathematically rigorous incarnation of Witten's assertion that boundaries of open topological strings create line defects in Chern-Simons theory \cite{Witten}. 

Using this theory, we rigorously establish the following prediction of Ooguri and Vafa: the coefficients of the HOMFLYPT polynomial of a link in the three-sphere count the holomorphic curves in the resolved conifold, with  boundary on (a push-off of) the link conormal.
\end{abstract}

\maketitle
\thispagestyle{empty}

\newpage
\renewcommand\contentsname{\vspace*{-25pt}}
\tableofcontents
\newpage

\begin{figure}
	\centering
	\includegraphics[width=.4\linewidth]{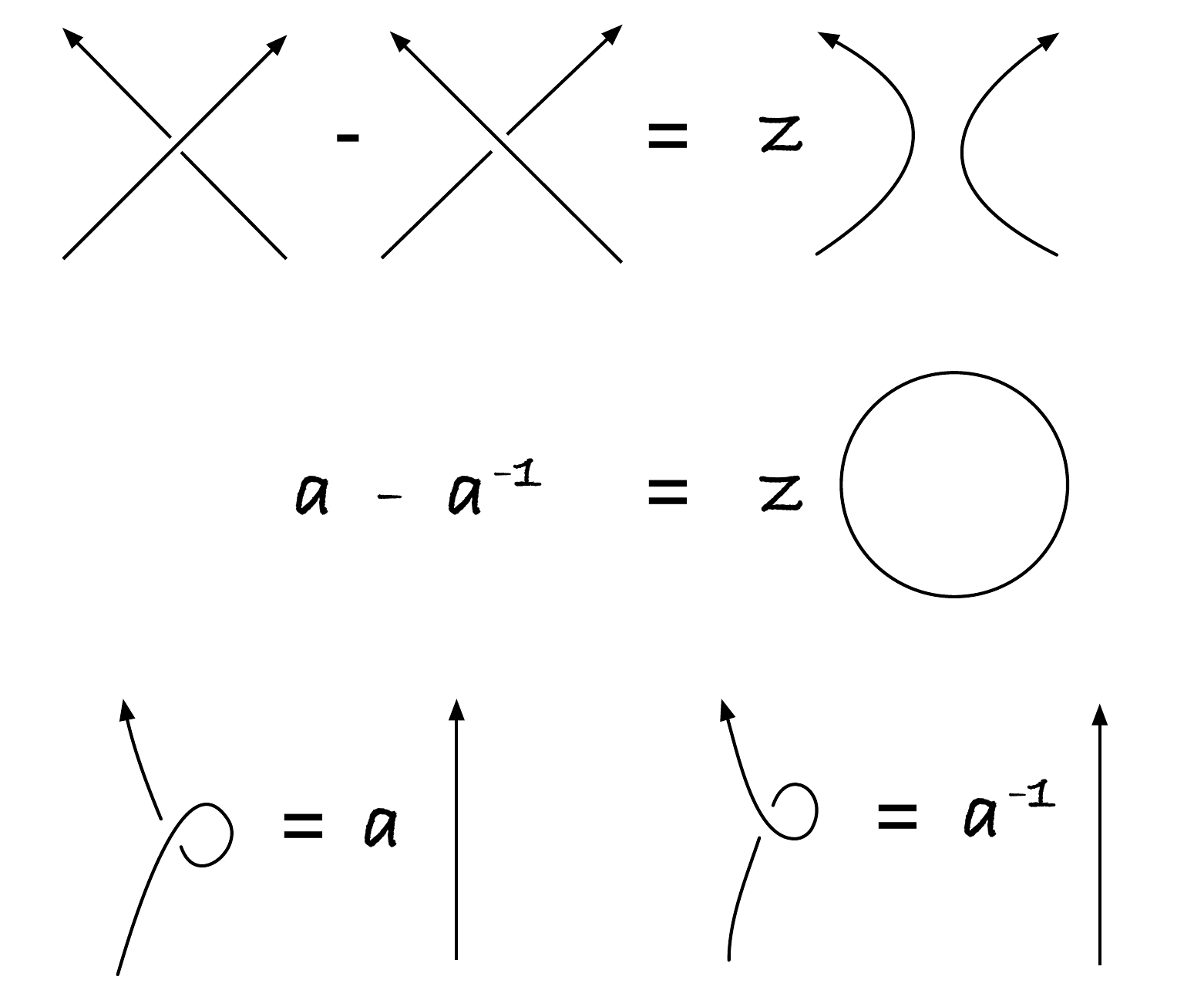}
	\caption{\label{fig:framedskein}The framed HOMFLYPT skein relations.}\label{fig:skeinrel}
\end{figure}

\section{Introduction}\label{Intro}

The skein relations in Figure \ref{fig:skeinrel} serve to define
and calculate the HOMFLYPT invariant of knots and links \cite{HOMFLY, PT}.  Said relations have celebrated interpretations as 
encoding the expected properties of Wilson lines in the Chern-Simons quantum field theory \cite{Witten-QFT}, or correspondingly of traces on braid group representations arising from the study of von Neumann algebras \cite{Jones, Jones2} or quantum groups \cite{Reshetikhin, Turaev-invariants}. 
Our purpose here is to show that the same skein relations emerge
from the study of holomorphic curves with Lagrangian boundary conditions in Calabi-Yau 3-folds, and, consequently, resolve certain long-standing problems regarding the definition and computation of counts of such curves.

\subsection{Skein relations from moduli  of holomorphic maps} 
Let $X$ be a symplectic Calabi-Yau 3-fold, i.e., a symplectic 6-manifold with trivial first Chern class $c_1(X)=0$ and fixed complex trivialization of $\det(TX)$, and $L \subset X$ an oriented Lagrangian with vanishing Maslov class, i.e., the rotation of $\det(L)$ in $\det(X)$ equals zero. Let $J$ be an almost complex structure compatible with the symplectic structure on $X$.

The moduli space of $J$-holomorphic curves in $(X, L)$ has expected dimension zero, 
so one might expect  to be able to define a count of such curves.  
However, moduli spaces of holomorphic curves with boundary themselves have codimension one boundaries that appear in 1-parameter families, so 
one naively does not expect deformation invariance of counts of curves in $(X,L)$, but rather some wall-and-chamber structure with 
counts locally constant in the chambers but changing across walls \cite{Fukaya-Tohoku}.   In fact the situation is even more involved: 
achieving transversality of moduli depends on some choice of perturbation, and for the same reason one would 
expect the resulting counts to depend on the choice of perturbation by some wall and chamber structure.  

The driving observation of the present article is that one can set up this counting problem so that the possible wall crossings precisely match the HOMFLYPT skein relations. Let us sketch how the skein relations appear. 

The generic or expected boundary point of moduli is a map 
$$
u_0\colon (S_0, \partial S_0) \to (X, L),
$$
where the domain curve $S_0$ has a single boundary node. 
There are two kinds of node, `elliptic' and `hyperbolic'.  In 1-parameter degenerations, locally in a $(\C^3,\R^3)$-chart of $(X,L)$, the 
nearby $u_\epsilon( S_\epsilon)$, $\epsilon\in [0,\epsilon_0)$, look as follows, 
\begin{align*}
    &\text{elliptic:} &&  u_\epsilon( S_\epsilon) = \{(x,y,z)\in \C^3 \colon x^2 + y^2 = \epsilon,\, z=0,\, \mathrm{Re}(x)\mathrm{Im}(y)\ge 0\}, \\ 
    &\text{hyperbolic:} &&   u_\epsilon( S_\epsilon) = \{(x,y,z)\in \C^3 \colon x^2 - y^2 = \epsilon,\, z=0,\, \mathrm{Im}(y)\ge 0\}.
\end{align*}
(In Sections \ref{ssec: ell boundary} and \ref{ssec: hyp boundary} we parameterize the moduli space by $r=\frac{1}{-\pi\log(\epsilon)}$, which is natural from the point of view of Floer gluing, see Theorem \ref{t : wall crossing gluing}.)

\begin{figure}
	\centering
	\includegraphics[width=.60\linewidth]{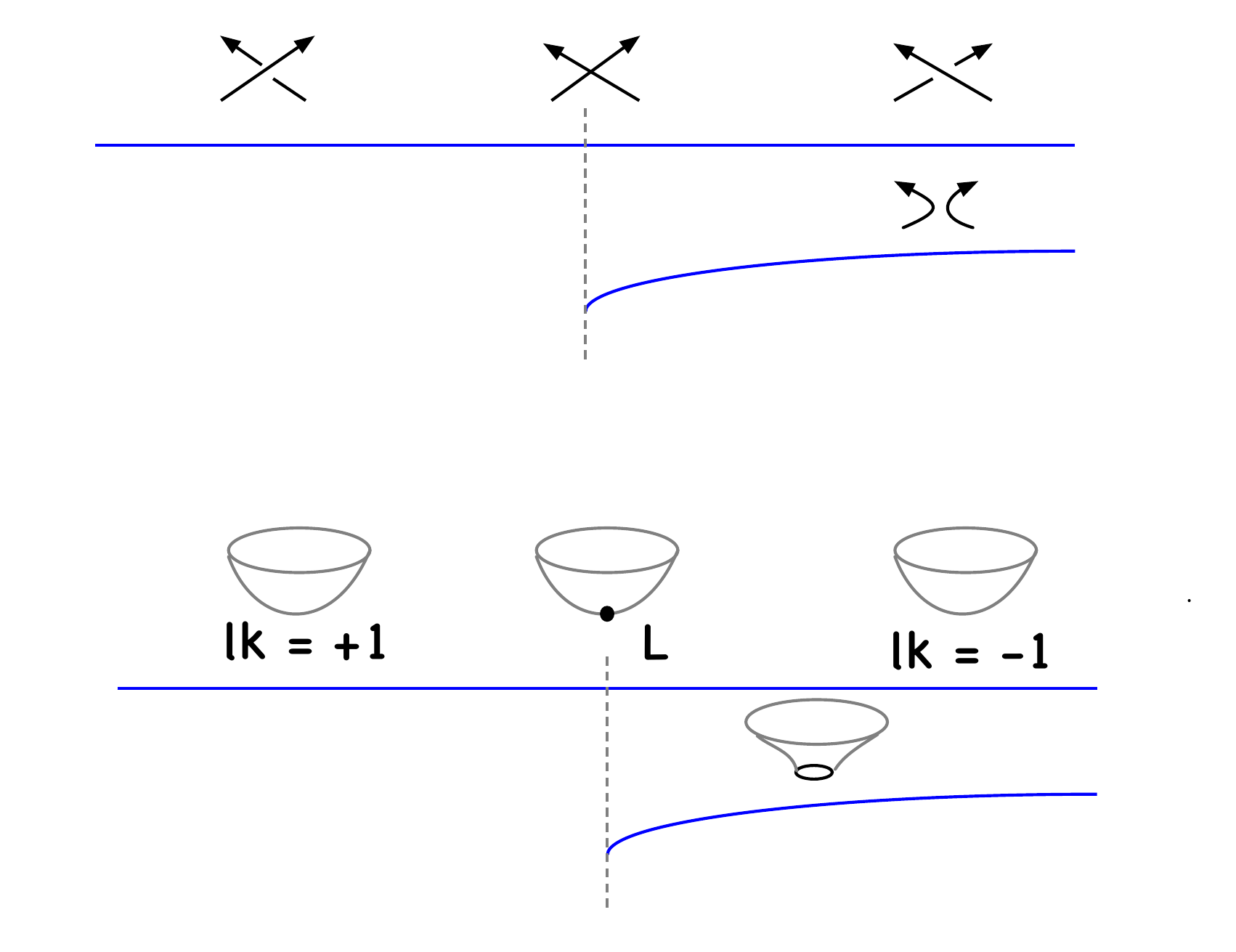}
	\caption{Moduli spaces near nodal curves, top hyperbolic and bottom elliptic. Solid blue lines indicates parameterized moduli spaces of curves. Over the line that goes from left to right we see the curves $\widetilde{u}_t$ near the crossing of boundary arcs in the hyperbolic case and with the Lagrangian in the elliptic case. Over the line with boundary in the middle we see the curves $u_t$ that becomes nodal at $t=0$. The grey dashed line is the wall whose `wall crossing formula' -- arising from the indicated association of points in distinct moduli spaces -- is a skein relation.}\label{fig:nodes}
\end{figure}

Consider a (generic) 1-parameter moduli space whose boundary corresponds to a curve with hyperbolic node, see Figure \ref{fig:nodes}. Then
$u_0(\partial S_0)$ is some singular
link, and $u_\epsilon(\partial S_\epsilon)$ is its resolution appearing on the right hand side of the first
skein relation in Figure \ref{fig:skeinrel}.  On the other hand, if we write $\widetilde{u}_0 \colon (\widetilde{S}_0, \partial \widetilde{S}_0) \to (X, L)$ 
for the map from the normalization, then $[\widetilde{u}_0]$ is an interior point in a 1-parameter moduli space.  Letting 
$\widetilde{u}_{\pm}\colon (\widetilde{S}_\pm, \partial \widetilde{S}_\pm) \to (X, L)$ denote nearby points on either side, $\widetilde{u}_{\pm} (\partial \widetilde{S}_\pm)$ provide the two knots on the left hand side of the first skein relation in Figure \ref{fig:skeinrel}, see Theorem \ref{t : wall crossing gluing}. 

We write $\chi(\,\cdot\,)$ for the Euler characteristic of the curve obtained by smoothing all nodes, sometimes called the arithmetic Euler characteristic, for which $\chi(S_{\pm}) = \chi(S_0) + 1$.  
Thus if we count curves weighted by $z^{-\chi}$, then {\em considering together the families $u_t$ and $\widetilde u_t$}, assuming transversality and possibly adjusting the sign of $z$, the curve count jumps by some multiple of the first skein relation.

The moduli spaces in the case of an elliptic node are similar.  The elliptic node can be viewed as a marked point constrained to lie on the Lagrangian which formally lowers the Euler characteristic of the curve by 1; normalization forgets this exotic marked point.  
The boundaries of the normalization family of curves satisfy $u_+(\partial S_+) \approx  u_- (\partial S_-)$ and neither intersects our $(\C^3,\R^3)$-chart.  Meanwhile  $u_0(\partial S_0)$ is isotopic to a union of 
$u_+(\partial S_+)$ with the single point where the elliptic node appears.  In $u_\epsilon (\partial S_\epsilon)$, 
this point expands to an unknot, see Theorem \ref{t : wall crossing gluing}.

Thus, if we count the maps weighted by $a^{u \OpenHopf L}$, where $u \OpenHopf L$ is some quantity satisfying
$$u_- \OpenHopf L + 1 \ = \ u_\epsilon \OpenHopf L \ = \ u_+ \OpenHopf L - 1,$$
then the wall crossing matches the second HOMFLYPT skein relation in Figure \ref{fig:skeinrel}.
Observe that the 1-parameter family connecting $u_+(S_+)$ and $u_-(S_-)$
is a family of 2-folds in a 6-fold, which at one moment (namely $\widetilde{u}_0$) meets the 3-fold $L$.  Thus we may expect to distinguish $u_+(S_+)$ and $u_-(S_-)$ by their linking number with $L$.  
Since we want in particular $u_+ \OpenHopf L - u_- \OpenHopf L  = 2$, we would like to
 choose
a 4-chain $C$ with $\partial C = 2L$, and define a linking number by the intersection number with this chain: $$u \OpenHopf L := u \cdot C.$$

However, the quantity $u \cdot C$ cannot be expected to be invariant under deformations.  Indeed,  
$u_+(S_+)$ and $u_-(S_-)$ may have boundary on $L$, and  $u_\epsilon(S_\epsilon)$ necessarily does.  
Thus, the intersection points of $u$ and $C$  may travel to the boundary of the domain curve, and then disappear.  
We correct for this as follows (compare \cite{ekholmslk}). We demand that the 4-chain $C$ is smooth near the boundary and that there is a vector field $\xi$ on $L$ such that the inward normal of $C$ along $L$ equals $C = \pm J \xi$.  Since $u$ is $J$-holomorphic, this 
ensures that points of $\partial u \cap C \ne 0$ are precisely those points where $\xi$
is parallel to $\partial C$.  Thus if we use $\xi$ to frame $u(\partial S)$, the moments when $a^{u \cdot C}$ changes are precisely
those when the framing of $u(\partial S)$ changes.  This is the third HOMFLYPT skein relation in Figure \ref{fig:skeinrel}.

\subsection{Skein-valued curve counting} 
Let $\Sk(L)$ be the skein module of $L$, i.e., the formal $\Z[a^{\pm}, z^{\pm}]$-linear span of framed links
in $L$, modulo the HOMFLYPT skein relations.   
For a map $u\colon (S, \partial S) \to (X, L)$ such that $u|_{\partial S}$ is an embedding, 
there is a well defined class $\langle\partial u\rangle  \in \Sk(L)$.  We fix a vector field $\xi$ and 4-chain $C$ as above in order to give framings to curves so as to define skein classes, and to define the linking number $\OpenHopf$. We write $\widehat{\Sk}(L)$ for the completion in $z^{-1}$.  
The above arguments suggest that the following curve count is deformation invariant (although its individual factors are not): 
\begin{equation} \label{the invariant} 
Z_{X,L}=\sum_{(u,S)\in\mathcal{M}} z^{-\chi (S)} \cdot a^{u \OpenHopf L} \cdot \langle\partial u\rangle \ \in \ \widehat{\Sk}(L),
\end{equation} 
where $\mathcal{M}$ is a suitable moduli space of curves.

Let us note that the skein relation does not preserve the topological type of the curves, and 
thus in order to obtain an invariant we will be forced to sum over such topological types.  On the other hand,  we may and often will restrict attention to maps of a fixed class in $H_2(X, L)$. 

As usual for curve counting, we also need our moduli of maps to be suitably compact and and transversely cut out.  To set up our skein valued counting, we will need in addition the following properties of the maps themselves.  First, for the boundary of our maps to define links in $L$, we will want $u|_{\partial S}$ to be an embedding.  In addition, in order that
wall crossings in parameter space directly match the skein relations, we will want that in 1-parameter families, $u(\partial S)$ is a singular link with one double point.  These conditions are what is expected from naive dimension counting.  

When, as for knot conormals, the Lagrangian has topology $S^1\times\R^2$, then curves with boundary in a basic homology class, see Definition \ref{def : basic class}, are somewhere injective, see Lemmas \ref{l : somwhere injectivity exact}, \ref{l : somewhere injective non-exact} for details. It then follows that for generic almost complex structure (as we recall in Lemma \ref{l:basicsurjectivity}) that the desired general position holds away from a locus of maps of the expected codimension, and we then have these desired properties for such curves.

We recall that a $J$-holomorphic map fails to be somewhere injective only if some component either maps by a multiple cover, or to a constant.   
Let us first fix attention on some situation where multiple covers are excluded a priori, e.g., by fixing some class in $H_2(X, L)$ where multiple covers cannot occur for topological reasons. We consider the remaining difficulty: maps which collapse some components
(`ghost bubbles').  We term maps without ghost bubbles `bare'.  In \cite{ghost} we show that a limit of bare maps can only develop a ghost bubble if the image of the non-contracted component has a singularity worse than an 
ordinary (see \cite{ionel-genus1, zinger-sharp, niu2016refined, doan-walpulski-embedded} for similar results).  Since such singularities do not appear for generic $J$ or generic paths of $J$, the locus of bare curves is  compact.  

In the present article, we detail the above reasoning, and deduce in Theorem \ref{thm:invariance} that \eqref{the invariant} is well defined and invariant under deformations when multiple covers are excluded a priori.  As we will see presently, there are already striking applications of our setup in situations where no multiple covers can occur.

In the sequel \cite{bare}, we use the techniques of the abstract perturbation theory of \cite{HWZ, HWZ-GW} to construct a similar theory  the presence of multiple covers, thereby establishing the well-definedness and invariance of \eqref{the invariant} in full generality, which is required for various further applications \cite{ekholm-shende-unknot, ekholm-shende-colored, scharitzer-shende-mirror-cpgl, scharitzer-shende-cluster, hu-schrader-zaslow, ekholm-longhi-nakamura-hopf, ekholm-longhi-shende, ekholm-longhi-park-shende,  hu-shende-strips}.  

\begin{remark}
The skein-valued curve counting is a mathematical shadow
of Witten's proposal that Lagrangian branes in the topological A-model carry the Chern-Simons gauge 
theory \cite{Witten}.  Witten reasoned at the level of the action functional, by
the cubic vertex of the open string field theory to the Chern-Simons functional, and then further argued directly that open strings
contribute Wilson line operators.  By contrast, what we are doing (in more physical terms) is showing that if the open topological
$A$-model is to have the expected background independence, then the contributions of the open strings must satisfy the same relations
as do the Chern-Simons Wilson lines, without ever computing what the contribution of such a string actually is, or indeed discussing
to what it should contribute.   
\end{remark}

\subsection{The Ooguri-Vafa conjecture} 
Using our setup, we will establish a proposal from the string theory literature \cite{GV-geometry, GV-gravity, OV}:
the coefficients of the HOMFLYPT invariant are counts of certain holomorphic curves.  

Let us  recall some relevant notions to state this precisely. 
The cotangent bundle $T^{\ast} S^{3}$ and the total space $X$ of the bundle $\mathcal{O}(-1)\oplus\mathcal{O}(-1)\to\C\P^{1}$ are 
respectively a deformation and a resolution of a quadric cone in $\C^{4}$.  
The cone itself is understood as the limit where respectively the radius of $S^{3}$ or the area of $\C \P^{1}$ tends to zero.
The space $X$ is called the \emph{resolved conifold}.  Passing between $T^\ast S^3$ and $X$ is the \emph{conifold transition}.  


Let $K$ be a $m$-component link in $S^3$ and let $L_{K;0} \subset T^\ast S^3$ be the conormal of $K$. Shifting $L_{K;0}$ along a closed 1-form dual to the tangent vector of $K$, by some $\epsilon > 0$, gives a non-exact Lagrangian $L_{K;\epsilon}$ disjoint from the 0-section.  
Topologically, $L_{K;\epsilon}$ is $\bigsqcup_{1}^m S^{1}\times\R^{2}$. 
The shift of $L_{K;\epsilon}$ determines a positive generator $\beta$ of $H_{1}(L_{K;\epsilon})$. 
Using the conifold transition, we identify $L_{K;\epsilon}$ with a Lagrangian in $X$ and denote it $L_{K}\subset X$.
%

Let $J$ be a generic almost complex structure on $X$ that is $\R$-invariant at infinity and standard near $\C\P^1$, see Section \ref{sec : somewhere inj} for details. Assume that $L_{K}$ lies in the $\R$-invariant region and let $\Gamma= \bigsqcup_1^m S^{1}\times\{0\}\in \Sk(L_{K})$, the class in the skein of $L_{K}$ represented by the union of positively oriented central circles, one for each component. Fix a $4$-chain for $L_{K}$ such the intersection number with $\C\P^1$ equals zero, $\C\P^1\OpenHopf L_{K}=0$.

\begin{thm}\label{thm:basic}
    Let $\M$ be the moduli of bare connected $J$-holomorphic curves $(u,S)$ in $X$ with boundary in $L_K=L_{K;\epsilon}$ in a basic class $\beta$. If either $k=1$ or if $k>1$ and $\epsilon$ is sufficiently small then for an almost complex structure $J$ as above, $\M$ is a transversely cut out compact oriented 0-manifold. Curves in $\M$ are embeddings with interior disjoint from $L_{K}$, the skein valued count of curves in $\M$ is invariant and equals
    \begin{equation}\label{eq:basic}
    Z_{X,L_K,\beta}' := \sum_{(u,S)\in\M} (-1)^{o(u)} a^{2 \deg(u)} z^{-\chi(S)} a_L^{u \OpenHopf L_K}\, \langle\partial u\rangle = H_K(a,z)\cdot \Gamma \in \Sk(L_K), 
    \end{equation}
    where $(-1)^{o(u)}$ is the orientation sign of $[u] \in \M$, $\deg(u)$ is the homological degree of the curve $u$, and $\chi(S)$ is the topological Euler characteristic of the domain.
    
    Furthermore, there exists generic $J$ that are arbitrarily SFT-stretched around $L_{K;\epsilon}$, see Section \ref{ssec:SFTlimitnoncompactL}, and for any sufficiently stretched such $J$, the boundary $\partial u\subset L_K$ of each $u\in\mathcal{M}$ is arbitrarily $C^1$-close to the central circles representing $\Gamma$ and $u \OpenHopf L_K=0$.  
\end{thm}
In \eqref{eq:basic}, since $u$ has boundary, the definition of $\deg(u)$ requires certain choices, and correspondingly the definition of the 
HOMFLYPT polynomial also requires a choice of framing on the link.  We explain in Section \ref{sec:proofthmbasic} how compatible choices
are naturally induced by any choice of vector field on $S^3$ transverse to the link.

Our proof (given in Section \ref{sec:proofthmbasic}) of Theorem \ref{thm:basic} will be an application of existence and invariance of the skein-valued curve counting for curves in basic homology classes and generic almost complex structures, plus a neck stretching argument. 
The argument {\em does  not} require consideration of multiply covered curves, so our proof of Theorem \ref{thm:basic} does not depend on the further constructions of \cite{bare}.

\begin{remark}  
Ooguri and Vafa arrived at this formula via a `large $N$ duality' relating the topological string theory on $T^*S^3$ with
$N$ branes on the zero section, and topological string theory in the resolved conifold.  Here the number $N$ of branes in $X_\epsilon$ 
is supposed to be related by a change of variables to a parameter measuring the area of the $\C\P^1$.   
In our setup, we do not explicitly have a `number of branes'.  Instead we have the 4-chain (of $S^3$), and the resulting Lagrangian linking number 
it determines.  
\end{remark}

\vspace{2mm} \noindent {\bf Acknowledgements.} 
 We thank 
Mohammad Abouzaid, Luis Diogo, Aleksander Doan, Kenji Fukaya, Penka Georgieva, Vito Iacovino, Melissa Liu, Pietro Longhi, John Pardon, and Jake Solomon for helpful
discussions.  Many of the ideas of this article were developed
in discussions during 
the spring 2018 MSRI-semester \emph{Enumerative geometry beyond
	numbers}. We thank MSRI for the wonderful environment.

\section{Linking numbers of open curves and Lagrangians}\label{sec:linking}
In this section we will define a linking number between a 3-dimensional Lagrangian submanifold $L$ in a 6-dimensional symplectic manifold $X$ and an open holomorphic curve $u\colon (S,\partial S)\to (X,L)$.  While the dimensions are appropriate, the objects are not transverse and one of them is not closed.  Nevertheless, an almost complex structure in a neighborhood of $L$ and certain additional choices, will allow us to define this number. Similar constructions were used before in the context of real algebraic geometry, see \cite{viro,ekholmslk,bjorklund}.

Let us first recall the definition of linking number in the usual sense. 
Suppose $M$ is an oriented manifold of dimension $n$ and $A, B \subset M$ are disjoint oriented cycles of dimensions
$a, b$ with $a+b = n-1$.  Assuming $B$ bounds,  then one defines the linking number $A \OpenHopf B$ by choosing 
a chain $C$ with $\partial C = B$, and setting $A \OpenHopf B := A \cdot C$.  
A different choice $C'$ of chain will have $\partial (C' - C) = 0$, and so changing the chain affects the linking number 
by the intersection of $A$ with the appropriate element of $H_{b+1}(M)$.  In particular, if $H_{b+1}(M) = 0$, then the linking number is well defined. 
If $M$ is non-compact, we will allow the chains $B$ and $C$ to have closed but not necessarily compact support, hence the above homology groups should be replaced by Borel-Moore homology, but we will still take $A$ to be compact. (In case $A$ is also non-compact, also linking at infinity between the  boundaries of $A$ and $C$ at infinity must be taken into account.)   

We return to our situation of interest. Let $L$ be a smooth orientable $n$-dimensional manifold, $n$ odd. Consider the cotangent bundle $T^\ast L$ of $L$ with its standard symplectic form $d(pdq)$. Fix an almost complex structure $J$ on a neighborhood of $L\subset T^\ast L$ compatible with $d(pdq)$. The symplectic form $d(pdq)$ and $J$ induces a Riemannian metric on $T^\ast L$ in the neighborhood and in particular give associated metrics on the vector bundles $TL$ and $T^\ast L$ along $L$. Fix a tangent vector field $\xi$ along $L$ with transverse zeros and let $\xi^\ast$ be the dual covector field. 

Consider the $\epsilon$-disk neighborhood $D_{\epsilon}^{\ast}L$ of $L$ in $T^{\ast}L$. 
Let $Z(\xi)$ denote the zeros of $\xi$. For $q\notin Z(\xi)$, 
let $\widehat {\xi^\ast}(q)=\frac{1}{|\xi^\ast(q)|}\xi^\ast(q)$ denote the normalization of $\xi^\ast$. For $q\in Z(\xi)$, let $(-1)^{i(q)}D_{\epsilon}(q)$ denote the $\epsilon$-disk in $T_{q}^\ast L$, oriented according to the index of $\xi$ at $q$.
 
Define $\partial_\epsilon \Delta_{\xi}\subset D_\epsilon^{\ast}L$ as
\[ 
\partial_\epsilon\Delta_{\xi}=
\overline{\left\{ (q,p)\colon p=\epsilon\,\widehat{\xi^\ast}(q) \right\}}
\ \cup \ \bigcup_{q\in Z(\xi)} (-1)^{i(q)} D_{\epsilon}(q),
\]
Let $b^\xi\colon L\times [0,1]\to D^\ast_\epsilon L$ be a smooth map such that 
\begin{align*}
&b^\xi|_{L\times 0} =\mathrm{id} \quad \text{ and }\quad \partial_t b^\xi(q,0)=J\xi(q), \text{ for }q\notin Z(\xi),\\ 
&b^\xi((L\setminus Z(\xi))\times(0,1])\cap L  = \varnothing \quad \text{ and }\quad
b^\xi(L\times 1)=\partial_\epsilon\Delta_\xi.    
\end{align*}
Then 
\[ 
\partial \, b^\xi(L\times[0,1])= L - \Gamma'_{\xi} - \sum_{q\in Z(\xi)} (-1)^{i(q)}  D_{\epsilon}(q),
\]  
where $\Gamma'_{\xi}$ denotes the graph of $\epsilon\,\widehat{\xi^\ast}$ over $L-Z(\xi)$. Consider the chain
\begin{equation}\label{eq : def Delta chain}
\Delta_{\pm\xi}:=b^\xi(L\times[0,1]) + b^{-\xi}(L\times[0,1])
\end{equation}
and note that
\[ 
\partial \Delta_{\pm\xi}  = 2\cdot L - \Gamma'_{\xi} - \Gamma'_{-\xi},
\]
since the fibers over critical points cancel.

Now let $X$ be any symplectic manifold of dimension $2n$ and $L \subset X$ a Lagrangian submanifold. Let $J$ be an almost complex structure on $X$ and $\xi$ be a vector field tangent to $L$ along $L$ with transverse zeros. Identify some neighborhood of $L$ in $X$ symplectically with $D_{\epsilon}^{\ast}L$. 
\begin{definition}\label{def : 4-chain}
An \emph{$L$-bounding chain compatible with $\xi$} is a singular chain $C$ with $\partial C = 2 L$, such that 
\[
C=\Delta_{\pm\xi} + C',
\]
where $C'$ is a chain transverse to $L$. 
\end{definition}

\begin{remark}
One can produce such $C$ by taking any chain $C''$ with $\partial C'' = \partial\Delta_{\pm\xi}$, and perturb slightly to obtain $C'$. Such a chain exists whenever $2L$ is null-homologous in $X$. 
\end{remark}

\begin{definition} \label{brane} 
A Lagrangian brane is a tuple $(L,\xi,J,C)$ where $L$ is a Lagrangian, $J$ an almost complex structure on $X$, $\xi$ is a vector field with transverse zeros, and $C$ is an $L$-bounding chain compatible with $\xi$.  We often denote the tuple just by $L$. 
\end{definition}

In the situation of the definition we may consider the 1-chain $\gamma := (C - \Delta_{\pm\xi}) \cap L \subset L$.  Then $\gamma$ is a cycle.  If $\gamma = 0 \in H_1(L, \Z)$ then we say that $C$ is \emph{admissible}.  

\begin{definition} \label{admissiblebrane}
An admissible Lagrangian brane is a Lagrangian brane together with a choice of chain $\sigma$ with $\partial \sigma = \gamma$. 
\end{definition}

We now specialize to $\dim L = 3$.\footnote{In other dimensions, the same data should allow to define linking numbers for appropriate dimensional parametric families of holomorphic maps.}

\begin{definition}\label{holomorphiconboundary}
If $S$ is a Riemann surface with boundary $\partial S$ and complex structure $j$ and if $u\colon (S,\partial S)\to (X,L)$ is a smooth map then we say that $u$ is \emph{holomorphic along the boundary} if
\[
    (du + J\circ du \circ j)|_{\partial S}=0.
\]
\end{definition}

\begin{definition} \label{transversetolag}
Let $L = (L, \xi, J, C)$ be a 3-dimensional Lagrangian brane.  
Let $u\colon (S, \partial S) \to (X, L)$ be a smooth map.  
We say $u$ is \emph{transverse} to $L$ if: 
\begin{enumerate}
	\item $u_*(T\partial S)$ is everywhere linearly independent from $\xi$.
	\item $u(\inr(S))$ is disjoint from $L$ and transverse to $C \setminus L$. 
	\item $u(\partial S)$ is disjoint from $\gamma$.
\end{enumerate}
\end{definition}

\begin{remark}
Note that the hypothesis in Definition \ref{transversetolag} that $u_*(T\partial S)$ is everywhere linearly independent from $\xi$
implies in particular that
$u_*(T\partial S)$ is nowhere zero, i.e., $u$ is an immersion of the boundary, and that $u(\partial S)$ is disjoint from $Z(\xi)$. Then since $u$ is holomorphic on the boundary it follows that there exists a neighborhood $N$ of $\partial S$ in $S$ such that $u(N\setminus \partial S)\cap C=\emptyset$.
\end{remark}

Let $L$ be an admissible 3-dimensional Lagrangian brane, and let $u\colon(S, \partial S) \to (X, L)$ be transverse to $L$ and holomorphic along the boundary. Inside $u|_{\partial S}^*TL$, let $\nu$ be the (positive) unit normal to the plane spanned by the
ordered basis of the tangent vector to $\partial u$ and $\xi$. Consider a collar neighborhood of the boundary, $\partial S\times[0,1]\subset S$. Extend $J \nu$ over $S\times[0,1]$ by multiplying by a cut off function with values in $[0,1]$ to a section of $u^* TX$ supported in a small neighborhood of the boundary inside the collar neighborhood.  Shift $u$ slightly along $J \nu$ and denote the resulting map $u_{J\nu}$. Then, since $u$ is holomorphic on the boundary and condition $(1)$ in Definition \ref{transversetolag} holds, $u_{J\nu}(\partial S)$ and $C$ are disjoint. 
\begin{definition} \label{def:lagrangianlinking}
Define 
\begin{equation}  \label{eq:linking}
u \OpenHopf L := u_{J\nu} \cdot C + \partial u \OpenHopf \gamma,
\end{equation} 
where $\partial u \OpenHopf \gamma := \partial u \cdot \sigma$ is the ordinary linking number in $L$.
\end{definition}

\begin{lemma}\label{l : intersections without shifting}
Let $u\colon (S, \partial S) \to (X, L)$ be an embedding transverse to $L$ and holomorphic along the boundary. If $I(u,C)$ denotes the algebraic number of interior intersection points between $u(S)$ and $C$ then $u_{J\nu} \cdot C=I(u,C)$ and consequently
\begin{equation}\label{eq:linking without shift}
u \OpenHopf L = I(u,C) + \partial u \OpenHopf \gamma.
\end{equation}
\end{lemma}

\begin{proof}
This is immediate, since $\nu$ is independent of $\xi$ there are no intersections contributing to $u_{J\nu}\cdot C$ near the boundary.
\end{proof}

The term $\partial u \OpenHopf \gamma$ is a sort of self-linking correction, which is present to ensure
the following invariance:

\begin{lemma} \label{l:throughgamma}
Let $u_t\colon (S, \partial S) \to (X, L) $ be a 1-parameter family of maps that are holomorphic along the boundary and that
satisfies the hypotheses of Definition \ref{transversetolag} save with $(3)$ replaced by 
`$u_t(\partial S)$ intersects $\gamma$ in general position'. Then $u_t \OpenHopf L$, which is defined away from the discrete set of $t$ where $u_t(\partial S)$ intersects $\gamma$, is constant.  
\end{lemma}
\begin{proof}
Evidently this quantity is locally constant where the conditions of  Definition \ref{transversetolag} hold.  It remains to check moments when $u_t(\partial S)$ crosses $\gamma$ in general position, i.e., at crossing moments the tangent vector of $u_t(\partial S)$ is linearly independent with the tangent vector of $\gamma$. 

Since the curve is holomorphic along the boundary we have the following local model up to first order. The ambient space is $\C^{3}$ with coordinates $(x_{1}+iy_{1},x_{2}+iy_{2},x_{3}+iy_{3})$. The Lagrangian $L$ is $\R^{3}\subset\C^{3}$ and the 4-chain $C$ is the subset given by $\{x_{1}=y_{2},x_{2}=y_{1}\}$. Then $\gamma$ is the subset of $L$ given by $x_{1}=x_{2}=0$ and we take the bounding 2-chain $\sigma$ to be $\{x_{1}\le 0, x_{2}=0\}$. 
The family of maps holomorphic on the boundary is $u_{t}(\xi+i\eta)=(t,\xi+i\eta,0)$, for $\xi+i\eta$ in the upper half plane. Then for $t< 0$ there is an intersection point contributing to $\partial u \cdot \sigma$ at $(t,0,0)$ and no intersection point in $u_{t} \cap C$, whereas for $t>0$ there is no intersection point in $\partial u\cap\sigma$ but the point $(t,it,0)$ contributes to $u_{t}\cdot C$. 

We conclude that the change in linking 
$u_t(\partial S) \OpenHopf \gamma$ is exactly compensated by the appearance or disappearance of 
an intersection point in $u_{J\nu} \cdot C$.  
\end{proof}

Another possible degeneration in a 1-parameter family is that at some moment, the tangent vector to $u(\partial S)$ becomes
a multiple of $\xi$ at one point of the boundary. Here, general position means that 
under the inverse of the exponential map, followed by orthogonal projection along $\xi$, the family of boundary curves $u_t(\partial S)$ in a neighborhood of the tangency gives a versal deformation of a semi-cubical cusp. Recall that we used the linear independence of $\xi$ and the boundary tangent to 
define the framing of $u(\partial S)$; consequently we term a degeneration of this type a {\em transverse framing degeneration}.
Indeed at such a moment the framing will change by $\pm 1$. \footnote{Recall that framing is a $\Z$-torsor, i.e.~it makes sense
to add integers to it, with the meaning that the ribbon of the knot undergoes more or less twists.}  

We used the same linear independence for another reason: to define the push off $u_{J\nu}$, which entered into our linking formula.  

\begin{lemma} \label{l:tangency} 
	In a 1-parameter family  $u_t\colon (S, \partial S) \to (X, L)$ of maps that are holomorphic along the boundary and satisfy the hypotheses of Definition \ref{transversetolag}, save at transverse framing degenerations, the framing plus $u_t \OpenHopf L$, which is defined away from a discrete set of $t$, is constant. 
\end{lemma}
\begin{proof}
We will show that the framing change is canceled by a change in the term $u_{J\nu} \cdot C$. Since the maps $u_t$ are holomorphic along the boundary we have the following local model up to first order. Consider local coordinates 
\[ 
(z_{1},z_{2},z_{3})=(x_{1}+iy_{1},x_{2}+iy_{2},x_{3}+iy_{3})\in\C^{3}
\]
on $X$ with $L$ corresponding to $\R^{3}$. Assume that the vector field $\xi$ is $\xi=\partial_{x_{3}}$. Then the 4-chain $C$ is locally given by $C=C^{+}+C^{-}$, where 
\[ 
C^{\pm}=\{y_{2}=y_{3}=0,\pm y_{3}\ge 0\}.
\]
The family of maps holomorphic on the boundary with general position tangency with $\partial_{x_{3}}$ is given by the maps $u_{\pm;t}\colon H\to\C^{3}$, where $H$ is the upper half plane and $t\in(-\epsilon,\epsilon)$,
\[ 
u_{\pm;t}(z)=(z^{2},z(z^{2}+t),\pm z).
\]
For $t<0$ the projection of $u|_{\partial H}$ to the $(x_{1},x_{2})$ has a double point at $z=\pm\sqrt{-t}$ that contributes to linking according to the sign of $u_{\pm;t}$. At $t=0$ the boundary has a tangency with $\partial_{x_{3}}$ and at $t>0$, $u_{\pm}(H)$ intersects $C^{\pm}$ at $u(i\sqrt{t})$ with the sign that agrees with the liking sign before the tangency.
\end{proof}

\begin{lemma}\label{l : indep of vector field}
For a map $u\colon (S, \partial S) \to (X,L)$ which is holomorphic along the boundary and transverse to $L$, the framing of $u(\partial S)$ plus 
$u \OpenHopf L$ is independent of the choice of $\xi$ and metric $g$. 
\end{lemma} 
\begin{proof}
Any two vector fields can be connected by a family of vector fields such that the zeros of the vector fields
never meet $u(\partial S)$.  The remaining possible failures of the conditions of Definition \ref{transversetolag} 
look exactly like those already considered in Lemma \ref{l:throughgamma} and Lemma \ref{l:tangency}, save that 
we are moving the vector field rather than moving the map. 
\end{proof}

\begin{corollary}
If $u\colon (S,\partial S)\to (X,L)$ is a map transverse to $L$ and holomorphic on the boundary then 
\begin{equation}\label{eq: rep ws skein}
a^{u\OpenHopf L}\cdot \langle \partial u \rangle\in \Sk(L),
\end{equation}
where $\langle \partial u \rangle$ is the element in the skein represented by the oriented link $u(\partial S)$ framed by $\xi$, is invariant under isotopies of $u$ through maps that are holomorphic on the boundary and transverse to $L$ and is independent of $\xi$.
\end{corollary}

\begin{proof}
    Follows from Lemmas \ref{l : intersections without shifting}, \ref{l:throughgamma}, \ref{l:tangency}
\end{proof}

\begin{remark}
    Note that neither factor in the left hand side of \eqref{eq: rep ws skein} alone is deformation invariant but their product is. 
\end{remark}

In case $L$ is disconnected, we can sometimes separate out the self-linking correction  
$\partial u \OpenHopf \gamma$ into individual contributions from the separate branes. 
Indeed if $L = \bigcup L_i$ and correspondingly $C = \bigcup C_i$, we may 
consider the 1-chains $\gamma_{ij} = (C_i - \Delta_{\pm\xi_i}) \cap L_j$.  As the Lagrangians $L_i$ 
are disjoint, each $\gamma_{ij}$ is a cycle.  Admissibility means that $\gamma = \sum_{i,j} \gamma_{ij}$ is a boundary; 
separating out into components, it means that $\sum_j \gamma_{ij}$ is a boundary.  However we could ask
that in fact each $\gamma_{ij}$ was already a boundary.  In this case it makes sense to write

\begin{equation} \label{eq:multilinking} 
u \OpenHopf L_i  := u_{J\nu} \cdot C_i +  \sum_j \partial u \OpenHopf \gamma_{ij} 
\end{equation}

The virtue of this convention is that $u \OpenHopf L = \sum u \OpenHopf L_i$.  (A defect of it is that 
if one viewed $(L_i, f|_{L_i}, g|_{L_i}, C_i)$ itself as a Lagrangian brane, 
then one would have the incompatible formula $u \OpenHopf L_i =  u_{J\nu} \cdot C_i + \partial u \OpenHopf \gamma_{ii}$.  
Nevertheless no confusion should arise.)  
The argument of Lemma \ref{l:throughgamma} applies also to these numbers. 

\begin{rmk}
We will consider both compact and non-compact Lagrangians and $4$-chains. 
In the non-compact cases we consider, outside a compact set our geometric objects are products with a half line.  We take our 4-chains and bounding chains for $\gamma$ as cycles with closed but possibly non-compact support. 

The question of whether $L$ and $\gamma$ bound then concerns Borel-Moore, rather than ordinary, homology.  
Recall that Borel-Moore homology can be calculated by the sequence 
$$
0 \to \mathrm{Ext}^1(H^{i+1}_c(X, \Z)) \to H_i^{BM}(X, \Z) \to \Hom(H^i_c(X, \Z), \Z) \to 0.
$$
In particular if $L$ is a handlebody, then $H_1^{BM}(L, \Z) = 0$, so a bounding chain for $\gamma$ can always
be found.  The difference between two choices of bounding chains lives in $H_2^{BM}(L, \Z) \cong \Z^{g}$, which is 
generated e.g.~by disks filling an appropriate disjoint collection of circles on $\partial L$. 
\end{rmk}

\begin{remark}
    In \cite{scharitzer-shende-mirror-cpgl}, it is explained how the theory developed in the present article can be generalized to Lagrangians not satisfying the admissibility condition by incorporating `linking lines' into the skein. 
\end{remark}

\section{Local models for crossings and nodal curves}
\label{sec:models}

In this section we present the relevant local models for
the wall-crossing instances in generic 1-parameter families.  

We use coordinates $(z_{1},z_{2},z_{3})=(x_{1}+iy_{1},x_{2}+iy_{2},x_{3}+iy_{3})$ on $\C^{3}$. The Lagrangian $L$ corresponds to $\R^{3}=\{y_{1}=y_{2}=y_{3}=0\}$, the vector field $\xi$ is $\xi=\partial_{x_{3}}$, and the 4-chain $C$ is 
\[ 
C= \{(z_{1},z_{2},z_{3})\colon y_{1}=y_{2}=0,y_{3}\ge 0\}\cup
\{(z_{1},z_{2}, z_{3})\colon y_{1}=y_{2}=0,y_{3}\le 0\},
\] 
where the two pieces are both oriented so that they induce the positive orientation on $\R^{3}$.

\subsection{Hyperbolic boundary}\label{ssec: hyp boundary}
We will consider a 1-parameter family of maps from a disjoint union of two half-disks. Let $D_+=\{z\in\C\colon |z|>1,\;\mathrm{Im}(z)\ge 0\}$. Consider the $1$-parameter family of maps $u_t^k\colon (D_+,\partial D_+)\to (\C^3,\R^3)$,
\[
u_t^1(z)=(z,z,t),\quad u_t^{2}(-z,z,0),\qquad -\delta\le t\le\delta.
\]

We write $u_t^{\,\mathrm{hyp}}=u_t^1(D_+)\cup u_t^2(D_+)$ for the image of these two maps. Note that if we give the boundary of the half-disks the boundary orientation then $\partial u_{\pm t}^{\,\mathrm{hyp}}$, $t>0$ gives two curves isotopic to those depicted on the left hand side of the skein relation in the top equation of Figure \ref{fig:skeinrel}.

We consider an additional family of maps. Let $r\in [0,\delta)$ and consider the half
\begin{equation}\label{eq : domain Da hyp}
D_r\subset \{x^2-y^2=r\}\cap\{|x|^2+|y|^2<1\}\subset \C^2,
\end{equation}
of the curve that induces the orientation of the boundary with positive $y$-component. Note that $D_r$ is a strip for $r>0$ and a strip with a hyperbolic boundary node for $r=0$ and that $\partial D_r\subset\R^2$.
Consider then the following additional family of maps from a strip, 
\begin{equation}\label{eq:realnode}
v_{r} \colon (D_r,\partial D_r) \to  (\C^{3},\R^3)=(\C^2\times\C,\R^2\times\R),\quad 
v_r(z)=(z,0).
\end{equation}
This is a model family for Gromov convergence to a map $v_0$ whose domain is a nodal curve, and whose normalization is given by ${u}_0^1\sqcup u_0^2$. Let $v_r^{\mathrm{hyp}}=v_r(D_r)$ denote the image. 
Then the boundary $\partial v_r^{\mathrm{hyp}}$ is a curve in the $x_1x_2$-plane in $\R^3$, isotopic to the curve giving the right hand side of the skein relation in the top equation of Figure \ref{fig:skeinrel}. 

\begin{definition} \label{standard hyperbolic degeneration}
    Let $L \subset X$ be a Lagrangian in a symplectic 6-manifold, and fix $p \in L$.  Suppose given families of subsets $A_\alpha \subset X$ for $\alpha \in (-\epsilon, \epsilon)$ and $B_\beta \subset X$ for $\beta \in [0, \epsilon')$.  We say that this pair of families is a standard boundary crossing/hyperbolic degeneration near $p$ if there are strictly increasing functions $r, s\colon \R \to \R$ with $r(0) = 0 = s(0)$  
    and $C^1$-paths in the Banach manifold of $C^1$ diffeomorphisms:
    $$\phi_\alpha,\psi_\beta\colon \mathrm{Nbhd}(p) \to  \mathrm{Nbhd}(0) \subset \C^3,\qquad \phi_0=\psi_0,$$
    such that the following holds:
    \begin{align*}
    \phi_\alpha(L) = \R^3,\; (d\phi_\alpha\circ J - J_0\circ d\phi_\alpha)|_L=0,&
    \quad \psi_\beta(L) = \R^3,\; (d\psi_\beta\circ J - J_0\circ d\psi_\beta)|_L=0,\\
    \phi_\alpha(A_\alpha) = u_{r(\alpha)}^{\,\mathrm{hyp}}, &\quad \psi_\beta(B_\beta) = v_{s(\beta)}^{\mathrm{hyp}}.
    \end{align*}
\end{definition}

\begin{remark}\label{r : hyp node as node}
   If the family $v_r\colon (D_r,\partial D_r)\to (\C^3,\R^3)$ is doubled (by Schwarz reflection) we obtain a family $w_r\colon C_r\to\C^3$ which is nodal at $r=0$ with a 1-dimensional fixed point locus.
\end{remark}

\subsection{Elliptic boundary}\label{ssec: ell boundary}
Consider the family of maps from a disk $D=\{z\in \C\colon |z|<1\}$
\[
u_{t}\colon D  \to  \C^{3},\quad u_t(z)=(z,iz,it),\qquad -\delta\le t\le \delta.  
\]
We write $u_t^{\,\mathrm{ell}}=u_{t}(D)$ for the image of the map.
The intersection $u_{t}^{\,\mathrm{ell}}\cap C$ is the point $(0,0,it)$. Note that  $\partial u_t^{\,\mathrm{ell}}$ gives a curve (without boundary) that intersects $C$ as in the left hand side of the skein relation in the second equation of Figure \ref{fig:skeinrel}.

We consider an additional family of maps. Let $r\in [0,\delta)$ and consider the half
\begin{equation}\label{eq : domain Da ell}
D_r\subset \{x^2+y^2=r\}\cap\{|x|^2+|y|^2<1\}\subset \C^2,
\end{equation}
of the curve that induces the positive orientation of the boundary circle in $\R^2$. Note that $D_r$ is a cylinder for $r>0$ and a cylinder with an elliptic boundary node (boundary circle of radius zero) for $r=0$.
Consider then the following additional family of maps from a cylinder, 
\begin{equation}\label{eq:imagnode}
v_{r} \colon (D_r,\partial D_r) \to  (\C^{3},\R^3)=(\C^2\times\C,\R^2\times\R),\quad 
v_r(z)=(z,0).
\end{equation}
This is a model family for Gromov convergence to a map $v_0$ whose domain is a nodal curve, and whose normalization is given by ${u}_0$. Let $v_r^{\mathrm{ell}}=v_r(D_r)$ denote the image. 
Then the boundary $\partial v_r^{\mathrm{ell}}$ is a curve in the $x_1x_2$-plane in $\R^3$, isotopic to the curve giving the right hand side of the skein relation in the second equation of Figure \ref{fig:skeinrel}.

\begin{definition}\label{standard elliptic degeneration}
    Let $L \subset X$ be a Lagrangian in a symplectic 6-manifold, and fix $p \in L$.  Suppose given families of subsets $A_\alpha \subset X$ for $\alpha \in (-\epsilon, \epsilon)$ and $B_\beta \subset X$ for $\beta \in [0, \epsilon')$.  We say that this pair of families is a standard Lagrangian crossing/elliptic degeneration near $p$ if there are strictly increasing functions $r, s\colon \R \to \R$ with $r(0) = 0 = s(0)$ and $C^1$-paths in the Banach manifold of $C^1$ diffeomorphisms:
    $$\phi_\alpha,\psi_\beta\colon \mathrm{Nbhd}(p) \to  \mathrm{Nbhd}(0) \subset \C^3,\qquad \phi_0=\psi_0,$$
    such that the following holds:
    \begin{align*}
    \phi_\alpha(L) = \R^3,\; (d\phi_\alpha\circ J - J_0\circ d\phi_\alpha)|_L=0,&
    \quad \psi_\beta(L) = \R^3,\; (d\psi_\beta\circ J - J_0\circ d\psi_\beta)|_L=0,\\
    \phi_\alpha(A_\alpha) = u_{r(\alpha)}^{\,\mathrm{ell}}, &\quad \psi_\beta(B_\beta) = v_{s(\beta)}^{\,\mathrm{ell}}.
    \end{align*}
\end{definition}

\begin{remark}\label{r : ell node as node}
   If the family $v_r\colon (D_r,\partial D_r)\to (\C^3,\R^3)$ is doubled (by Schwarz reflection) we obtain a family $w_r\colon C_r\to\C^3$ which is nodal at $r=0$ with a 0-dimensional fixed point locus.  
\end{remark}

\section{Boundary geometry of Floer gluing} \label{sec: gluing}

In this section we show that, under appropriate transversality hypotheses, moduli of $J$-holomorphic curves are standard near crossings/nodal moments. 

\subsection{Families of almost complex structures and transverse crossings}
We recall the expression for the linearization of the Cauchy-Riemann operator. 
Fix a symplectic manifold $X$, a Lagrangian $L$, and a metric $g$ on $X$ for which $L$ is totally geodesic. 

Consider a 1-parameter family $J_t$ of almost  complex structures on $X$. Note that the $t$-derivative of $J_t$ is an endomorphism $\dot J_t\colon TX\to TX$ such that $\dot J_t\circ J_t + J_t\circ\dot J_t=0$.  

Let $(S, \partial S)$ be a Riemann surface with area form $\omega_S$ and complex structure $j$. We write $\End^{0,1}(TS) \subset \End(TS)$ for the bundle  of  endomorphisms $\gamma\colon TS\to TS$ such that $j\circ\gamma+\gamma\circ j=0$. 
These give the first-order variations of the complex structure on $S$, as $(j + \epsilon \gamma)^2 = O(\epsilon^2)$. 

Consider the connection $\nabla^S$ of the Riemannian metric $\omega_S(\cdot,j\cdot)$ and fix a minimal number of such variations $\gamma$ that span the cokernel of the operator $\zeta \mapsto \nabla^S\zeta+j\circ\nabla^S \zeta \circ j$ acting on vector fields $\zeta$ along $S$ that are tangent to the boundary. Let $\mathrm{Def}(S)\subset \End^{0,1}(TS)$ denote the space spanned by the variations $\gamma$.

Let $u\colon (S,\partial S)\to (X,L)$ be a solution to the Cauchy-Riemann equations $\bar\partial_J u=0$. 
Let $\nabla$ denote the connection on the bundle $u^\ast TX$ pulled back from the metric connection on $X$. 

The linearization of the $J_0$-Cauchy-Riemann operator at $(u,S)$ is the map 
$$
 L_{(u,S)}\bar\partial_{J_0}: 
    H^3(S, u^* TX)  \times \mathrm{Def}(S) \times \R \to H^2(S, \Hom^{0,1}(TS, u^* TX))
$$
sending 
\begin{equation}
    \label{eq : linearization C-R operator}
    (v, \gamma, \epsilon) \ \mapsto \ (\nabla v + J_0\circ\nabla v\circ j)
- \tfrac12\left(\nabla_v J_0 +\epsilon \dot J_0\right)\circ\partial_{J_0} u\circ j - \tfrac12\partial_{J_0} u\circ j\gamma.
\end{equation}

\begin{definition} \label{Jt transverse crossing}
    Let $X$ be a symplectic 6-manifold with $c_1(X)=0$, $L \subset X$ a Maslov zero Lagrangian submanifold, and $J_t$ a 1-parameter family of almost complex structures.  Let $u\colon(S, \partial S) \to (X, L)$ be a $J_0$-holomorphic map with smooth domain curve. 
    Assume the (Fredholm index $1$) linearized operator $L_{(u,S)}\bar\partial_J$ (of the $t$-parameterized operator $\bar\partial_{J_t}$) has a $1$-dimensional kernel generated by a solution $(v,\gamma,\epsilon)$ with $\epsilon\ne 0$.      
    We say that $(u,S)$ is a $J_t$-transverse crossing if either: 
    \begin{itemize}
    \item (Hyperbolic crossing.) The map $u$ is an immersion and fails to be an embedding solely at the points $s_1, s_2 \in \partial S$, $u(s_1) = p = u(s_2)$.  Moreover, the three vectors $\partial_\tau u(s_1)$, $\partial_\tau u(s_2)$, and $v(s_1)-v(s_2)$ are linearly independent in $T_p L$. 

    In this case we write $(u^{\bullet},S^\bullet)$ for the map whose domain is the nodal curve obtained from $S$ by identifying $s_1$ and $s_2$.
    
    \item (Elliptic crossing.) The map $u$ is an embedding with interior disjoint from $L$ except at a single interior point $s\in S$, $u(s) = p$, such that $du(T_s S)\cap T_p L=0$.  Moreover, $du(T_s S)+T_p L+ \R v(s)= T_p X$.

    In this case we write $(u^{\bullet},S^\bullet)$ for the map whose domain is the nodal curve obtained from $u$ by viewing $s$ as an elliptic node. 
\end{itemize}
\end{definition}

\begin{theorem}\label{t : wall crossing gluing} 
Let $L \subset X$ be a Maslov zero Lagrangian in a symplectic $6$-manifold with $c_1(X)=0$, and let $J_t$ be a 1-parameter family of almost complex structures.  Let $u\colon (S, \partial S) \to (X, L)$ be a $J_0$ holomorphic map which is a $J_t$-transverse crossing in the sense of Definition \ref{Jt transverse crossing}, and let  $u^\bullet$ be the associated map from the corresponding nodal domain.  We write $\mathcal{M}(X, L, J) = \bigsqcup_t \mathcal{M}(X, L, J_t)$ for the set of all $J_t$ holomorphic maps; we endow this set with the topology of Gromov convergence. Then the following hold. 
\begin{itemize}
\item[$(a)$] There is a neighborhood $u \in \Omega(u) \subset \mathcal{M}(X, L, J)$ so that, for all $t$ with $|t|$ sufficiently small, $\Omega(u) \cap \mathcal{M}(X, L, J_t)$ consists of a single map $u_t$, which moreover is an embedding for $t\ne 0$. 
\item[$(b)$]    
There is a neighborhood
$v \in \Omega(v) \subset \mathcal{M}(X, L, J)$ so that, for all $t \ne 0$ with $|t|$ sufficiently small, 
there is at most one $J_t$-holomorphic map in $\Omega(v) \cap \mathcal{M}(X, L, J_t)$, which is moreover an embedding.  
\item[$(c)$] Let $A_t$ denote the family of images of the maps $u_t$ in $(a)$. Let $B_t$ be the image of map $v_t$ in $(b)$ if there exists such a map at $t$, or, otherwise, the empty set. 
Then $(X, L, p, A_t, B_t)$ is a standard boundary crossing/hyperbolic degeneration or a standard Lagrangian crossing/elliptic degeneration, in the sense of Definition \ref{standard hyperbolic degeneration} or Definition \ref{standard elliptic degeneration}, respectively. 
\end{itemize}
\end{theorem}

Theorem \ref{t : wall crossing gluing} will be proved in the following sections.  
The steps are as follows. In Section \ref{ssec : coord nbh of maps} we fix a standard family of maps which will include all holomorphic maps near the $J_t$-transverse solution $(u,S)$ and its associated nodal curve $(u^\bullet,S^\bullet)$, and on which we can run the usual arguments around Floer gluing. 
We also describe the kernel of the linearized Cauchy-Riemann operator at $(u,S)$ and $(u^\bullet,S^\bullet)$ on these spaces. In Section \ref{ssec : familes of sols} we study 1-parameter families of solutions. In Section \ref{sssec : sols hc and ec} we find unique 1-parameter family of solutions $(u_t,S_t)$ near $u=(u_0,S_0)$. In Sections \ref{sssec : pre glue hyp} and \ref{sssec : pre glue ell}, we define ``pre-glued'' families of maps $(u^\bullet_r,S^\bullet_r)$, such that $(u^\bullet_0,S^\bullet_0)=(u^\bullet,S^\bullet)$, and for which the pair of families $(u_t,S_t)$ and $(u_r^\bullet, S_r^\bullet)$ are manifestly standard in the sense of Definitions \ref{standard hyperbolic degeneration} or \ref{standard elliptic degeneration}.  We check that this remains true if  $(u^\bullet_r,S^\bullet_r)$ is replaced by any family sufficiently close to it in an appropriate norm. Finally, in Section \ref{sssec : sols hn en} we show that if we apply Floer gluing to the pre-glued family we indeed produce a family of solutions sufficiently close to $(u^\bullet_t,S^\bullet_t)$, which then implies Theorem \ref{t : wall crossing gluing}. 

In Section \ref{sssec : fine points}, we describe the second order germ of the 1-parameter families of solutions near nodal curves.

\subsection{Coordinates on neighborhoods of maps}\label{ssec : coord nbh of maps}

\subsubsection{Metric and coordinate choices}\label{ssec : exp map and J}

\begin{lemma} \label{lagrangian adapted metric} \cite[Lemma 5.6]{EESLCHRn}  
      Let $(X, \omega)$ be a symplectic manifold with compatible almost complex structure $J$.  Let $L\subset X$ be a Lagrangian.  
      Then there is a metric $g$ on $X$ such that $g= \omega(J\cdot, \cdot)$ on $L$, 
      
    such that $L$ is totally geodesic for $g$, and with the property that if $V$ is a Jacobi field along a geodesic in $L$, then so is $J\cdot V$. 
\end{lemma}

Note that we previously asked only that $L$ was totally geodesic; the condition on Jacobi fields is newly added.
We choose such metric in order that our boundary conditions will take a standard Dirichlet/Neumann form in local coordinates.  

For a metric $g$ on $X$, we write 
$\nabla$ for the Levi-Civita connection of $g$ and $\exp$ for its exponential map.

\begin{lemma}[Lemma 3.2 in \cite{EESPxR}]\label{l : Jacobi field metric} 
    In the situation of Lemma \ref{lagrangian adapted metric}, if $w\colon(S,\partial S)\to (X,L)$ is $J$-holomorphic on the boundary, if $v$ is a section of $w^\ast(TX)$ which is tangent to $L$ along $\partial S$ and such that $(\nabla v - J \circ \nabla v \circ j)|_{\partial S}=0$, and if $\exp$ denotes the exponential map in the metric of Lemma \ref{lagrangian adapted metric}, then also  $\exp_w(v)\colon (S,\partial S)\to (X,L)$, $\exp_w(v)(z)=\exp_{w(z)}(v(z))$ is $J$-holomorphic along the boundary.
\end{lemma}

We will also later want local coordinates with the following property: 
\begin{lemma} \label{local coordinates}
    Let $(X, \omega)$ be a symplectic manifold, $J$ a compatible almost complex structure, $L \subset X$ a Lagrangian,  and $p \in L$ a point.  
Consider also $\C^n$ with the standard symplectic form $\omega_{\mathrm{std}}$ and complex structure $I$.  Then there is a local symplectomorphism
    $\psi\colon (\C^n, 0) \to (X, p)$
    such that $\psi(\R^n) = L$ and 
    \begin{equation}\label{eq: J close to standard}
\psi^\ast(J)(0)=I,\quad |\psi^\ast(J)(x)- I|=\Ordo(|x|^2),\quad x\in\R^6.
\end{equation}
\end{lemma}
\begin{proof}
Identifying $(\C^n,\R^n)$ with $(T_p X,T_pL)$ we find $\alpha\colon (\C^n,\R^n)\to (X,L)$ such that $\alpha^\ast J(0)=I$. Write $(u,v)\in\R^n\times\R^n$ for $u+Iv\in\C^n$. Taylor expanding $J(u,v)=\alpha^\ast J$ we find
\[
J(u,v) = I + A_1 u+ A_2 v + \mathcal{O}(|u|^2+|v|^2),
\]
where $A_k I=I A_k$, $k=1,2$, and hence
\[
A_k=\left(\begin{matrix}a_k  & b_k\\ b_k & -a_k\end{matrix}\right).
\]
Define new coordinates $(u,v)=\beta(x,y)$ by
\begin{align*}
u_s &= x_s - \tfrac12\sum_t (a_1^{st}+b_2^{st})y_t^2\\
v_s &= y_s +\sum_t a_1^{st}x_ty_t + \tfrac12\sum_t (a_1^{st}x_t^2 + b_1^{st} y_t^2). 
\end{align*}
One verifies $d\beta^{-1}J(u,v)d\beta=I+\mathcal{O}(|x|^2+|y|^2)$ and $\beta(\R^n)=\R^n$. Taking 
$\psi=\alpha\circ\beta$ the lemma follows.
\end{proof}

Below we will consider $1$-parameter families $J_t$ of almost complex structures around $J_0$ and we will fix a family of diffeomorphisms as in Lemma \ref{local coordinates} for $J_0$, that will use throughout Sections \ref{sssec : hc}, \ref{sssec : hn}, \ref{sssec : ec}, and \ref{sssec : en}:

\begin{lemma}\label{l : diffeomorphisms for J}
Consider a cotangent bundle $T^\ast L$ and let $J_t$, $t\in[0,1]$, be a $1$-parameter family of almost complex structures compatible with the standard symplectic form on $T^\ast L$. Then there is a family of diffeomorphisms $\psi_t$ of $T^\ast L$ such that
\begin{equation}\label{eq : Jcomplex along L}
\psi_t|_L=\mathrm{id},\quad (J_t\circ d\psi_t- d\psi_t\circ J_0)|_L=0,
\end{equation}
and such that $\phi_t=\mathrm{id}$ outside a neighborhood of the zero section.
\end{lemma}

\begin{proof}
A Riemannian metric on $L$ gives an isomorphism between the tangent and cotangent bundles of $L$. The tangent bundle $TL$ carries a canonical complex structure $J_{\mathrm{can}}\colon T(TL)|_L\to T(TL)|_L$ along the zero section that takes a tangent vector $v$ to $L$ to itself viewed as a vector tangent to the fiber. If $J$ is an almost complex structure on $TL$, define the map $\phi_J\colon TL\to TL$, $\phi_J(x,v)=(x,Jv)$ in a neighborhood of the zero section and interpolate to the identity in a slightly larger neighborhood. Then along $L$, $J\circ d\phi_J=d\phi_J J_{\mathrm{can}}$. To see that lemma holds take $\psi_t=\phi_{J_t}\circ\phi_{J_0}^{-1}$.
\end{proof}

Consider now a Lagrangian $L\subset X$ and a family of almost complex structures $J_t$ on $X$. Let $N\subset X$ be a neighborhood of $L$ identified with a neighborhood of $L$ in $T^\ast L$. Then Lemma \ref{l : diffeomorphisms for J} gives a diffeomorphism 
\begin{equation}\label{eq : diffeo psi_t}
\psi_t\colon X\to X, \quad \psi_t=\mathrm{id}|_{X\setminus N},
\end{equation}
such that \eqref{eq : Jcomplex along L} holds. It follows in particular that if $w\colon (T,\partial T)\to (X,L)$ is $J_0$-holomorphic along the boundary then $\psi_t\circ w$ is $J_t$-holomorphic along the boundary.

\subsubsection{Neighborhoods of hyperbolic crossing curves}\label{sssec : hc}
Let $u\colon (S, \partial S) \to (X, L)$ be any $J$-holomorphic map  from a Riemann surface. 
We write $H^3(S,u^\ast(TX))$ for 
the Sobolev space of vector fields with 3 derivatives in $L^2$ that satisfy the boundary condition 
\begin{equation}\label{eq : boundary conditions sections}
v|_{\partial S}\in u^\ast(TL),\quad (\nabla v-J\circ\nabla\circ j)|_{\partial S}=0.
\end{equation}
We write $V_{\mathrm{con}}$ for the space of conformal variations of $S$, i.e., a neighborhood of $S$ in moduli (after stabilizing, if the original curve has unstable domain). 

Given 
$(v, \kappa, t) \in H^3(S,u^\ast(TX))\times V_{\mathrm{con}}\times (-\delta,\delta)$, 
we define a map $u_{v, \kappa, t}\colon (S,\partial S)  \to  (X,L)$,
\begin{equation}\label{hc before puncturing parameterization}
 z \mapsto  \psi_t(\exp_{u(z)}(v(z))),
\end{equation}
where $\exp$ is the exponential map of the Riemannian metric of Lemma \ref{l : Jacobi field metric} and $\psi_t\colon X\to X$ is the diffeomorphism in \eqref{eq : diffeo psi_t}.

It is clear that $v\mapsto \psi_t(\exp_u(v))$ gives a diffeomorphism between the maps in $H^3(S,u^\ast(TX))$ that satisfy the boundary conditions and a neighborhood of the map $u$ in the subset of $H^3(S,X)$ that satisfy the boundary condition. It follows in particular that any $J_t$-holomorphic map near $u$ is of the form $u_{v, \kappa, t}$ for a unique $(v, \kappa, t)$.   

Recall that we formulated the notion of `transverse crossing' (Definition \ref{Jt transverse crossing}) assuming the linearization $L_{(u,S)}(\bar\partial_J)$ (given by the formula in \eqref{eq : linearization C-R operator}) has one dimensional kernel and asking that if $(v,\kappa,k\partial_t)$ is an element in the kernel, then 
\[
du (T_{s_1}\partial S) + du (T_{s_1}\partial S)+ \R\cdot (v(s_1)-v(s_2))= T_pL.
\] 

We will now translate this hypothesis into a parameterization of maps suitable to study gluing. As a first step we add boundary marked points $\zeta_1$ and $\zeta_2$ to the map near the intersection point. Let $S^\star$ denote the corresponding domain. The resulting maps are then parameterized by 
\[
H^3(S,u^\ast(TX))\times V_{\mathrm{con}}\times V_{\mathrm{con}}^\star \times (-\delta,\delta),
\]
where  $V_{\mathrm{con}}^\star=a_1\times a_2$, where $a_j$ are arcs in the boundary centered at $s_j$ and $\zeta_j\in a_j$, $j=1,2$, gives the locations of the marked points. The local coordinate map is the same as before, \eqref{hc before puncturing parameterization} and there are natural evaluation maps $\ev_k(u_{\kappa,v,t})=u_{\kappa,v,t}(\zeta_k)\in L$, $k=1,2$. The solution space now gains two dimensions since we may move the marked points along the boundary and the kernel of the linarization $L_{(u,S^\star)}\bar\partial_J$ at the solution with $u(\zeta_1)=u(\zeta_2)$ is the three-dimensional space spanned by:
\begin{equation}\label{eq : kernel marked hc}
(v,\kappa,\partial_t),\quad  m_1\in T_{s_1}a_1\subset TV^\star_{\mathrm{con}},\quad m_2\in T_{s_2}a_2\subset TV^\star_{\mathrm{con}}, 
\end{equation}
where the first tangent vector is the previous kernel element.

We next replace the marked points by punctures in $\partial S$ at $\zeta_1$ and $\zeta_2$ in a neighborhood of $\zeta_j=s_j$, with strip neighborhoods $[0,\infty)\times[0,1]\subset S$ where the puncture is at $\infty$. Let $S^\circ$ denote the corresponding punctured domain. 
Assume that the punctures are sufficiently close to $s_j$ and that the strip neighborhoods are sufficiently small so that their images lie in the local coordinate system of Lemma \ref{local coordinates}. In these coordinates, the boundary condition along $[0,\infty)\times\{k\}$, $k=0,1$ is $\R^3\subset\C^3$ (for both boundary components). 

We write 
\[
\mathbf{V}_{\mathrm{hc}} \ := \ H^{3,\delta}(S^\circ, u^\ast(T^\ast X))
\] 
for the weighted Sobolev space of vector fields with three derivatives in $L^2$, which satisfy the boundary condition \eqref{eq : boundary conditions sections}, with norm given by
\begin{equation}
\|v\|^2 \ = \ \sum_{k=0}^{3}\int_{S^\circ} |d^{(k)}v|^2 w_{\delta} \,dA
\end{equation}
Here  $w_\delta\colon S^\circ\to [1,\infty)$ is some fixed function which takes the value $1$ in the complement of the strip neighborhoods of the punctures and is $e^{\delta|\sigma|}$ in coordinates $(\sigma,\tau)\in[0,\infty)\times[0,1]$ near the punctures, where $0<\delta<\pi$. 

Since the weight function is exponentially growing, vector fields $v\in \mathbf{V}_{\mathrm{hc}}$ decay at the puncture, and hence $\exp_u(v)$ takes punctures to the origin of $\R^3$. We vary the image of punctures by using cut-off constant solutions (sometimes called asymptotic constants). These are defined as follows. 

Let $\hat c\in\R^3$, and let $\beta\colon [0,\infty)\times[0,1]\to [0,1]$ be a cut off function where $\beta=1$ in $[3,\infty)\times [0,1]$, $\beta=0$ outside $[2,\infty)\times [0,1]$. Consider $[0,\infty)\times[0,1]\subset S^\circ$ as the strip neighborhood of $\zeta_k$, $k=1,2$, and
extend $\beta$ to all of $S^\circ$ by $0$. Assume that $\hat c$ is sufficiently small that it is contained in the coordinate chart of Lemma \ref{local coordinates} and define $c_k\colon (S^\circ,\partial S^\circ)\to (\C^3,\R^3)\subset (X,L)$, $k=1,2$:
\begin{equation}\label{eq : local solution hyp}
c_k(z)= 
\begin{cases}
0 & z\in S^\circ\setminus[0,\infty)\times[0,1],\\
\beta(z)\hat c & z\in [0,\infty)\times[0,1].
\end{cases}
\end{equation}

We define  $V_{\mathrm{sol}}$ as the space of maps $(c_1,c_2)$ as in \eqref{eq : local solution hyp} constructed from vectors $(\hat c_1,\hat c_2)$ in $U\times U$, where $U$ is a small neighborhood of the origin in $\R^3$. Then $V_{\mathrm{sol}}=U\times U$, where the isomorphism takes $(c_1,c_2)$ to its pair of asymptotic constants, and $TV_{\mathrm{sol}}= \R^3\times \R^3$.  


Also the conformal structure on $S^\circ$ varies. We think of these as changes of the complex structure $j$ on $S^\circ$. These are of two types, variations on $S$ itself and variations corresponding to moving the punctures. To keep the boundary conditions we take variations that are conformal along the boundary. We pick local coordinates
\[
V_{\mathrm{con; hc}} = V_{\mathrm{con}} \times V^\circ_{\mathrm{con}}. 
\]
Here $V^\circ_{\mathrm{con}}$ corresponds to moving the location of the punctures. 

From this data we may describe a family of maps near $(u,S)$: 
\begin{equation}\label{eq : loc coord bc}
\mathbf{V}_{\mathrm{hc}} \times V_{\mathrm{sol}} \times V_{\mathrm{con; hc}}\times (-\delta,\delta),\quad (v,c,\kappa,t)\mapsto (\psi_t(\exp_{u}(v)+c), S^\circ_\kappa).
\end{equation}
Here if $c=(c_1,c_2)$ and if $([0,\infty)\times[0,1])_k$ denotes the strip neighborhood of $\zeta_k$ then
\begin{equation}\label{eq : local solution}
\exp_{u(z)}(v(z))+c(z)= 
\begin{cases}
\exp_{u(z)}(v(z)) & z\in S^\circ\setminus\bigcup_{k=1,2}([0,\infty)\times[0,1])_k,\\
\exp_{u(z)}(v(z))+ c_k(z) & z\in ([0,\infty)\times[0,1])_k, k=1,2,
\end{cases}
\end{equation}
where `$+$' in the second row refers to addition in the local coordinates $(\C^3,\R^3)$.

We denote the linearization of the Cauchy-Riemann operator in these coordinates by:
\[
L_{(u,S^\circ)}\bar\partial_J\colon \ \mathbf{V}_{\mathrm{hc}} \times V_{\mathrm{sol}} \times T_{[S^\circ]} V_{\mathrm{con; hc}}\times \R  \ \to \ 
H^{2,\delta}(S^\circ; \mathrm{Hom}^{0,1}(TS^\circ, u^\ast(TX))).
\]

\begin{remark}
   As usual, the reason to use the weighted Sobolev space $H^{3, \delta}$ is that $\bar \partial_J$ would not be Fredholm on the corresponding $H^3$.  Imposing the weight forces the elements of $H^{3, \delta}$ to vanish at infinity, which we must correct for by re-introducing the translations along the Lagrangian, i.e., $V_{\mathrm{sol}}$.  The space $V_{\mathrm{con}; \mathrm{hc}}^\circ$ plays a similar role: the vector fields on $S$ which infinitesimally move the punctures can be choosen as cut-off vector fields of the form $\bar\partial (\beta\zeta)$, where $\beta$ are cut-off functions as above and where $\zeta$ is a holomorphic vector field in the strip region with exponential growth, $\zeta(\sigma,\tau)= b e^{\pi(\sigma+i\tau)}$, where $b\in\R$.  These do not satisfy the $\delta$ decay condition, and so it is natural to discuss them separately.  
\end{remark}

There is a continuous linear map $H^3(S^\star,u^\ast(TX))\to H^{3,\delta}(S^\circ,u^\ast(TX))$ which pre-composes a vector field $v$ with the change of coordinates $[0,1]\times[0,\infty]\to H$, $(\sigma,\tau)\to e^{-\pi(\sigma+i\tau)}$, see Lemma \ref{l:H3toH3delta}.
The map in particular takes the 3-dimensonal kernel of $L_{(u,S^\circ)}\bar\partial_J$ into the 3-dimensional kernel of $L_{(u,S^\star)}\bar\partial_J$ which allows us to determine the latter from \eqref{eq : kernel marked hc} as follows. First, the image of $(v,\kappa,\partial_t)$ is
\[
(v,\kappa,\partial_t)+b_1+b_2,
\]
where $b_k\in TV_{\mathrm{sol}}^\circ\approx \R^3$ is given by $\partial_t\ev_k\in \R^3$, $k=1,2$. Second the image of $m_k$ is $\gamma_k+c_k$, where $\gamma_k\in TV^\circ_{\mathrm{con}}$ corresponds to the linearized movement of the puncture determined by $m_k$ and where $c_k\in TV_{\mathrm{sol}}^\circ$ is given by $\partial_{m_k}\ev_k\in\R^3$. 
By the transverse crossing assumption, $c_1$, $c_2$, and $b_1+b_2$ in $\R^3\approx T_pL$ are then linearly independent.

Thus, the subspace of $\mathbf{V}_{\mathrm{hc}}\oplus V_{\mathrm{con; hc}}$ spanned by 
\begin{equation}\label{eq : approx kernel bc}
\gamma_1+c_1 \ , \ \gamma_2 + c_2 \ , \ \partial_t + b_1 + b_2,     
\end{equation}
is an approximate kernel in the sense that $L_{(u,S^\circ)}\bar\partial_J$ is invertible on its $L^2$-complement. 

\begin{remark}
For $k\ge 4$, the above assertion regarding kernel elements follows by commutativity of the diagram
\[
\begin{CD}
H^k(S^\star,u^\ast(TX)) @>>>  H^{k,\delta}(S^\circ,u^\ast(TX))\\
@V{L\bar\partial_J}VV          @VV{L\bar\partial_J}V \\
H^{k-1}(S^\star,\mathrm{End}(TS^\star,u^\ast(TX))) @>>>  H^{k-1}(S^\circ,\mathrm{End}(TS^\circ,u^\ast(TX)))
\end{CD}.
\]
The requirement $k \ge 4$ is because the map in the bottom row requires regularity $k-1\ge 3$, see Lemma \ref{l:H3toH3delta}. To treat $k=3$ as we do in our argument above, we pass to elements in the kernel of $L\bar\partial_J$, which are smooth by elliptic regularity, and then transport to the punctured domain.  
\end{remark}


\subsubsection{Neighborhoods of hyperbolic node curves}\label{sssec : hn}
We describe local coordinates for maps near the nodal map $(u^\bullet,S^\bullet)$ associated to the map $(u,S)$ in Section \ref{sssec : hc}. 

We consider first the variations of the conformal structure of $S^\bullet$. We have the same data as for $S^\circ$: the variation of the conformal structure of $S$ together with the location of the boundary punctures. In addition, we have a new conformal parameter, the gluing parameter for the neck near the node, we write $V_{\mathrm{neck}}=[0,\epsilon)$ and take the corresponding neck to have length $\frac{2}{r}$, $r\in V_{\mathrm{neck}}$. We define
\[
V_{\mathrm{con; hn}} = V_{\mathrm{con}} \times V^\circ_{\mathrm{con}}\times V_{\mathrm{neck}}.
\]
We write $\rho=\frac{1}{r}$ and let $S^\bullet(\rho)$, $\rho\in[\rho_0,\infty)$, denote the domain obtained from $S^\bullet$ by removing $(\rho,\infty)\times[0,1]$ from the neighborhoods of the punctures and gluing the two remaining strip regions along $\rho\times[0,1]$. We call the region resulting from gluing the cut off puncture neighborhoods, $[-\rho,\rho]\times[0,1]\subset S^\bullet(\rho)$ the \emph{gluing neck} and identify $[-\rho,0]\times[0,1]$ and $[0,\rho]\times[0,1]$ as subsets of the strip neighborhood of $s_1$ and $s_2$, respectively, in the natural way.

Let $\beta_\rho\colon [-\rho,\rho]\times[0,1]\to(\C,\R)$ be a cut-off function that is equal to $1$ outside $ [-2,2]\times[0,1]$ and equal $0$ on $[-1,1]\times [0,1]$. Define $u^\bullet_\rho\colon S^\bullet(\rho)\to (X,L)$ by
\[
u^\bullet_{\rho,t}(z) = 
\begin{cases}
\psi_t(u^\bullet(z)), & z\in S^\ast(\rho)\setminus [-\rho,\rho]\times[0,1],\\
\psi_t(\beta(\sigma,\tau)\cdot u^\bullet(\sigma,\tau)), & z=(\sigma,\tau)\in [-\rho,\rho]\times[0,1],
\end{cases}
\]
where the multiplication $\beta\cdot u^\bullet$ refers the local coordinates near the node. Then $u^\bullet_{\rho,t}$ is $J_t$-holomorphic along the boundary. We will center our coordinates at $u^\bullet_{\rho,t}$.

We introduce a weight function $w_{\rho,\delta}\colon S^\bullet(\rho)\to[1,\infty)$,
\[
w_{\rho,\delta}(z)=
\begin{cases}
1, & z\in S^\bullet(\rho)\setminus [-\rho,\rho]\times[0,1],\\
e^{\delta|\sigma|}, & z=(\sigma,\tau)\in [-\rho,\rho]\times[0,1].
\end{cases}
\]

Let $H^{3,\delta}(S^\bullet(\rho),u^\ast TX)$ be the Sobolev space
with norm
\[
\|v\|^2 \ = \ \sum_{k=0}^{3}\int_{S^\bullet(\rho)} |d^{(k)}v|^2 w_{\rho,\delta} \,dA.
\]
Let $H_0^{3,\delta}(S^\bullet(\rho),u^\ast TX)\subset H^{3,\delta}(S^\bullet(\rho),u^\ast TX)$ be the codimension $3$ subspace of sections $v$ that satisfy the condition that the real part (in $\R^3\subset\C^3$) of the average over the middle line in the gluing neck vanishes:
\[
\mathrm{Re}\int_0^1 v(0,t) dt = 0. 
\]

Let $\alpha_\rho\colon [-\rho,\rho]\times[0,1]\to [0,1]$ be a cut off function where $\alpha_\rho=1$ in $[2-\rho,\rho-2]\times [0,1]$, $\alpha_\rho=0$ outside $[1-\rho,\rho-1]\times [0,1]$.
Extend $\alpha_\rho$ to all of $S^\bullet(\rho)$ by $0$. Assume that $\hat c\in\R^3$ is sufficiently small that it is contained in the coordinate chart of Lemma \ref{local coordinates} and define $c\colon (S^\bullet(\rho),\partial S^\bullet(\rho))\to (\C^3,\R^3)\subset (X,L)$:
\begin{equation}\label{eq : local solution hn}
c(z)= 
\begin{cases}
0 & z\in S^\bullet(\rho)\setminus[-\rho,\rho]\times[0,1],\\
\alpha(z)\hat c & z\in [-\rho,\rho]\times[0,1].
\end{cases}
\end{equation}

Let $V_{\mathrm{sol}}$ denote the 3-dimensional space of cut-off constants $c\colon S^\bullet(\rho)\to(\C^3\R^3)\subset(X,L)$ and define
\[
\mathbf{V}_{\mathrm{hn},\rho} \ = \ H^{3,\delta}_0(S^\bullet(\rho);u_{\rho}^\ast TX)\oplus V_{\mathrm{sol}}. 
\]
Let $\mathbf{V_{\mathrm{hn}}}=\bigsqcup_{r\in [0,\epsilon)} \mathbf{V}_{\mathrm{hn},\,1/r}$, where for $r=0$ we let $\mathbf{V}_{\mathrm{hn},\,\infty}=\mathbf{V}_{\mathrm{hc}}$.

The neighborhood of $(u^\bullet, S^\bullet)$ is then
\begin{equation}\label{eq : loc coord hn}
\mathbf{V}_{\mathrm{hn}}\times V_{\mathrm{con;\, hn}}\times (-\delta,\delta),\quad ((v,c),(\kappa,r),t)\mapsto (\psi_t(\exp_{u^\bullet_{1/r,0}}(v)+c), S^\bullet(1/r)),
\end{equation}
where the addition  `$+$' is understood in the local chart as before.

Let $\hat u^\bullet_{\rho,t}$ be any map that agrees with $u^\bullet_{\rho,t}$ on $S^\bullet(\rho)\setminus [-\frac12\rho,\frac12 \rho]\times[0,1]$, that takes $[-\frac12\rho,\frac12 \rho]\times[0,1]$ into the $(\C^3,\R^3)$ coordinates and such that 
\[
\left|\hat u^\bullet_{\rho,t}|_{[-\tfrac12\rho,\tfrac12 \rho]\times[0,1]}\right|_{C^2}=\mathcal{O}(|e^{-2\delta}\rho|). 
\]
Consider the linearization $L_{(\hat u^\bullet_{\rho,t},S^\bullet(\rho))}\bar\partial_{J}$, for fixed $\rho=1/r$ :
\[
L_{(\hat u^\bullet_{\rho,t},S^\bullet(\rho))}\bar\partial_{J}\colon \mathbf{V}_{\mathrm{hn},\rho}\times TV_{\mathrm{con}}\times TV^\circ_{\mathrm{con}}\times\R \ \to \ 
H^{2,\delta}(S^\bullet(\rho),
\mathrm{Hom}^{0,1}
(T S^\bullet(\rho),
(\hat u^\bullet_{\rho,t})^\ast(TX)))
\]
(fixed $\rho$ means that there is no component along $TV_{\mathrm{neck}}$).
\begin{lemma}\label{l : uniform invertiblity of differential}
There exists $\rho_0>0$, $t_0>0$, and $C>0$ such that for all $\rho>\rho_0$ and $t<t_0$
\[
\|((v,c),\kappa,k\partial_t)\|\le C\|L_{(\hat u^\bullet_{\rho,t},S^\bullet(\rho))}\bar\partial_{J}((v,c),\kappa,k\partial_t)\|,
\]
in particular the Fredholm index zero map $L_{(\hat u^\bullet_{\rho,t},S^\bullet(\rho))}\bar\partial_{J}$ is a uniformly invertible isomorphism.
\end{lemma}

\begin{proof}
The argument follows standard lines for Floer gluing.
Consider first $t=0$.
Assume the assertion is false then there exists a sequence $w_\rho=((v_\rho,c_\rho),\kappa_\rho)$ such that 
$\|w_\rho\|=1$ and $\|L_{(\hat u^\bullet_{\rho,t},S^\bullet(\rho))}\bar\partial_{J} w_\rho\|\to 0$, as $\rho\to\infty$. Let $\beta\colon S^\bullet(\rho)\to[0,1]$ be a cut off function equal to $1$ outside $[-2,2]\times[0,1]$ and $0$ inside $[-1,1]\times[0,1]$ and consider the sequences $(\beta v_\rho,c,\kappa_\rho)$.  

Then $((\beta v_\rho, c_\rho),\kappa_\rho)\in \mathbf{V}_{\mathrm{bc}}$ and $L_{(u^\circ,S^\circ)}\bar\partial_J((\beta v_\rho, c_\rho),\kappa_\rho)\to 0$ and it follows from \eqref{eq : approx kernel bc} and that a subsequnece of $(\beta v_\rho, c_\rho),\kappa_\rho)$ converges to an element of the kernel of $L_{(u^\circ,S^\circ)}\bar\partial_J$. We claim that this element must equal to zero. To see this note that $c_\rho$ gives the same asymptotic constants at $s_1$ and $s_2$ in $V_{\rm sol}$. However, as the asymptotic constants of the elements in the kernel are vectors in the direction of the tangent vectors of $u$ at $s_1$ and $s_2$ which are linearly independent we conclude that the limit must have trivial component along the first two kernel elements in \eqref{eq : approx kernel bc}. A similar argument shows that also the component along the last kernel vector vanishes: the difference of asymptotic constants give the $\partial_t$ component, but the difference induced by $c_\rho$ equals $0$. 

We finally check that $(1-\beta)v_\rho\to 0$ also in the middle. This follows from the elliptic estimate for the usual $\bar\partial$-operator on the strip $\R\times[0,1]$ with $\R^3$ boundary condition and negative exponential weight $e^{-\delta|\sigma|}$. We conclude that the result holds for $t=0$. 

Noting that $|L_{(\hat u^\bullet_{\rho,t},S^\bullet(\rho))}\bar\partial_{J}-L_{(\hat u^\bullet_{\rho,0},S^\bullet(\rho))}\bar\partial_{J}|=\mathcal{O}(|t|)$, where the norm is the operator norm the lemma follows.  
\end{proof}

\subsubsection{Neighborhoods of elliptic crossing curves}\label{sssec : ec}
We construct a neighborhood of maps and domains for elliptic crossings in direct analogy with Section \ref{sssec : hc}.
Consider a $J_t$-transverse elliptic crossing curve $(u,S)$ as in Definition \ref{Jt transverse crossing}. 

We again parameterize a neighborhood of the map $(u,S)$ by $H^3(S,u^\ast(TX))\times V_{\mathrm{con}}\times (-\delta,\delta)$, given
$(v, \kappa, t) \in H^3(S,u^\ast(TX))\times V_{\mathrm{con}}\times (-\delta,\delta)$, 
we define $u_{v, \kappa, t}\colon (S,\partial S)  \to  (X,L)$,
\begin{equation}\label{ec before puncturing parameterization}
 z \mapsto  \psi_t(\exp_{u(z)}(v(z))).
\end{equation}
As before, any $J_t$-holomorphic map near $u$ is of the form $u_{v, \kappa, t}$ for a unique $(v, \kappa, t)$.  Here, `transverse crossing' (Definition \ref{Jt transverse crossing}) means that the linearization $L_{(u,S)}\bar\partial_J$ has one dimensional kernel and if $(v,\kappa,k\partial_t)$ is an element in the kernel, then 
\[
du (T_{s} S) + \R\cdot v(s) + T_{u(s)}L = T_{u(s)} X.
\] 

Add an interior marked point $\zeta$ to the map near the intersection point. Let $S^\star$ denote the corresponding domain. The resulting maps are then parameterized by 
\[
H^3(S,u^\ast(TX))\times V_{\mathrm{con}}\times V_{\mathrm{con}}^\star \times (-\delta,\delta),
\]
where $V_{\mathrm{con}}^\star\approx U$, where $U$ is a small disk around $s$. The local coordinate map is the same as before, \eqref{ec before puncturing parameterization}, and there is a natural evaluation map $\ev(u_{\kappa,v,t})=u_{\kappa,v,t}(\zeta)\in X$. The solution space gains two dimensions since the marked point can move and the kernel of the linarization $L_{(u,S^\star)}\bar\partial_J$ at the solution with $u(\zeta)\in L$ is the three-dimensional space spanned by:
\begin{equation}\label{eq : kernel marked ec}
(v,\kappa,\partial_t),\quad  m_1,m_2\in T_{s}U\subset TV^\star_{\mathrm{con}}, 
\end{equation}
where the first tangent vector is the previous kernel element.

We replace the marked point by an interior puncture in $S$ at $\zeta$, with cylinder neighborhood $[0,\infty)\times S^1\subset S$ where the puncture is at $\infty$. Let $S^\circ$ denote the corresponding punctured domain. 
Assume that the puncture is sufficiently close to $s_j$ and that the strip neighborhoods are sufficiently small so that their images lie in the local coordinate system of Lemma \ref{local coordinates}. 

We write 
\[
\mathbf{V}_{\mathrm{ec}} \ := \ H^{3,\delta}(S^\circ, u^\ast(T^\ast X))
\] 
for the weighted Sobolev space of vector fields with three derivatives in $L^2$, which satisfy the boundary condition \eqref{eq : boundary conditions sections}, with norm given by
\begin{equation}
\|v\|^2 \ = \ \sum_{k=0}^{3}\int_{S^\circ} |d^{(k)}v|^2 w_{\delta} \,dA
\end{equation}
Here  $w_\delta\colon S^\circ\to [1,\infty)$ is some fixed function which takes the value $1$ in the complement of the cylinde neighborhood of the puncture and is $e^{\delta|\sigma|}$ in coordinates $(\sigma,\tau)\in[0,\infty)\times S^1$ near the puncture, where $0<\delta<\pi$. 

Since the weight function is exponentially growing, vector fields $v\in \mathbf{V}_{\mathrm{ec}}$ decay at the puncture, and hence $\exp_u(v)$ takes punctures to the origin of $\R^3$. Again, we vary the image of punctures by using cut-off constant solutions (sometimes called asymptotic constants) defined as follows. 

Let $\hat c\in\C^3$, and let $\beta\colon [0,\infty)\times S^1\to [0,1]$ be a cut off function where $\beta=1$ in $[3,\infty)\times S^1$, $\beta=0$ outside $[2,\infty)\times S^1$. Consider $[0,\infty)\times S^1\subset S^\circ$ as the strip neighborhood of $\zeta$, and
extend $\beta$ to all of $S^\circ$ by $0$. Assume that $\hat c$ is sufficiently small that it is contained in the coordinate chart of Lemma \ref{local coordinates} and define $c\colon (S^\circ,\partial S^\circ)\to (\C^3,\R^3)\subset (X,L)$:
\begin{equation}\label{eq : local solution ec}
c(z)= 
\begin{cases}
0 & z\in S^\circ\setminus[0,\infty)\times S^1,\\
\beta(z)\hat c & z\in [0,\infty)\times[0,1].
\end{cases}
\end{equation}

We define $V_{\mathrm{sol}}\approx U\times U$, where $U$ is a neighborhood of the origin in $\R^3$ as the space of maps as in \eqref{eq : local solution ec}. Then $TV_{\mathrm{sol}}\approx \C^3$.  


Also the conformal structure on $S^\circ$ varies. We think of these as changes of the complex structure $j$ on $S^\circ$. These are of two types, variations on $S$ itself and variations corresponding to moving the punctures. To keep the boundary conditions we take variations that are conformal along the boundary. We pick local coordinates
\[
V_{\mathrm{con;\, ec}} = V_{\mathrm{con}} \times V^\circ_{\mathrm{con}}. 
\]
Here $V^\circ_{\mathrm{con}}$ corresponds to moving the location of the puncture. 

From this data we may describe a family of maps near $(u,S)$: 
\begin{equation}\label{eq : loc coord ec}
\mathbf{V}_{\mathrm{ec}} \times V_{\mathrm{sol}} \times V_{\mathrm{con;\, ec}}\times (-\delta,\delta),\quad (v,c,\kappa,t)\mapsto (\psi_t(\exp_{u}(v)+c), S^\circ).
\end{equation}
Here 
\begin{equation}\label{eq : exponetial map ec}
\exp_{u(z)}(v(z))+c(z)= 
\begin{cases}
\exp_{u(z)}(v(z)) & z\in S^\circ\setminus[0,\infty)\times S^1,\\
\exp_{u(z)}(v(z))+ c_k(z) & z\in [0,\infty)\times S^1,
\end{cases}
\end{equation}
where `$+$' in the second row refers to addition in the local coordinates $\C^3$.

We denote the linearization of the Cauchy-Riemann operator in these coordinates by:
\[
L_{(u,S^\circ)}\bar\partial_J\colon \ \mathbf{V}_{\mathrm{ec}} \times TV_{\mathrm{sol}} \times T_{[S^\circ]} V_{\mathrm{con;\, ec}}\times \R  \ \to \ 
H^{2,\delta}(S^\circ; \mathrm{Hom}^{0,1}(TS^\circ, u^\ast(TX))).
\]

There is a continuous linear map $H^3(S^\star,u^\ast(TX))\to H^{3,\delta}(S^\circ,u^\ast(TX))$ which pre-composes a vector field $v$ with the change of coordinates $S^1\times[0,\infty]\to \C$, $(\sigma,\tau)\to e^{-2\pi(\sigma+i\tau)}$, see \cite[Lemma 3.14]{bare}. The map in particular takes the 3-dimensonal kernel of $L_{(u,S^\circ)}\bar\partial_J$ into the 3-dimensional kernel of $L_{(u,S^\star)}\bar\partial_J$ which allows us to determine the latter from \eqref{eq : kernel marked ec} as follows. First, the image of $(v,\kappa,\partial_t)$ is
\[
(v,\kappa,\partial_t)+b,
\]
where $b\in TV_{\mathrm{sol}}^\circ\approx \C^3$ is given by $\partial_t\ev\in \C^3$. Second the image of $m_k$ is $\gamma_k+c_k$, where $\gamma_k\in TV^\circ_{\mathrm{con}}$ corresponds to the linearized movement of the puncture determined by $m_k$ and where $c_k\in TV_{\mathrm{sol}}^\circ$ is given by $\partial_{m_k}\ev\in\C^3$. 
By the transverse crossing assumption, $c_1$, $c_2$, $b$, and $\R^3\approx T_pL$ then generate $\C^3$ as a real vector space.

Thus, the subspace of $\mathbf{V}_{\mathrm{ec}}\oplus V_{\mathrm{con; ec}}$ spanned by 
\begin{equation}\label{eq : approx kernel ec}
\gamma_1+c_1 \ , \ \gamma_2 + c_2 \ , \ \partial_t + b,     
\end{equation}
is an approximate kernel in the sense that $L_{(u,S^\circ)}\bar\partial_J$ is invertible on its $L^2$-complement.

\subsubsection{Neighborhoods of elliptic node curves}\label{sssec : en}
We describe local coordinates for maps near the nodal map $(u^\bullet,S^\bullet)$ associated to map $(u,S)$ in Section \ref{sssec : ec}. The construction is directly analogous to Section \ref{sssec : hn}.

We consider first the variations of the conformal structure of $S^\bullet$. We have the same data as for $S^\circ$: the variation of the conformal structure of $S$ together with the location of the interior puncture. In addition, we have a new conformal parameter, the gluing parameter for the neck near the node, we write $V_{\mathrm{neck}}=[0,\epsilon)$ and take the corresponding neck to have length $\frac{2}{r}$, $r\in V_{\mathrm{neck}}$. As before, we define
\[
V_{\mathrm{con;\, en}} = V_{\mathrm{con}} \times V^\circ_{\mathrm{con}}\times V_{\mathrm{neck}}.
\]
We write $\rho=\frac{1}{r}$ and let $S^\bullet(\rho)$, $\rho\in[\rho_0,\infty)$, denote the domain obtained from $S^\bullet$ by removing $(\rho,\infty)\times S^1$ from the neighborhood of the puncture. We call the remaining region of $[0,\infty)\times S^1$, i.e. $[0,\rho]\times S^1\subset S^\bullet(\rho)$, the \emph{gluing neck}. 

Let $\beta_\rho\colon [0,2\rho]\times[0,1]\to [0,1]$ be a cut-off function that is equal to $1$ in $ [0,\rho-1]\times[0,1]$, equal to $0$ on $[\rho+1,2\rho]\times S^1$, and such that $\beta(\sigma,\tau)+\beta(2\rho-\sigma,\tau)=1$. Define $u^\bullet_\rho\colon S^\bullet(\rho)\to (X,L)$ by
\[
u^\bullet_{\rho,t}(z) = 
\begin{cases}
\psi_t(u^\bullet(z)), & z\in S^\ast(\rho)\setminus [0,\rho]\times S^1,\\
\psi_t(\beta(\sigma,\tau)\cdot u^\bullet(\sigma,\tau)+(1-\beta(\sigma))\cdot u^\dagger(2\rho-\sigma,\tau)), & z=(\sigma,\tau)\in [0,\rho]\times S^1,
\end{cases}
\]
where $u^\dagger$ denotes complex conjugation of $u$ in $\C^3$.
Then $u^\bullet_{\rho,t}(\rho\times S^1)\subset U\subset \R^3$ where $U$ is a small neighborhood of the origin in the local coordinates that take $U$ to $L$,   and $u^\bullet_{\rho,t}$ is $J_t$-holomorphic along the boundary. We will center our coordinates at $u^\bullet_{\rho,t}$.

We introduce a weight function $w_{\rho,\delta}\colon S^\bullet(\rho)\to[1,\infty)$,
\[
w_{\rho,\delta}(z)=
\begin{cases}
1, & z\in S^\bullet(\rho)\setminus [0,\rho]\times S^1,\\
e^{\delta|\sigma|}, & z=(\sigma,\tau)\in [0,\rho]\times S^1.
\end{cases}
\]

Let $H^{3,\delta}(S^\bullet(\rho),u^\ast TX)$ be the Sobolev space
with norm
\[
\|v\|^2 \ = \ \sum_{k=0}^{3}\int_{S^\bullet(\rho)} |d^{(k)}v|^2 w_{\rho,\delta} \,dA.
\]
Let $H_0^{3,\delta}(S^\bullet(\rho),u^\ast TX)\subset H^{3,\delta}(S^\bullet(\rho),u^\ast TX)$ be the codimension $3$ subspace of sections $v$ that satisfy the condition that the average over the boundary circle vanishes:
\[
\int_{S^1} v(\rho,t) dt = 0\in\R^3. 
\]

Let $\alpha_\rho\colon [0,\rho]\times S^1\to [0,1]$ be a cut off function where $\alpha_\rho=1$ in $[2,\rho]\times S^1$, $\alpha_\rho=0$ outside $[1,\rho]\times S^1$.
Extend $\alpha_\rho$ to all of $S^\bullet(\rho)$ by $0$. Assume that $\hat c\in\R^3$ is sufficiently small that it is contained in the coordinate chart of Lemma \ref{local coordinates} and define $c\colon (S^\bullet(\rho),\partial S^\bullet(\rho))\to (\C^3,\R^3)\subset (X,L)$:
\begin{equation}\label{eq : local solution en}
c(z)= 
\begin{cases}
0, & z\in S^\bullet(\rho)\setminus[0,\rho]\times S^1,\\
\alpha(z)\hat c, & z\in [0,\rho]\times S^1.
\end{cases}
\end{equation}

Let $V_{\mathrm{sol}}$ denote the 3-dimensional space of cut-off constants $c\colon S^\bullet(\rho)\to(\C^3\R^3)\subset(X,L)$ and define
\[
\mathbf{V}_{\mathrm{en},\rho} \ = \ H^{3,\delta}_0(S^\bullet(\rho);u_{\rho}^\ast TX)\oplus V_{\mathrm{sol}}. 
\]
Let $\mathbf{V_{\mathrm{en}}}=\bigsqcup_{r\in [0,\epsilon)} \mathbf{V}_{\mathrm{en},\,1/r}$, where for $r=0$ we let $\mathbf{V}_{\mathrm{en},\,\infty}=\mathbf{V}_{\mathrm{ec}}$.
The neighborhood of $(u^\bullet, S^\bullet)$ is then
\begin{equation}\label{eq : loc coord en}
\mathbf{V}_{\mathrm{en}}\times V_{\mathrm{con;\, en}}\times (-\delta,\delta),\quad ((v,c),(\kappa,r),t)\mapsto (\psi_t(\exp_{u^\bullet_{1/r,0}}(v)+c), S^\bullet(1/r)),
\end{equation}
where the addition  `$+$' is understood in the local chart as before.

Let $\hat u^\bullet_{\rho,t}$ be any map that agrees with $u^\bullet_{\rho,t}$ on $S^\bullet(\rho)\setminus [\frac12\rho,\rho]\times S^1$, that takes $[\frac12\rho,\rho]\times S^1$ into the $(\C^3,\R^3)$ coordinates and such that 
\[
\left|\hat u^\bullet_{\rho,t}|_{[\tfrac12\rho,\rho]\times S^1}\right|_{C^2}=\mathcal{O}(|e^{-2\delta}\rho|). 
\]
Consider the linearization $L_{(\hat u^\bullet_{\rho,t},S^\bullet(\rho))}\bar\partial_{J}$, for fixed $\rho=1/r$ :
\[
L_{(\hat u^\bullet_{\rho,t},S^\bullet(\rho))}\bar\partial_{J}
\colon \mathbf{V}_{\mathrm{en},\rho}\times TV_{\mathrm{con}}\times TV^\circ_{\mathrm{con}}\times\R \ \to \ 
H^{2,\delta}(S^\bullet(\rho),
\mathrm{Hom}^{0,1}
(T S^\bullet(\rho),
(\hat u^\bullet_{\rho,t})^\ast(TX))),
\]
(again fixed $\rho$ means that there is no component along $TV_{\mathrm{neck}}$.) 
\begin{lemma}\label{l : uniform invertiblity of differential e}
There exists $\rho_0>0$, $t_0>0$, and $C>0$ such that for all $\rho>\rho_0$ and $t<t_0$
\[
\|((v,c)\kappa,k\partial_t)\|\le C\|L_{(\hat u^\bullet_{\rho,t},S^\bullet(\rho))}\bar\partial_{J}((v,c),\kappa,k\partial_t)\|,
\]
in particular the Fredholm index zero map 
$L_{(\hat u^\bullet_{\rho,t},S^\bullet(\rho))}\bar\partial_{J}$ is a uniformly invertible isomorphism.
\end{lemma}

\begin{proof}
The proof is directly analogous to the proof of Lemma \ref{l : uniform invertiblity of differential} and differs only in the argument where it is shown that the limit in the kernel of $L_{(\hat u^\circ,S^\circ)}\bar\partial_{J}$, see \eqref{eq : approx kernel ec}, equals zero. In this case the components along $\gamma_k+c_k$ and $\partial_t+b$ equals zero since the asymptotic constants $c\in V_{\mathrm{sol}}$ for the nodal curve takes values in $\R^3$ whereas $c_k$, $k=1,2$ and $b$ span the complement of $\R^3$. 
\end{proof}

\subsection{1-parameter families of solutions}\label{ssec : familes of sols}
In this section we construct families of solutions to the Cauchy Riemann equations near $J_t$-transverse crossings and associated nodal curves.

\subsubsection{Solutions near hyperbolic and elliptic crossings}\label{sssec : sols hc and ec}
The following result follows directly from the definition.
\begin{lemma}\label{l : sols hc and ec}
    If $(u,S)$ is a $J_t$-transverse hyperbolic (or elliptic) crossing then the space of solutions 
    \[
    \bar\partial_{J_t}^{-1}(0)\subset \mathbf{V}_{\mathrm{hc}}\times V_{\mathrm{con}}\times (-\delta,\delta)
    \quad (\text{or}\quad \bar\partial_{J_t}^{-1}(0)\subset \mathbf{V}_{\mathrm{ec}}\times V_{\mathrm{con}}\times (-\delta,\delta))
    \]
    is a 1-manifold, the projection $\bar\partial_{J_t}^{-1}(0)\to(-\delta,\delta)$ is a $C^1$-diffeomorphism and solutions that project to $t\ne 0$ are embeddings.
\end{lemma}
\begin{proof}
This is an immediate consequence of the implicit function theorem for Fredholm operators and the description of the kernel of  $L_{(\hat u^\circ,S^\circ)}\bar\partial_{J}$ in Section \ref{sssec : hc} (or in Section \ref{sssec : ec}). 
\end{proof}

We write $A_t=u_t(S_t)$ for the images of the family of solutions in Lemma \ref{l : sols hc and ec}.

\subsubsection{Pre-gluing for hyperbolic nodes}\label{sssec : pre glue hyp}
We construct approximate solutions to the $\bar\partial_{J_t}$-equation in a neighborhood of $(u^\bullet, S^\bullet)$ in $\mathbf{V}_{\mathrm{hc}}\times V_{\mathrm{hc, con}}\times (-\delta,\delta)$ for each sufficuiently small $r\in V_{\mathrm{neck}}=[0,\epsilon)$.

In coordinates that satisfy \eqref{eq: J close to standard}, since $u$ is $J_0$-holomorphic with non-zero derivative at $s_k$, $k=1,2$, we have  Fourier expansion in the strip neighborhoods $[0,\infty)\times[0,1]$ of the nodes of the form
\[
u(\sigma,\tau)= c_{k;1} e^{-\pi(\sigma+i\tau)} + \Ordo(e^{-2\pi \sigma}),\quad c_{k,1}\in\R^3, \quad k=1,2.
\]
Consider map $u_{\mathrm{hol};\rho}\colon[-\frac12\rho,\frac12\rho]\times[0,1]\to(\C^3,\R^3)$ given by
\[
u_{\mathrm{hol};\rho}(\sigma,\tau)=
c_{1,1} e^{-\pi((\sigma+\frac12\rho)+i\tau)} + c_{2,1} e^{\pi((\sigma-\frac12\rho)+i\tau))}.
\]

\begin{remark}\label{r : preglu = model hyp}
Note that there is a real linear map $B\colon\C^3\to\C^3$ such that takes the image $u_{\mathrm{hol};\rho}([-\frac12\rho,\frac12\rho]\times[0,1])$ to the strip $D_{1/\rho}$ in the standard family in \eqref{eq:realnode}, $B(u_{\mathrm{hol};\rho}([-\frac12\rho,\frac12\rho]\times[0,1]))=D_{1/\rho}$. 
\end{remark}

Let $\beta\colon [-\rho,\rho]\times [0,1]\to [0,1]$ be a cut-off function such that $\beta=0$ on $S^\bullet(\rho)\setminus [-\frac12\rho-2,\frac12\rho+2]\times[0,1]$ and equal to $1$ on $[-\frac12\rho,\frac12\rho]\times[0,1]$. Define the pre-glued map $u_{\mathrm{pre};\rho,t}\colon (S^\bullet(\rho),\partial S^{\bullet}_\rho)\to (X,L)$
\begin{equation}\label{eq : preglue hn}
u_{\mathrm{pre};\rho,t}^\bullet(z)=
\begin{cases}
u^\bullet_{\rho,t}(z), & z\in S^\bullet(\rho)\setminus [-\rho,\rho]\times[0,1],\\
\psi_t((1-\beta(\sigma,\tau))\cdot u^\bullet(\sigma,\tau)+\beta(\sigma,\tau)\cdot u_{\mathrm{hol};\rho}(\sigma,\tau)), &
z=(\sigma,\tau)\in[-\rho,\rho]\times[0,1].
\end{cases}
\end{equation}

We next show that maps sufficiently close to $u_{\mathrm{pre};\rho,t}$ together with the hyperbolic crossing solution family of Lemma \ref{l : sols hc and ec} give standard families. Let $\|\cdot \|$ denote the norm induced by the norm in $\mathbf{V}_{\mathrm{hn},\rho}$.

\begin{lemma}\label{l : standard family hn}
    If $w_\rho\colon (S^\bullet(\rho),\partial S^\bullet(\rho))\to (X,L)$ is any family of maps with
    \[
    \|w_\rho-u_{\mathrm{pre};\rho,t}^\bullet\|=\mathcal{O}(e^{-(\pi-\eta)\rho}),
    \]
    where $\eta<\delta/2$ then there exists $r_0=\frac{1}{\rho_0}$ such that the family $B_r=w_{1/r}(S^\bullet(1/r))$ together with the family $A_t$, see Section \ref{sssec : sols hc and ec}, form a standard family.
\end{lemma}
\begin{proof}
The minimal distance between the two arcs in $u_{\mathrm{pre};\rho,t}(\partial S^\bullet(\rho))$ near $p\in L$ is bounded below by $K e^{-\pi\rho}$ for some constant $K>0$. By Remark \ref{r : preglu = model hyp}, it follows that any family of maps of $C^1$-distance $\mathcal{O}(e^{-(\pi+\gamma)\rho})$-distance from $u_{\mathrm{pre};\rho,t}$ satisfies the lemma. Since the $H^{3,\delta}$-norm controls the $C^1$-norm and since the weight function is larger than $e^{\frac12\delta\rho}$ where $u_{\mathrm{pre};\rho,t}$ differs from $u^{\bullet}_{\rho,t}$ the lemma follows.  
\end{proof}

We next show that $u^\bullet_{\rho;t}$ is almost holomorphic. Norms refer to the norm induced from 
$H^{2,\delta}(S^\bullet(\rho),\mathrm{Hom}^{0,1}(T S^\bullet(\rho),(\hat u^\bullet_{\rho,t})^\ast(TX)))$ by the local coordinates
\begin{lemma}\label{l : estimate pre glue hn}
    We have
    \[
    \|\bar\partial_{J_t}u_{\mathrm{pre};\rho,t^\bullet}\|=\mathcal{O}(|t|+e^{-(\frac32\pi-\delta)\rho}).
    \]
\end{lemma}
\begin{proof}
Consider first the case $t=0$, i.e., we estimate $\bar\partial_Ju_{\mathrm{pre};\rho,0}^\bullet$. By definition, $\bar\partial_Ju_{\mathrm{pre};\rho,0}^\bullet$ is holomorphic outside the region $[-\frac12\rho,\frac12\rho]\times[0,1]$ in the neck. Inside this region we have $|u_{\mathrm{pre};\rho,0}(\sigma,\tau)|=\mathcal{O}(e^{-\pi\sigma})$, along the map $|(J_0-I)|_{u_{\mathrm{pre};\rho,0}^\bullet(\sigma,\tau)}|=\mathcal{O}(e^{-2\pi\sigma})$ by \eqref{eq: J close to standard}. Since the glued in map is $I$-holomorphic and the weight of the exponential weight of the norm is $\delta$, it follows that
\[
\|\bar\partial_{J_0}u_{\mathrm{pre};\rho,0}^\bullet\|=\mathcal{O}(e^{-(\frac32\pi-\delta)\rho}).
\]

Taylor expanding $J_t=J_0+t\Gamma + \mathcal{O}(t^2)$ we see that $\|\bar\partial_{J_t}w-\bar\partial_J w\|=\mathcal{O}(|t|\|dw\|)$. The lemma follows. 
\end{proof}

\subsubsection{Pre-gluing for elliptic nodes}\label{sssec : pre glue ell}
We construct approximate solutions to the $\bar\partial_{J_t}$-equation in a neighborhood of $(u^\bullet, S^\bullet)$ in $\mathbf{V}_{\mathrm{ec}}\times V_{\mathrm{ec, con}}\times (-\delta,\delta)$ for each sufficiently small $r\in V_{\mathrm{neck}}=[0,\epsilon)$, in direct analogy with Section \ref{sssec : pre glue hyp}.

As before, we have the Fourier expansion in coordinates that satisfy \eqref{eq: J close to standard}
\[
u(\sigma,\tau) = c_{1} e^{-2\pi(s+it)} + \Ordo(e^{-4\pi s}),\quad c_{1}\in\C^3.
\]

Consider map $u_{\mathrm{hol};\rho}\colon[\frac12\rho,\rho]\times S^1\to(\C^3,\R^3)$ given by 
\[
u_{\mathrm{hol};\rho}(\sigma,\tau)=
c_{1} e^{-2\pi(\sigma+i\tau)} + c_{1}^\dagger e^{2\pi(2\rho-\sigma+i\tau))}.
\]

\begin{remark}\label{r : preglu = model ell}
Note that there is a real linear map $B\colon\C^3\to\C^3$ that takes the image $u_{\mathrm{hol};\rho}([\frac12\rho,\rho]\times S^1)$ to the cylinder $D_{1/\rho}$ in the standard family in \eqref{eq:imagnode}, $B(u_{\mathrm{hol};\rho}([\frac12\rho,\rho]\times S^1))=D_{1/\rho}$. 
\end{remark}

Let $\beta\colon [0,\rho]\times S^1\to [0,1]$ be a cut-off function such that $\beta=0$ on $S^\bullet(\rho)\setminus [\frac12\rho-2,\rho]\times S^1$ and equal to $1$ on $[\frac12\rho,\rho]\times[0,1]$. Define the pre-glued map $u_{\mathrm{pre};\rho,t}\colon (S^\bullet(\rho),\partial S^{\bullet}_\rho)\to (X,L)$
\begin{equation}\label{eq : preglue en}
u_{\mathrm{pre};\rho,t}^\bullet(z)=
\begin{cases}
u^\bullet_{\rho,t}(z), & z\in S^\bullet(\rho)\setminus [0,\rho]\times S^1,\\
\psi_t((1-\beta(\sigma,\tau))\cdot u^\bullet(\sigma,\tau)+\beta(\sigma,\tau)\cdot u_{\mathrm{hol};\rho}(\sigma,\tau)), &
z=(\sigma,\tau)\in[0,\rho]\times S^1.
\end{cases}
\end{equation}

We show that maps sufficiently close to $u_{\mathrm{pre};\rho,t}$ together with the elliptic crossing solution family of Lemma \ref{l : sols hc and ec} give standard families. Let $\|\cdot \|$ denote the norm induced by the norm in $\mathbf{V}_{\mathrm{en},\rho}$.

\begin{lemma}\label{l : standard family en}
    If $w_\rho\colon (S^\bullet(\rho),\partial S^\bullet(\rho))\to (X,L)$ is any family of maps with
    \[
    \|w_\rho-u_{\mathrm{pre};\rho,t}^\bullet\|=\mathcal{O}(e^{-(2\pi-\eta)\rho}),
    \]
    where $\eta<\delta/2$ then there exists $a_0=\frac{1}{\rho_0}$ such that the family $B_a=w_{1/a}(S^\bullet(1/a))$ together with the family $A_t$ see Section \ref{sssec : sols hc and ec}, form a standard family.
\end{lemma}
\begin{proof}
The smallest axis of the ellipse $u_{\mathrm{pre};\rho,t}(\partial S^\bullet(\rho))$ near $p\in L$ is bounded below by $Ke^{-2\pi\rho}$ for some constant $K>0$. By Remark \ref{r : preglu = model ell}, it follows that any family of maps of $C^1$-distance $\mathcal{O}(e^{-(2\pi+\gamma)\rho})$-distance from $u_{\mathrm{pre};\rho,t}$ satisfies the lemma. Since the $H^{3,\delta}$-norm controls the $C^1$-norm and since the weight function is larger than $e^{\frac12\delta\rho}$ where $u_{\mathrm{pre};\rho,t}$ differs from $u^{\bullet}_{\rho,t}$ the lemma follows.  
\end{proof}

We next show that $u^\bullet_{\rho;t}$ is almost holomorphic. Norms refer to the norm induced from 
$H^{2,\delta}(S^\bullet(\rho),\mathrm{Hom}^{0,1}(T S^\bullet(\rho),(\hat u^\bullet_{\rho,t})^\ast(TX)))$ by the local coordinates
\begin{lemma}\label{l : estimate pre glue en}
    We have
    \[
    \|\bar\partial_{J_t}u_{\mathrm{pre};\rho,t}^\bullet\|=\mathcal{O}(|t|+e^{-(3\pi-\delta)\rho}).
    \]
\end{lemma}
\begin{proof}
Directly analogous to the proof of Lemma \ref{l : estimate pre glue hn}. The difference in the size of the exponent in the right hand side comes from the difference in decay rates, the analogue of the argument in the proof of Lemma \ref{l : estimate pre glue hn} is as follows:

In the region where $u_{\mathrm{pre};\rho,0}^\bullet$ differs from $u^\bullet$, $|u_{\mathrm{pre};\rho,0}^\bullet(\sigma,\tau)|=\mathcal{O}(e^{-2\pi\sigma})$, along the map $|(J_0-I)|_{u_{\mathrm{pre};\rho,0}(\sigma,\tau)}|=\mathcal{O}(e^{-4\pi\sigma})$ by \eqref{eq: J close to standard}. Since the glued in map is $I$-holomorphic and the weight of the exponential weight of the norm is $\delta$, it follows that
\[
\|\bar\partial_{J_0}u_{\mathrm{pre};\rho,0}^\bullet\|=\mathcal{O}(e^{-(3\pi-\delta)\rho}).
\]
\end{proof}

\subsubsection{Floer gluing and solutions near hyperbolic and elliptic nodes}\label{sssec : sols hn en}
Floer gluing is a version of Newton iteration for Fredholm maps. We will apply it to the maps 
\begin{equation}\label{eq : dbar map hn}
\bar\partial_{J_t}\colon \mathbf{V_{\mathrm{hc},\rho}}\times V_{\mathrm{con}}\times V_{\mathrm{con}}^\circ\times(-\delta,\delta)\to H^{2,\delta}(S^\bullet(\rho), \mathrm{Hom}^{0,1}(TS,(u^\bullet_{\rho,t})^\ast(TX))).
\end{equation}
near the map with hyperbolic node and to
\begin{equation}\label{eq : dbar map en}
\bar\partial_{J_t}\colon \mathbf{V_{\mathrm{ec},\rho}}\times V_{\mathrm{con}}\times V_{\mathrm{con}}^\circ\times(-\delta,\delta)\to H^{2,\delta}(S^\bullet(\rho), \mathrm{Hom}^{0,1}(TS,(u^\bullet_{\rho,t})^\ast(TX))).
\end{equation}
near the map with elliptic node. 

A general version of the Floer-Picard lemma is given in \cite[Lemma 6.1]{ES}. To use this lemma, we denote the maps in \eqref{eq : dbar map hn} or \eqref{eq : dbar map en} for $\rho=1/r$ varying in $(0,\epsilon)$ simply as $f\colon X\to B$ and think of $X\to (0,\epsilon)$ and $B\to (0,\epsilon)$ as fibrations of Banach spaces. Then \cite[Lemma 6.1]{ES} says the following. Assume that there is a Taylor expansion at $0\in X_a$ in the fiber direction
\[
f_r(x) = f_r(0) + df_r(0)x + N_r(x)
\]
such that $df_r(0)$ is surjective and admits a smooth family of uniformly bounded right inverses $Q_a\colon B_r\to X_r$. Assume also that there exists $C>0$ such that the non-linear term satisfies 
\[
\|N_r(y)-N_r(x)\|_{B_r} = C\|x-y\|_{X_r}(\|x\|_{X_r}+\|y\|_{X_r}).
\]
Then if $\|Q_r f_r(0)\|<\tfrac{1}{8C}$, for $\delta<\frac{1}{4C}$, $f^{-1}_r(0_B)\cap \{x\colon \|x-0\|_{X}<\delta\}$ is a smooth submanifold diffeomorphic to a bundle over $(0,\epsilon)$ with fiber $\mathrm{ker}(df_r(0))$. 

\begin{lemma}\label{l : sols hn and en}
The set of solutions $\bar\partial_{J_t}^{-1}(0)$ in \eqref{eq : dbar map hn} and \eqref{eq : dbar map en} is a smooth $1$-manifold that projects diffeomorphically to $(0,\epsilon_1)\subset (0,\epsilon)$ for all sufficiently small $\epsilon_1>0$. Furthermore, if $u^\bullet_{\mathrm{glue};\rho,t}$ is the solution at $\rho=1/r$, $\bar\partial_{J_t}(u^\bullet_{\mathrm{glu};\rho,t})$ then
\begin{equation}\label{eq : initial value}
\|u^\bullet_{\mathrm{glu};\rho,t}-u^\bullet_{\mathrm{pre};\rho,t}\|=\mathcal{O}(\|\bar\partial_{J_t}(u^\bullet_{\mathrm{pre};\rho,t})\|).
\end{equation}
\end{lemma}

\begin{proof}
We apply the Floer-Picard lemma as described above. Invertibility of the differential follows from Lemmas \ref{l : uniform invertiblity of differential} and \ref{l : uniform invertiblity of differential e}. The estimate for the non-linear term follows from standard arguments estimating the norm of the exponential map acting on a vector field, see \cite[Lemma 8.16]{EESLCHRn} and \cite[Equation (4.17)]{EESPxR}. The statement on the solution space follows.

The estimate \eqref{eq : initial value} follows from \cite[Equation (6.3)]{ES} in the proof of \cite[Lemma 6.1]{ES}. (It holds since the solution is produced by iteration of a contraction.)  
\end{proof}

\subsubsection{Proof of Theorem \ref{t : wall crossing gluing}}
Let $(u,S)$ be a $J_t$-transverse hyperbolic or elliptic crossing. Lemma \ref{l : sols hc and ec} give the $1$-parameter family of solutions with images $A_t$. 

We now apply Lemma \ref{l : sols hn and en} to $u_{\mathrm{pre};\rho,t}$ where we take $|t|\le e^{-(2\pi-2\delta)\rho}$ in the elliptic case and $|t|<  e^{-(\pi-2\delta)\rho}$ in the hyperbolic case. Then it follows from Lemmas \ref{l : estimate pre glue en} and \ref{l : estimate pre glue hn} in combination with \eqref{eq : initial value} that
\[
\|u^\bullet_{\mathrm{glu};\rho,t}-u^\bullet_{\mathrm{pre};\rho,t}\|=\mathcal{O}(e^{e^{-(2\pi-\delta)\rho}}),\quad
\|u^\bullet_{\mathrm{glu};\rho,t}-u^\bullet_{\mathrm{pre};\rho,t}\|=\mathcal{O}(e^{e^{-(\pi-\delta)\rho}}),
\]
in the elliptic and hyperbolic cases. If $B_a$ are the images of the solutions then Lemmas \ref{l : standard family en} and \ref{l : standard family hn} implies the theorem.\qed

\subsubsection{Detailed boundary information}\label{sssec : fine points}
For skein valued curve counting, we will require only a weak form of control of moduli spaces as in the definitions of standard families in Definitions \ref{standard hyperbolic degeneration} and \ref{standard elliptic degeneration}. 

The Floer gluing arguments above that control boundaries of moduli spaces near nodal curves give more precise information. Although not necessary for the applications in this paper, we include the following stronger result. 
Let $\mathcal{M}(X,L,J)$ denote the space of solutions $\bar\partial_{J_t}^{-1}(0)$ of Lemma \ref{l : sols hn and en} compactified by adding the nodal solution $(u^\bullet,S^\bullet)$, $\bar\partial_{J_0}(u^\bullet)=0$. 

\begin{proposition}\label{p : fine points}
For generic path $J_t$ of almost complex structures, there exists a constant $c\ne 0$ such that
the germ of the projection $P\colon\mathcal{M}(X,L,J)\to [0,\epsilon)\times (-\delta,\delta)$, $P(u_{\mathrm{glue};\rho,t})=(\frac{1}{\rho},t)$ and $P(u^\bullet)=(0,0)$ at $(0,0)$ agrees up to second order with the curve $r\mapsto (r,c r^2)$, $r\in[0,\epsilon)$. 
\end{proposition}

\begin{proof}
Lemmas \ref{l : uniform invertiblity of differential} (respectively \ref{l : uniform invertiblity of differential e}) implies that a unit vector in the kernel of $D_{\mathrm{en};\rho,t}$ (respectively $D_{\mathrm{en};\rho,t}$) has a non-zero component $\nu_\rho\in TV_{\mathrm{neck}}$, and that $|\nu_\rho|>\nu_0>0$ for all $\rho\in[\rho_0,\infty)$.
At $(u^\bullet,S^\bullet)$, $TV_{\mathrm{neck}}$ equals the kernel of the differential, and the cokernel of the linearized operator $L_{(u^\bullet,S^\bullet)}\bar\partial_{J_0}$ (with $t$ fixed at $t=0$) has cokernel of dimension $1$ spanned by $\dot J_0$. Recall that the second derivative of a quadratic map from the kernel of the differential to the cokernel, where then can be identified with a quadratic map $T[0,\epsilon)\to T(-\delta,\delta)$, $r\mapsto c r^2$. Imposing the codimension one condition $c=0$ would reduce the index of the correspondingly constrained $\bar\partial_{J_t}$-map. Hence if the family $J_t$ is generic, $c\ne 0$. The lemma follows.  
\end{proof}

\subsection{Orientations at hyperbolic and elliptic crossings/nodes}
In this section we show that relative orientations between crossing and nodal curves are compatible between the hyperbolic and elliptic cases. For generalities about orientations of moduli spaces, see Appendix \ref{sec : basic orientations}. 

We first show that relations between orientations of moduli spaces at elliptic and hyperbolic nodal curves and respective normalizations are locally determined. Let $J_t$, $t\in(-\epsilon,\epsilon)$ be a generic path of almost complex structures and let $\Omega(u)\subset\mathcal{M}(X,L)$ be a $1$-manifold around the crossing curve $u_0$ at $t=0$ and $\Omega(v)\subset\mathcal{M}(X,L)$ be a $1$-manifold with boundary the nodal curve $v$ as in Theorem \ref{t : wall crossing gluing}, see Proposition \ref{p : fine points} for parameterizations.  

\begin{lemma}\label{l : boundary orientation}
The orientation of moduli space $\Omega(u)$ at $u$ together with the local crossing/node model determines the orientation of the moduli space $\Omega(v)$ at $v$ (and vice versa). 
\end{lemma}

\begin{proof}
To see this we consider determinant bundles of the $\bar\partial_J$ operator. We start in the less complicated elliptic case. Here we degenerate the linearized operator over the curve with an extra boundary to the closed curve with a sphere attached which further has a disk attach to it. Here the sphere and the disk constitutes a local model glued to the original closed curve. Marked points and gluing constraints are all even dimensional and the rotation parameters in the disk and sphere are coupled. It follows that the relevant index bundle is expressed as a product of the index bundle over the original curve and an index bundle in the model.

In the case of the hyperbolic node we argue similarly. Introduce two boundary marked points where we attach two disks each with a further marked point where the maps are required to agree. Now the boundary marked points shifts the index by $1$ each and the orientation depends on the order. The gluing locus is the inverse image of the diagonal under the product of evaluation maps that also depends on the order, switching the order switches orientation twice. Hence the result is independent of order and the relevant index bundle is expressed as a product of the index bundle of the original curve and the index bundle in the model.      
\end{proof}

\begin{remark} 
In the language of Proposition \ref{p : fine points}, Lemma \ref{l : boundary orientation} says that the orientation of $\Omega(u)$ at $t=0$, $\pm\partial_t$ determines the orientation $\pm\partial_{r}$ at $\Omega(v)$ at $(r,t)=(0,0)$. 

We consider also the effect of the sign of $c$ in Proposition \ref{p : fine points}: if $\Omega(v)$ is oriented by $\partial_r$ at $(t,r)=0$ and if $c>0$ then $\Omega(v)\ne \emptyset$ over $t>0$ and the oriented tangent vector of $\Omega(v)$ has positive component along $+\partial_t$, whereas if $c<0$ then $\Omega(v)\ne \emptyset$ over $t<0$ and the oriented tangent vector of $\Omega(v)$ has positive component along $-\partial_t$.    
\end{remark}

We next compare the relative signs between the crossing and nodal families of Lemma \ref{l : boundary orientation} in the elliptic and hyperbolic cases. Consider an oriented standard hyperbolic or elliptic crossing family over $t\in(-\delta,\delta)$, with crossing at $t=0$. Then for $t\ne 0$ the maps in the family has a crossing sign: the crossing sign of the oriented nearby boundary strands in the hyperbolic case and the crossing sign between the curve and the Lagrangian in the elliptic case. We orient $(-\delta,\delta)$ so that the positive crossing maps appear in the positive half interval $(0,\delta)$. The orientation of the crossing family is then either positive or negative as compared to the orientation of $(-\delta,\delta)$. Consider an oriented standard nodal family, at its nodal instance the family is oriented by the normal to the nodal locus we say that the orientation is either outward or inward according to whether the orientations points out of or into the space of smooth domains. Lemma \ref{l : boundary orientation} then gives two functions
\[
F_{\mathrm{e}}, F_{\mathrm{h}}\colon \{\text{positive, negative}\} \to \{\text{inward, outward}\},
\]
where $F_{\mathrm{e}}$ and $F_{\mathrm{h}}$ give the orientations induced from $\Omega(u)$ on $\Omega(v)$ in the elliptic and hyperbolic cases, respectively.

\begin{lemma}\label{l : same orientations elliptic and hyperbolic}
The functions $F_{\mathrm{e}}$ and $F_{\mathrm{h}}$ agree.
\end{lemma}

\begin{proof}
\begin{figure}[htp]
	\centering
	\includegraphics[width=.55\linewidth]{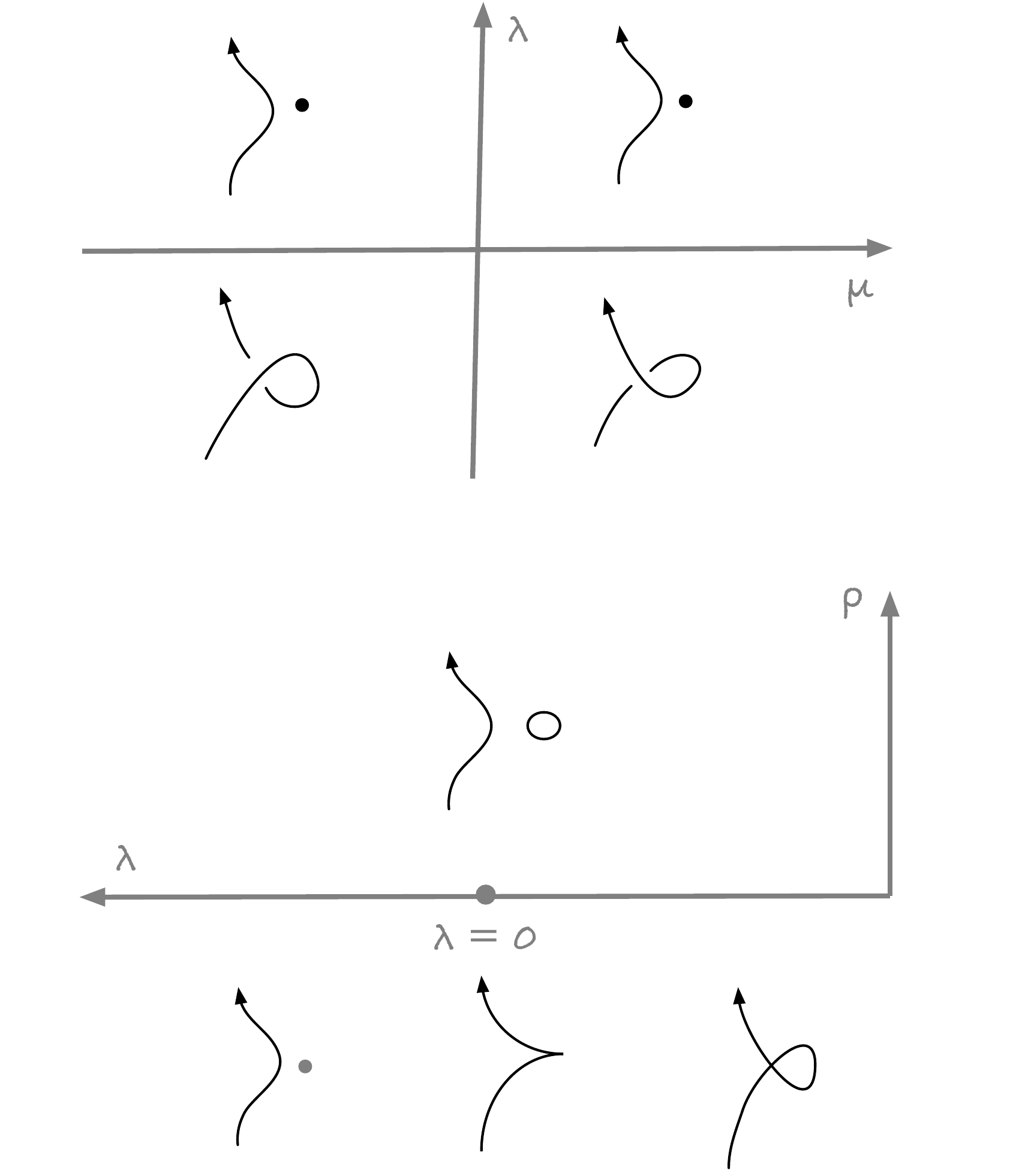}
	\caption{Relation between hyperbolic and elliptic nodes near a cusp.}
	\label{fig:cusp}
\end{figure}

Consider the 2-parameter family of holomorphic maps defined for $\zeta$ in a neighborhood of the origin in the upper half plane $\{z=s+it\in\C\colon y\ge 0 \}$, see Figure \ref{fig:cusp}, upper part:
\[ 
z\mapsto \left(\zeta^{2}, \zeta(\zeta^{2}+\lambda), \mu \zeta\right),
\]   
where $(\lambda,\mu)$ lies in a neighborhood of the origin in $\R^{2}$. For $\mu=0$, $\lambda<0$ the holomorphic map has a hyperbolic boundary crossing and for $\mu=0$, $\lambda>0$ the curve has an elliptic boundary crossing.

Just as in the 1-parameter case in Lemma \ref{l : boundary orientation}, the orientation of this 2-dimensional space of curves is determined by the index bundle over the central curve with a cusp together with an orientation of the 2-dimensional $(\lambda,\mu)$-parameter space.  Consider the 1-parameter family of curves with $\mu\ne 0$ fixed and $\lambda$ changing sign. Such a family corresponds to a framing change and we know that the curve count in the framed skein module does not change. 

Consider next the lower part of Figure \ref{fig:cusp}. Here we depict the $\lambda$-axis as the boundary of another family of holomorphic curves of Euler characterstic one smaller. For generic $\lambda$ and $\rho$ the boundary is as depicted. For $\lambda<0$ the distance between the components shrinks faster than the radius of the circle, for $\lambda>0$ the circle shrinks faster, and at $\lambda=0$ they shrink at the same rate and we limit to a cuspidal curve. 

We conclude from these pictures and the orientation on the 2-parameter family that the orientation of curve at $(\lambda,\rho)$, $\rho>0$, is induced by either the inward or outward orientation in both the elliptic and the hyperbolic cases, i.e., $F_{\mathrm{e}}=F_{\mathrm{h}}$.  
\end{proof}

\section{Bare compactness and somewhere injectivity in basic classes}
Deformation invariance of skein valued curve counts depend crucially on the results of Section \ref{sec: gluing} that control wall crossings in generic 1-parameter families of holomorphic maps. 
Here we discuss two additional ingredients. 
In Section \ref{ssec : ghost} we recall a compactness result for holomorphic curves without ghost components. In Section \ref{sec : somewhere inj} we show that holomorphic curves in basic homology classes are somewhere injective, which gives transversality by standard arguments (in turn recalled in Appendix \ref{ssec : gen Fredholm injective}).

\subsection{Bare compactness}\label{ssec : ghost}
As mentioned in Section \ref{Intro}, skein valued curve counts depend on controlling the wall crossings in generic 1-parameter families of bare holomorphic curves (i.e., curves with only positive area components) and on compactness for such curves. Note that compactness for bare curves is not a consequence of Gromov compactness, a sequnece of bare curves could converge to a curve with a ghost component attached. The following result shows that generically such ghost bubbling can be avoided in Gromov limits. 

\begin{theorem} \label{lem:compactness} \cite[Theorem 1.1]{ghost}
	Let $u_\alpha\colon(S_\alpha, \partial S_\alpha) \to (X, L)$ be a sequence of immersed $J_{\alpha}$-holomorphic curves.  Suppose the sequence converges to a $J$-holomorphic $u\colon (S, \partial S) \to (X, L)$, 
	with a collapsed component $S_0$, i.e., $u(S_0)$ is a point.  
Let the bare (positive symplectic area) part of $u$ be $u_+\colon (S_+, \partial S_+) \to (X, L)$.

Then the map $u_{+}$ has a triple point, or a double point with linearly dependent tangents, or there is a point $\zeta \in S_+$ where $S_0$ is attached and where the differential of $u_{+}$ at $\zeta$ vanishes, $du_{+}(\zeta)=\partial_J u_+(\zeta) = 0$.  
\end{theorem}

Let us mention that the above result has many related antecedents: 
\cite{Zinger-reduced, Zinger-reduced-CY, Hu-Li-reduced, RMMP}.

\subsection{Somewhere injectivity for curves in basic homology classes}\label{sec : somewhere inj}
In this section we establish injectivity properties for holomorphic curves. Recall that we call a holomorphic curve $u\colon (S,\partial S)\to (X,L)$ bare provided all its irreducible components have positive area. We say that $(u,S)$ is somewhere injective if in any irreducible component of $S$ there is a non-empty open subset $U$ such that $u|_U$ is an embedding and $u^{-1}(u(U))=U$.    

We start with the exact case. Consider an exact symplectic $2n$-manifold $X$ with symplectic form $\omega=d\lambda$, which outside a compact subset has the form $[0,\infty)\times Y$, where $Y$ is a contact $(2n-1)$-manifold with contact form $\alpha$ and where $\lambda=e^t\alpha$ for $t\in[0,\infty)$. Let $L=L_0\cup L_1\cup\dots\cup L_m$ be an asymptotically cylindrical embedded Lagrangian in $X$ where $L_0\subset X$ is a compact exact Lagrangian and each component $L_j$, $1\le j\le m$, of $L$ has topology $S^1\times\R^{n-1}$ and ideal boundary a Legendrian $S^1\times S^{n-2}$. The condition that $L$ is asymptotically cylindrical is the following. Fix a complex structure $J_\alpha$ on the contact plane and consider the $\R$-invariant complex structure $J_{\R}$on $[0,\infty]\times Y$ given by $J_\alpha$ on the contact planes and such that $J\partial_t=R_\alpha$, where $R_\alpha$ is the Reeb vector field of $\alpha$ on $Y$. Together with the symplectic form $d(e^t\alpha)$, $J_{\R}$ induces a Riemannian metric $d(e^t\alpha)(\cdot,J\cdot)$ and we say that $L$ is asymptotically cylindrical with ideal boundary a Legendrian $\Lambda\subset Y$ if there exists $T>0$ such that the distance between $L\cap [T,\infty)\times Y$ and $[T,\infty)\times \Lambda$ is bounded. Consider a homology class $\beta\in H_2(X, L\cup S)$. 

\begin{definition}\label{def : basic class}
The class $\beta\in H_2(X,L\cup S)$ is a \emph{basic class} if its image under the connecting homomorphism $\partial\colon H_2(X,L\cup S)\to H_1(L)$ projects to a generator $\gamma_j\in H_1(L_j)$ for $1\le j\le m$ (note $H_1(L_j)=\Z$ here) and if
\[
    \int_{\gamma_j}\lambda>0.
\]
\end{definition}

Consider the case when $L=L_0\cup L_1$ has only one non-compact component.
\begin{lemma}\label{l : somwhere injectivity exact}
Let $L\subset X$ be as above. If $J$ is an almost complex structure on $X$ that is uniformly compatible with $\omega$ then any bare $J$-holomorphic curve with boundary on $L$ in a basic homology class is injective in the complement of finitely many points and in particular somewhere injective. 
\end{lemma} 

\begin{proof}
The proof is similar to \cite[Theorem 2.9]{D-RET}. Let $u\colon (S,\partial S)\to (X,L)$ be a $J$-holomorphic map in a basic homology class $\beta$ with $\partial\beta=\gamma_0+\gamma_1\in H_1(L_0)\oplus H_1(L_1)$. Then $u$ has area $a=\int_{\gamma_1}\lambda>0$ and is absolutely area minimizing (with respect to the metric $\omega(\cdot,J\cdot)$) in its homology class.   

Consider the image $C=u(S)$ as a rectifiable current. By unique continuation for holomorphic maps, $C$ can be written as
\begin{equation}\label{eq:integralcurrentdecomp} 
	C= C_{1}+ 2C_{2} +\dots + mC_{m},
\end{equation}  
where $C_{j}$ is the subset of $C$ where the multiplicity equals $j$. It is clear that each $C_{j}$ is a holomorphic rectifiable current. 
	
Assume that $u$ is not somewhere injective. Then $C_{1}=\emptyset$ in \eqref{eq:integralcurrentdecomp}. We show below that $C_{m}$ is an integral current, i.e.~that $\partial C_{m}$ is rectifiable and contained in $L$. But then the area of $C_m$ is at least $a$ which contradicts \eqref{eq:integralcurrentdecomp}. It follows that the maximal multiplicity $m$ equals $1$, hence $C=C_1$ as claimed.

We check that $C_m$ is an integral current.
We first show that the boundary of $C_{m}$ is contained in $L$. Assume that $p\notin L$ lies in the support of $C_{m}$. By Federer's refinement of Sard's theorem \cite[Theorem 3.4.3]{Federer} we may assume that $C_{m}$ is a smooth 2-dimensional submanifold around $p$. By unique continuation it follows that there is a disk $\Delta$ centered at a point in $u^{-1}(p)$ such $u(\Delta)$ lies in the support of $C_{m}$. Hence $p$ does not lie on the boundary of $C_{m}$ and its boundary is contained in $L$. 
	
Consider a point $q\in L$ on the boundary of $C_{m}$. Again by the refinement of Sard's theorem,  one may assume that $q=u(z)$ is a regular value of $u|_{\partial S}$. By unique continuation, the boundary of $C_{m}$ consists of an arc in $L$ around $q$. It follows by monotonicity for minimal surfaces that the total length of the boundary arcs in the boundary of $C_m$ is finite. Thus $\partial C_{m}$ is a rectifiable curve in $L$ and hence $C_{m}$ is an integral current. 
\end{proof}

We next consider a generalization of Lemma \ref{l : somwhere injectivity exact} where $X$ is allowed to be non-exact. We take $X$ to be asymptotically cylindrical at infinity and assume that in the cylindrical end $[0,\infty)\times Y$, the symplectic form $\omega$ on $X$ has the form $\omega=d(e^t\alpha)+\omega_0$, where $\alpha$ is a contact form on $Y$ and $\omega_0$ is bounded. We consider Lagrangians $L=L_1\cup\dots L_m$ in the cylindrical end, $L\subset [0,\infty)\times Y$, as before with components of topology $S^1\times\R^{n-1}$ and asymptotically cylindrical at infinity. (The previously considered compact Lagrangian $L_0$ is now empty). 

\begin{lemma}\label{l : omega0 small}
If $\omega_0$ is sufficiently small then there exists an isotopy $L^s$ of $L=L^0$ in a cotangent neighborhood of $L$ such that $L^s$ is Lagrangian with respect to $d(e^t\alpha)+(1-s)\omega_0$.   
\end{lemma}

\begin{proof}
This is a consequence of Moser's trick. Fix a cotangent neighborhood of $L$ and let $\omega(s)=d(e^t\alpha)+(1-s)\omega_0$. Then $\frac{d}{ds}\omega(s)=-\omega_0$. Since $\omega_0$ is closed in $T^\ast L$ is also exact, $-\omega_0=d\beta$. Then the flow of the time dependent vector field $X(s)$ that satisfies $\beta=\omega(s)(X(s),\cdot)$ gives the desired isotopy.
\end{proof}

We consider basic classes in the non-exact setting. Let $\beta\in H_2(X,L)$ and $\partial\beta\in H_1(L)$ with components $\gamma_j\in H_1(L_j)$ that are generators. Assume that $\omega_0$ is sufficiently small so that Lemma \ref{l : omega0 small} holds and let $L^s$, $0\le s\le 1$, be an isotopy as there so that $L^1$ is Lagrangian with respect to $d(e^t\alpha)$. Let $\gamma_j^1$ be the image of $\gamma_j$ in $H_1(L^1_j)$ under the isotopy. We say that $\beta$ is a \emph{basic class} if 
\begin{equation}\label{eq : basic non-exact}
\int_{\gamma_j^1} \alpha>0,
\end{equation}
for all $j$.

Assume that $L$ lies in the cylindrical end $L\subset [2,\infty)\times Y$ and that the almost complex structure $J=J_{\R}$ in $[0,1]\times Y$ is an $\R$-invariant almost complex structure compatible with $d(e^t\alpha)$. Assume that $L=L^0$ admits an isotopy $L^s$ as in Lemma \ref{l : omega0 small}. If $L$ has more than one component we furthermore assume that there are translation invariant nested subsets $N_j\subset N'_j\subset [2-\delta,\infty)\times Y$, such that $L_j\subset N_j$, where $L_j$, $j=1,\dots,k$ are the components of $L$, and such that no holomorphic curve intersects $N'_j\setminus N_j$.   

\begin{lemma}\label{l : somewhere injective non-exact}
If $L$ and $J$ is as above then any bare $J$-holomorphic curve in a basic class is injective on the boundary in the complement of finitely many points and in particular somewhere injective. 
\end{lemma}

\begin{proof}
Consider first the case of one component. Let $\psi_s\colon X\to X$ be a 1-parameter family of diffeomorphisms equal to the identity outside a neighborhood of $L$ and such that $\psi_s(L)=L_s$. Write $J_+$ for the almost complex structure $d\psi_1^{-1}\circ J\circ d\psi$ in $[\frac32,\infty)\times Y$. Note that if $\omega_0$ is sufficiently small compared to $d(e^t\alpha)$ then $\psi_1$ is close to the identity and $J_+$ is compatible with $d(e^t\alpha)$ in $[1,\infty)\times Y$.

Let $u\colon (S,\partial S)\to (X,L)$ be a $J$-holomorphic map. Then $\psi_1\circ u\colon (S,\partial S)\to (X,L_1)$ is a $J_+$-holomorphic map. Let $\phi\colon[0,\infty)\to[0,1]$ be a smooth function equal to $0$ in $[0,1]$ and equal to $1$ in $[2,\infty)$. 

For $T>0$, consider the map 
\begin{equation}\label{eq: shift vertically}
\phi_T\colon [0,\infty)\times Y\to [0,\infty)\times Y,\quad (t,y)\mapsto (t+T\phi(t),y).
\end{equation}
Consider the almost complex structure $J_T$ that is $\R$-invariant in the region $[0,T]\times Y$ and equal to $J_+$ in $[T,\infty)\times Y$. It is easy to see that the $J_T$ is compatible with $d(e^t\alpha)$. 

Consider the $J_T$-holomorphic map $\phi_T\circ u$ and apply the argument of the proof of Lemma \ref{l : somwhere injectivity exact} to get the decomposition of the image current $C=\phi_T\circ u(S)$:
\begin{equation}\label{eq : multiplicity decomposition}
C=C_1+ 2C_2 +\dots +m C_m.
\end{equation}
The area of any closed component of $C_j$ is unaffected by $\phi_T$ and the total area of $C$ grows linearly with $T$. At least one of the integral currents $C_k$ must have boundary $\partial C_k$ in a homology class that is a multiple of the generator $\gamma$ of $H_1(L)$ on which the flux $\mathfrak{a}_T$ of $\phi_T$ is positive, but then, since $k\ge 2$ we would find that for $T$ sufficiently large, the area of $C$ is at least $\frac74\cdot \int_{\gamma} (e^t\alpha)$. However, the curve $\phi_T\circ u$ lies in a basic homology class and the area is $\le \frac54\int_{\gamma} (e^t\alpha)$. We conclude that only $C_1$ has boundary on $L$. The lemma follows.  

In the case of many components, we show that the boundary of each $C_k$ has only positive components. We argue by contradiction. Assume that $C_k$ has boundary component in $L_j$ that is negative. Apply \eqref{eq: shift vertically} with $\phi_T$ as there and in addition cut off from $1$ to zero in $N_j'\setminus N_j$. Since no curve passes the cut off region the resulting curve is $J_T$-holomorphic but will have negative area for sufficiently large shift. We conclude that all $C_k$ have only positive boundary components. It is then immediate from the decomposition \eqref{eq : multiplicity decomposition} that $C=C_1$. The lemma follows.
\end{proof}

We finally show that neighborhoods as required in Lemma \ref{l : somewhere injective non-exact} for many component Lagrangians exists for link conormals provided the shift is suffficietly small. 

\begin{lemma}\label{l : nested nbhds for links}
Let $K$ be a $m$-component link $m\ge 2$. Then there exists $\epsilon_0>0$ such that for all $\epsilon<\epsilon_0$, $L_{K;\epsilon}$ admits neighborhoods $N_{j}\subset N_j'\subset [\epsilon/2,\infty)\times ST^\ast S^3$ such that no holomorphic curve in a basic homology class intersects $N_j'\setminus N_j$. 
\end{lemma}

\begin{proof}
Take $\epsilon$ so that $\epsilon^{1/8}$ is smaller that the minimal distance between the links. Let $N_j\subset N'_j$ be $\epsilon^{1/4}$-neighborhoods of the lift of the knot shifted by its tangent vector to $S_\epsilon T^\ast S^3$. By monotonicity, a curve in $N'_j\setminus N_j$ has area $\ge c \epsilon^{1/2}$ for some $c>0$. Since the area of a curve in the basic homology class equals $\mathcal{O}(\epsilon)$ the lemma follows.  
\end{proof}

\section{Skein valued curve counts}
In this section we collect properties of holomorphic curves and show how they lead to deformation invariant skein valued curve counts.

By stable map, we mean as usual a holomorphic map with finitely many automorphisms.\footnote{Some authors also allow maps with compact automorphism groups, which means they consider the disk with one marked point stable. For those authors, the limit of the elliptic degeneration is a curve with a ghost such disk attached.  For us, the corresponding limit is instead a curve with an `elliptic node' as modeled in Remark \ref{r : ell node as node}, and as used in the gluing analysis in Sections \ref{sssec : en} and \ref{sssec : pre glue ell}.} 

\begin{theorem}\label{thm:somewhereinjective}
Let $(X,L)$ be as in Section \ref{sec : somewhere inj} and let $A\in H_{2}(X,L)$ be a basic homology class.
Suppose $\mathcal{J}$ is a subset of the space of almost complex structures that gives transversality for holomorphic curves in class $A$.  
Let $\M(J)$ for $J\in\mathcal{J}$ be the moduli space of bare stable $J$-holomorphic maps.  Then the following hold. 

\begin{enumerate}
\item 
\label{prop:coherence} 
	{\bf Coherence.} For any $J \in \Jj$, we have $u \in \M(J)$ if and only if $\tilde{u} \in \M(J)$. 
\item 
\label{prop:codimzero}
{\bf Codimension zero transversality.}  
	There is a nonempty subset $\Jj^\circ \subset \Jj$ such that for $\lambda \in \Jj^\circ$, 
	\begin{enumerate}
	\item \label{prop:codimzerocompact} 
		$\M_{g,h}(J)$ is compact. 
	\item \label{prop:isolated} 
		$\M(J)$ is an oriented zero-manifold.
	\item \label{prop:embedded} 
	        Any map $u$ corresponding to a point in $\M(J)$ is an embedding of 
	        a smooth curve
	        and is transverse to $L$ in the sense of Definition \ref{transversetolag}.  In particular, $u$ is bare.  
	\item \label{prop:trivialcobordism}
		For any path $J\colon[0,1]\to\Jj^\circ$, the map $\M(J) \to [0,1]$ is a proper cover 
		of  manifolds, and $\M(J)$ can be oriented so that the orientations of $\M(J_0), \M(J_1)$ are recovered from the boundary orientation.  
	\end{enumerate}
\item \label{prop:codimone} {\bf Codimension one transversality.}
	Any two points $J_0, J_1 \in \Jj^\circ$ can be connected 
	by a path $J\colon[0,1]\to \Jj$ such that: 
	\begin{enumerate}
	\item \label{prop:codimonecompact} 
		$\M_{g,h}(J)$ is compact.  		
	\item \label{prop:boundarycobordism} 
        $\M(J)$ is an  1-manifold with boundary, to which  
        the orientation 
        on $\M(J)|_{\mathcal{J}^0}$  from \eqref{prop:trivialcobordism}
        extends.  Boundary  points over the interior $(0,1)$ are precisely the maps in $\mathcal{M}(J)$ with nodal domain.
	\item \label{prop:onecrossingatatime}
        The locus of $t \in [0,1]$ so where \eqref{prop:trivialcobordism} or \eqref{prop:embedded} fails is discrete, and the failure over some $t_0$ occurs at only point of $\mathcal{M}(J_{t_0})$. If \eqref{prop:trivialcobordism} fails at $[u]$, we term it a \emph{critical point}; if \eqref{prop:embedded} fails, 
		  we term it 
		  a \emph{crossing}. 
	\item \label{prop:crossingtypes} 
		The universal map over a neighborhood of a crossing $[u] \in U \subset
		\Mm[0,1]$ takes one of the following forms: 
		\begin{enumerate}
		\item \emph{Hyperbolic crossing}.  
			The map $u$ is an immersion everywhere and an 
			embedding save at two points $s_1\ne s_2$ in the boundary such that $u(s_1)=u(s_2)$. We require that the images of the two boundary 
			tangent vectors at $s_1$ and $s_2$ are linearly independent from each other and from $\xi$, and 
			that they together with the first order variation of the $1$-parameter family at the double point, $\partial_t u(s_1)-\partial_t u(s_2)$, span the tangent space of $L$.  
		\item \emph{Elliptic crossing}.  
			The map $u$ is an embedding, but some interior point $s$ of the domain is 
			mapped to $L$.  The map at this point is transverse to the 4-chain $C$. 
			At the crossing moment the tangent space of the curve and the tangent 
			space of $L$ together with the first order variation of $1$-parameter 
			family $\partial_t u(s)$ span the tangent space of $X$. 
		\item \emph{Framing change}: 
			$u$ is an embedding, but 
			$\partial u$ becomes tangent to $\xi$ or intersects $\gamma$ generically (as 
			in Lemmas \ref{l:throughgamma} or \ref{l:tangency}). 
		\end{enumerate}
	\end{enumerate}
\item \label{prop:standard} {\bf Gluing.} 
	Let $[u] \in \M(t_0)$ be an elliptic or hyperbolic 
	crossing as in $\mathrm{(3d)}$ above.  Let $v = u^\bullet$ be the map with the same image which is an embedding of a nodal curve.  Let $\M_u$ and $\M_v$ be small neighborhoods of $[u]$ and $[v]$ in $\M(t_0 - \epsilon, t_0+\epsilon)$,  
	and let $u_t,v_t$ be the corresponding families of maps. Then there are neighborhoods $D, D^\bullet$ of the crossing point and node such that $(u_t(D), v_t(D^\bullet))$ is a standard
	hyperbolic or elliptic degeneration, in the sense of 
	Definition \ref{standard hyperbolic degeneration} or \ref{standard elliptic degeneration}, respectively.  
    
    Moreover, the orientations of $\M_u$ and $\M_v$, see Section \ref{sec : basic orientations}, are related as follows: the moduli space $\M_v$ is non-empty over a half-interval $(t_0-\epsilon,t_0]$ or $[t_0,t_0+\epsilon)$, over which we may compare its orientation with $\M_u$ using the projections of both to $(t_0-\epsilon, t_0+\epsilon)$. 
    We require that there is a global sign $\sigma=\pm 1$ such that the orientations of $\mathcal{M}_u$ and $\mathcal{M}_v$ differ by the product $\sigma\cdot\nu(u_t)$ where $\nu(u_t)$ is the local crossing sign of $u_t$, 
    i.e., $\nu(u_t)$ is the crossing sign between boundary arcs in the hyperbolic case and between $L$ and the curve
    in the elliptic case.  
\end{enumerate}
\end{theorem}
\begin{proof}

\begin{enumerate}
\item Holds trivially at all $J\in\mathcal{J}$. 
\item Define $\mathcal{J}^{\circ}$ as the subset of generic $J\in\mathcal{J}$ where every bare solution is transversely cut out, is embedded and in particular has boundary an embedded link, is transverse to the 4-chain $C$, and has interior disjoint from $L$. Then $\mathcal{J}^{\circ}$ is open and dense by Lemmas 
\ref{l:genframing}, \ref{l:genemb}, and \ref{l:gen4chain}.  By transversality our solution space is an 
oriented 0-manifold. 
Evidently Property \eqref{prop:isolated} is satisfied.  Property \eqref{prop:embedded} holds by request, and Property \eqref{prop:trivialcobordism} by transversality.

We turn to property \eqref{prop:codimzerocompact}.  
It is {\em not} simply a consequence of Gromov compactness, since we have taken 
$\M(J)$ the moduli space of bare curves, and not all curves.  That is, by Gromov compactness we know that 
a sequence of bare curves must converge to some stable map; it remains to exclude the possibility 
that this limit is a non-bare curve.  But by Theorem \ref{lem:compactness}, such a limit would have 
underlying bare curve a singular $J$-holomorphic curve, which contradicts  property \eqref{prop:embedded}.  

\item Lemmas \ref{l:genframing}, \ref{l:genemb},  and \ref{l:gen4chain} classify the degenerations possible in a 
generic one parameter family.  Properties (3b) - (3e) follow immediately.

We turn to Property \eqref{prop:codimonecompact}.  It is {\em not} simply a consequence of Gromov compactness, since we have taken $\M(J)$ the moduli space of bare curves, and not all curves.  By Gromov compactness, a 1-parameter family will have a limit, which
may however be a non-bare curve.  However, from Theorem \ref{lem:compactness}, a limiting non-bare curve will 
either have a non-immersed point or a triple point in the image of its bare component.  By coherence (property \eqref{prop:coherence}), this means that at this $J$, we have a map from a smooth domain with either a triple point or a non-immersed point.  This contradicts property \eqref{prop:crossingtypes}. 

\item This follows from Theorem \ref{t : wall crossing gluing} and Lemma \ref{l : same orientations elliptic and hyperbolic}. 
\end{enumerate}
\end{proof}

We are now in a position to define the invariant.

\begin{dfn}\label{eq:invtdef}
Let $X$ be a 3-dimensional Calabi-Yau manifold and let $L=L_0\cup L_1\cup\dots\cup L_m$ be as in Section \ref{sec : somewhere inj} with Maslov index zero and a brane structure in the sense of Definition \ref{brane}.  Fix a basic class $d \in H_2(X, L)$. Let $\Sk(L)$ be the skein of $L$, in variables $(a,z)$, $a=(a_{1},\dots,a_{k})$.  
For $J \in \Jj^{\circ}(X, L)$, we define 
\begin{equation*}  
Z_{X,L,d;J} \ = \
\sum_{(u,S) \in \M(X,L, d, J)} \!\!  w(u)   \cdot z^{-\chi(S)} \cdot a^{u \OpenHopf L} \cdot
\langle\partial u\rangle \ \in \ \Sk(L) 
\end{equation*} 
\end{dfn}

The sum is over a  finite set of points for each coefficient of $z$ by Properties 
\eqref{prop:codimzerocompact}, \eqref{prop:isolated}, with the second of these providing the function $w$. 
By Property \eqref{prop:embedded}, $\partial u$ is a framed link in $L$, so we may regard it as an element
of the skein.  Note that in any given homology class, the Euler characteristic of (even a disconnected) representative
is bounded above, so the invariant is a Laurent series in $z$. 

We emphasize that we count {\em only bare maps} from possibly disconnected curves.  That is, the 
image of each irreducible component has nonzero symplectic area.

\begin{theorem} \label{thm:invariance}
Under the hypotheses of Theorem \ref{thm:somewhereinjective},
$Z_{X,L,d;J}$
is independent of $J \in \Jj^\circ(X,L)$. 
\end{theorem}

\begin{proof} 
We will use only the properties enumerated in Theorem \ref{thm:somewhereinjective}.

It follows from Properties \eqref{prop:trivialcobordism} and  \eqref{prop:boundarycobordism} that
$Z_{X,L,d;\lambda}$
is locally constant in $\Jj^\circ(X,L)$.  Consider now any two points $J_0, J_1 \in \Jj^\circ$, 
and a path $\lambda_t$ connecting them satisfying Property \eqref{prop:codimone}.  
We want to study the $t$ with $J_t \notin \Jj^\circ$, in order to show that also
the count 
$Z_{X,L,d;J}$
does not change when passing these. 

Note that to prove the desired 
constancy of 
$Z_{X,L,d;J}$, it is enough to work order by order in $z$. 
Bounding the order of $z$ bounds the topological type, thus the set of walls $t_0$ we must consider is finite by 
Property \eqref{prop:onecrossingatatime}.  Fix one such $t_0$. 

As $t \to t_0$, by \eqref{prop:trivialcobordism}, the curve counts are locally constant, and the moduli space
is just a disjoint union of intervals.  By compactness \eqref{prop:codimonecompact}, each such interval
has a limit at $t_0$.   Consider all components for which the limiting curve remains smooth, 
embedded, transverse to $L$; i.e., \eqref{prop:embedded} remains true at the limit.  Then, even if 
\eqref{prop:trivialcobordism} fails, nevertheless
\eqref{prop:boundarycobordism} suffices to ensure that these components of moduli have the same contribution
on both sides of $t_0$. 

Let us consider the remaining components, i.e., those for which the limiting curve fails \eqref{prop:embedded}. 
Taking the normalization gives a map from a smooth domain, which is strictly greater Euler characteristic, so 
among those we are considering at the current order in $z$.   Recall there is a unique such map, as per
\eqref{prop:onecrossingatatime}, and that it must take one of the forms classified by Property \eqref{prop:crossingtypes}. 
Note that since \eqref{prop:embedded} fails, we have demanded that \eqref{prop:trivialcobordism} does not; 
i.e. the component of moduli containing this map from a smooth domain locally projects to $[0,1]$ by an isomorphism.

For crossings of `framing change' type, by applying Lemmas \ref{l:throughgamma} and \ref{l:tangency}, and 
comparing to the framing change skein relation, we see that the term 
$a^{u \OpenHopf L} \cdot \langle\partial u\rangle$ is itself invariant. 

We now turn to hyperbolic and elliptic crossings.
First, the hyperbolic case.  
Let $u$ be the map with smooth domain and nodal image, and 
$u_\times$ the corresponding embedding of a nodal curve. 
We will write $u_{\pm \epsilon}$ for the curves immediately
before and after $u$ in its family, and
$u_{\smoothing}$ for the deformation of $u_\times$ which
appears on one side (either before or after).  We write 
$w(u_\bullet)$ for the degree of the 0-chain on 
the corresponding moduli space. 

According to Property \eqref{prop:onecrossingatatime}, curves unrelated to $u$
undergo no critical moments.  The 4-chain intersection and 
$H_2(X, L)$ of all the $u_\bullet$ are the same.  Thus 
the total change in 
$Z_{X,L,d;\lambda}$
is a multiple of 
$$
w(u_\epsilon) z^{-\chi(u_\epsilon)}\langle \partial u_\epsilon \rangle \
-  \ w(u_{-\epsilon})z^{-\chi(u_{-\epsilon})} \langle \partial u_{-\epsilon} \rangle 
\pm \ w(u_\smoothing)z^{-\chi(u_\smoothing)}  \langle \partial u_{\smoothing} \rangle.
$$
By Properties \eqref{prop:boundarycobordism} and \eqref{prop:trivialcobordism},
$w(u_\epsilon) = w(u_{-\epsilon})$.  By Property \eqref{prop:standard} 
this quantity is also equal to $w(u_\smoothing)$.  
In addition, note that
$\chi(u_\epsilon) = \chi(u_{-\epsilon}) = \chi(u_\smoothing) + 1$. 
Thus the above discrepancy is a multiple of
\begin{equation} \label{hypskein}
\langle \partial u_\epsilon \rangle  \ - \ \langle \partial u_{-\epsilon} \rangle \ \pm \ 
z \langle \partial u_{\smoothing} \rangle,
\end{equation}
Recall that Property \eqref{prop:standard} asks that 
the crossing moment is locally given by the model 
in Section \ref{sec:models}, in the sense of Definition
\ref{standard hyperbolic degeneration}.  
As explained in Section \ref{ssec: hyp boundary}, 
the images of the boundary under 
$u_{\pm \epsilon}$ differ precisely by a crossing change
$\overcrossing  \leftrightarrow  \undercrossing$, whereas
the boundary of $u_\smoothing$ is correspondingly the 
$\smoothing$.  

Now let us consider the elliptic crossing.  We use similar notations
as above, save writing $u_\bigcirc$ instead of $u_\smoothing$. As before, Property \eqref{prop:standard} asks that the crossing moment is locally given by the model 
in Section \ref{sec:models}, in the sense of Definition
\ref{standard elliptic degeneration} and the same considerations apply, with one exception: 
$u_{\pm \epsilon}$ have one more or one less intersection 
with the 4-chain than $u_\circ$ does.  Thus now the relation is
$$ 
a^{\pm 1} \langle \partial u_\epsilon \rangle  \ - \ a^{\mp 1} \langle \partial u_{-\epsilon} \rangle \ \pm \ z \langle \partial u_{\bigcirc} \rangle.
$$
As explained in Section \ref{ssec: ell boundary},
the difference between the boundaries is that 
$\partial u_\bigcirc$ has an extra unknot as compared to
$\partial u_{\pm \epsilon}$, the latter two being isotopic to each other. 
Thus, the above expression 
is a multiple of 
\begin{equation} \label{ellskein}
a \ - \ a^{-1} \ \pm \ z \langle \bigcirc \rangle.
\end{equation}

Finally, by Property \eqref{prop:standard} the signs in \eqref{hypskein} and \eqref{ellskein} agree.
Thus, changing $z$ to $-z$ in our original conventions if necessary, 
we may ensure that the sign in formula \eqref{hypskein} is $+$ and in formula \eqref{ellskein} is $-$. 
With this choice, formulas \eqref{hypskein} and \eqref{ellskein} are simply the skein relations, hence zero in the skein. 
\end{proof}

\section{SFT and the conifold transition} \label{sec:conifold} 

In this section we compare curve counts in $T^* S^3$, in $T^*S^3 \setminus S^3$, and in the 
resolved conifold $X$.  The key ingredient is SFT  compactness and the stretching arguments it makes possible.   

\subsection{The conifold transition}

Let us review the conifold transition.  Consider the locus $X_0 = \{w^2 + x^2 + y^2 + z^2 = 0\}$ in $\C^4$.  On the 
one hard, this locus is the image of the total space of the bundle $X = \Oo(-1) \oplus \Oo(-1)$ over $\C\P^1$, under a map which collapses 
the $\C\P^1$ and is elsewhere an embedding.   On the other hand, there is the deformation $X_\epsilon = \{w^2 + x^2 + y^2 + z^2 = \epsilon\}$, which
is symplectomorphic to $T^* S^3$.  The conifold transition means we replace one with the other. 

From a symplectic point of view, $X_\epsilon$ provides an exact symplectic filling of the contact cosphere bundle $S^* S^3$:  outside of a compact set, $X_{\epsilon}$ agrees with the positive symplectization $[0,\infty)\times ST^{\ast}S^{3}$ with its standard symplectic form $d(e^{t}\,pdq)$, where $pdq$ is the action form restricted to $ST^{\ast}S^{3}$ and $t$ a coordinate on $[0,\infty)$. The resolved conifold is not quite a symplectic filling of $ST^{\ast} S^{3}$. However, far from the central $\C\P^{1}$, also $\tilde X$ looks topologically like $[0,\infty)\times ST^{\ast}S^{3}$ and the symplectic form is deformation equivalent to $d(e^{t}\,pdq)+\omega$, where $\omega$ is independent of $t$. For large $t$ this can be viewed as a small perturbation that vanishes in the limit $t\to\infty$. 

We say that $X$ is an asymptotic symplectic filling of $ST^{\ast}S^{3}$.  More generally, a symplectic manifold with positive end which is analogously asymptotic to a symplectization of a contact manifold is said to be asymptotically convex at infinity. In this case,  we say that a Lagrangian $L\subset X$ is asymptotically Legendrian at infinity if it is, near infinity,  obtained by fixed non-exact deformation of the Lagrangian cylinder $\R\times\Lambda$.  The fixed amount by which we deform is exponentially small with respect to the symplectic form $d(e^t pdq)$ as $t\to\infty$. 
Note then that, in sufficiently large disk bundles there is an arbitrarily small change of the $\R$-invariant symplectization almost complex structure (for instance, given by conjugation by a map taking a Lagrangian to its shift) making the original $\R$-invariant curves holomorphic with the non-exact boundary condition, see also \cite{Koshkin} for basic results about holomorphic curves in this setting.

\subsection{SFT stretching}\label{ssec : SFT stretch}
In this section we review SFT-stretching in the context we will require it. A complete account can be found in \cite{BEHWZ}.

Consider a symplectic manifold $(X,\omega)$ with a Lagrangian submanifold $L$.  Let $Y\subset X$ be a closed co-oriented codimension $1$ contact hyper-surface. This means that there is a conformally symplectic vector field $Z$ (i.e., $L_Z\omega=\omega$) defined in a neighborhood of $Y$ such that $Z$ is everywhere transverse to $Y$. Then $\alpha=\iota_Z\omega$ is a contact form on $Y$. If $Y\times(-\delta,\delta)$ is the neighborhood of $Y$ obtained by flowing $t\in(-\delta,\delta)$ units along $Z$ with initial condition in $Y$ then the symplectic form has the form 
\[
\omega=d(e^t\alpha)=e^t(dt\wedge\alpha + d\alpha)
\]
in $Y\times(-\delta,\delta)$. We assume that the intersection $\Lambda=L\cap Y$ is a Legendrian submanifold (i.e., $\alpha|_\Lambda=0$) and that $L$ is invariant under the flow of $Z$ in some flow neighborhood of $Y$. 

We assume that the conact form $\alpha$ on $Y$ is Morse (or Morse-Bott). This means that its Reeb orbits (flow loops of $R_\alpha$) and Reeb chords (flow lines of $R_\alpha$ connecting $\Lambda$ to itself) are isolated and transversely cut out (or come in smooth families that are Bott-transverse). This condition holds after small perturbation of $Y$ or $Z$. We recall that the action of a Reeb orbit or chord $\gamma$ is $\int_\gamma\alpha$.

Assume that $Y$ separates $X$. Then $X\setminus Y= X^{\circ,+}\cup X^{\circ,-}$ where $X^{\circ,+}$ ($X^{\circ,-}$) is the component where the vector field $Z$ point inward (outward) along the boundary. Let $X^\pm$ be the closure of $X^{\circ,\pm}$. Fix $s>0$ and $0<\epsilon<\delta$ and consider the three symplectic manifolds with symplectic forms as indicated:
\begin{align*} 
X^{+}_{s} &= X^+\cup_Y  ((-s-\epsilon,0]\times Y);\quad \omega \text{ on } X^+,\; d(e^t\alpha) \text{ on } (-s,0]\times Y,\\
X^{0}_{s} &= (-s-\epsilon,s+\epsilon)\times Y;\quad d(e^t\alpha) \text{ on } (-s,s)\times Y, \\
X^{-}_{s} &= X^-\cup_{Y} ([0,s+\epsilon)\times Y);\quad \omega \text{ on } X^-,\; d(e^t\alpha) \text{ on } [0,s)\times Y.
\end{align*}

Define 
\[ 
X_{s}= X^{-}_{s}\cup_{\tau_{2s}} X^{+}_{s},   
\]
where
\[
\tau_{2s}\colon (s,s+\epsilon)\times Y\to (-s-\epsilon,s)\times Y,\quad 
\tau_{2s}(\sigma,y)=(\sigma-(2s+\epsilon),y).
\]
Note that $\tau_{2s}$ is the conformally symplectic map that intertwines the symplectic structure 
$d (e^{\sigma}\alpha)$ with $d (e^{\sigma-(2s+\epsilon)}\alpha)$. 

Let $J^\pm$ be the following almost complex structure on $X_s^\pm$. Consider the cylindrical pieces in $X_s^{\pm}$ as subsets of $\R\times Y$ and take $J^\pm$ on these cylindrical pieces to be invariant under translation in the $\R$-direction, such that $J^\pm(\mathrm{ker}(\alpha)\to \mathrm{ker}(\alpha)$ and $J^\pm$ both agree with some complex structure $J_{\mathrm{ker}(\alpha)}$ on $\mathrm{ker}(\alpha)$ that is compatible with $d\alpha$, and
$J\partial_\sigma=R_\alpha$, where $\alpha$ is the contact form on $Y$ and $R_\alpha$ its Reeb vector field (uniqely defined by $\iota_{R_\alpha}d\alpha=0$, $\alpha(R)=1$). On $X^\pm$, $J$ is a fixed (independent of $s$) almost complex structure compatible with $\omega$. Let $J_s'$ be the almost complex structure on $X_s$ that agrees with $J^\pm$ on $X^\pm_s$.  

We identify $X_0$ with $X$ and define diffeomorphisms $\phi_{s}\colon X_{s}\to X$ by shrinking $X^{0}_{s}$ to $X^{0}_{0}$. We transport $J'_s$ to a family $J_{s}$ of almost complex structures on $X$, all compatible with the symplectic structure $\omega$: $J_{s}=d\phi_{s}\circ J'_{s}\circ d\phi_{s}^{-1}$.

We state the specialization of the SFT-compactness theorem to our situation.    
\begin{theorem}\label{thm:SFTlimit}\cite[Theorem 10.3, Section 11.3]{BEHWZ}
Let $u_{s}$, $s\to\infty$ be a sequence of $J_{s}$-holomorphic curves with boundary on $L$ and of uniformly bounded area. Then there exists a subsequence $u_{s'}$ that converges to a holomorphic building in $X^{-}_{\infty}$, $X^{0}_{\infty}$, and $X^{+}_{\infty}$, where the levels are joined at Reeb orbits and Reeb chords of $\Lambda$ in $Y$. Furthermore, the total action of the Reeb orbits and chords at the negative ends of the curves in $X_{\infty}^{+}$ is bounded in terms of the symplectic area of the curves. \qed 
\end{theorem}

We will use SFT-stretching Theorem \ref{thm:SFTlimit} to deform complex structures near Lagrangian submanifolds that are either compact or asymptotically cylindrical. 

\subsubsection{SFT-stretching near compact Lagrangians}\label{ssec:SFTlimitcompactL}
If $L\subset X$ is a compact Lagrangian submanifold then $L$ has a neighborhood symplectomorphic to a $\delta$-neighborhood of the $0$-section in its cotangent bundle $T^\ast L$ with the standard symplectic structure $d(pdq)$. Consider a Riemannian metric on $L$. The $\epsilon$-cosphere bundle $S_\epsilon T^\ast L=\{p^2=\epsilon^2\}$ is then a contact hypersurface with conformally symplectic vector field $p\cdot \partial_p$ and with Reeb flow given by the natural lift of the geodesic flow on $L$. Here the action of a Reeb orbit is equal to ($\epsilon$ times) its length and the Conley-Zehender index of a Reeb orbit equals the Morse index of the underlying geodesic. 

The three symplectic manifolds that results form stretching around $L$ are $X^+_\infty= X\setminus L$, with a negative end $(-\infty,0]\times ST^\ast L$ near $L$, $X^0_\infty=\R\times ST^\ast L$, and $X^-_\infty=T^\ast L$.  


\subsubsection{SFT-stretching near cylindrical Lagrangians}\label{ssec:SFTlimitnoncompactL}
Let $X$ be an asymptotically cylindrical symplectic manifold with asymptotic cylindrical end $[0,\infty)\times Y$ and let $L\subset X$ be an asymptotically cylindrical Lagrangian with cylindrical end $[0,\infty)\times\Lambda$.
Then $L=\overline{L}\,\cup_\Lambda\, ([0,\infty)\times\Lambda)$, where $\overline{L}$ is a compact manifold with boundary $\partial\overline{L}=\Lambda$. We consider a Riemannian metric on $L$ of the form
\[
g=
\begin{cases}
\overline{g} &\text{ on }\overline{L},\\
dt^2\oplus \phi(t) g_\Lambda &\text{ on } [0,\infty)\times Y,
\end{cases}
\]
where $g_\Lambda$ is a metric on $\Lambda$ and $\phi\colon [0,\infty)\to [1,2]$ is an increasing function. Then closed geodesics of $g$ all lie in the interior of $\overline{L}$ and the distance from $\Lambda\times \{t_0\}$ to $\Lambda\times\{t_1\}$ is bounded below by $|t_1-t_0|$.

Consider a function $f_T\colon L\to \R$ such that $f_T=0$ on $\overline{L}$ and $f_T(t,\lambda)=\beta(t-T)$, where $\beta\colon \R\to[0,1]$ is a cut off function equal to $0$ on $(-\infty,0]$ and equal to $1$ on $[1,\infty)$. Consider again a cotangent neighborhood $T^\ast L$ of $L$. Take
\[
Y=\left\{(q,p)\in T^\ast L\colon \tfrac12 p^2 + f_T(q)=\epsilon^2\right\},
\]
and note that if $H=-pdq(\nabla f_T)$ then $Z=p\cdot\partial_p + X_H$ is a conformally symplectic vector field transverse to $Y$ along $Y$. 

Note that $Y$ is naturally subdivided into two pieces: an interior piece of the form $ST^\ast (\overline{L}\cup_\Lambda ([0,T]\times\Lambda))$ and a cap diffeomorphic to $DT^\ast L|_{\Lambda}$, the disk cotangent bundle of $L$ restricted to $\Lambda$. Reeb orbits and chords are of two kinds: Reeb orbits corresponding to closed geodesics in $\overline{L}$, and Reeb orbits or chords that restrict to Reeb flow lines over geodesics that connect the boundary of $\overline{L}\cup_\Lambda ([0,T]\times\Lambda)$ to itself. The action of the latter type of geodesics are $>\frac12 T$. 

Consider now stretching along $Y$ for curves of bounded area. If we take $T$ sufficiently large we find that the limiting holomorphic buildings must have Reeb asymptotics only at Reeb orbits in $ST^\ast \overline{L}$.

\begin{example}\label{ex : stretch around solid torus}
We will apply this for $L$ of topology $S^{1}\times\R^{2}$. Here we take a metric with a unique geodesic loop that goes once around the generator $S^1\times\{0\}$ of $\pi_{1}(S^{1}\times\R^{2})$.  All other geodesic loops are multiples of this basic loop, and all closed Reeb orbits have Conley-Zehnder index $0$.
\end{example}

\subsection{Relating $T^*S^3$ with the symplectization of $ST^* S^3$}
We generalize Theorem \ref{thm:invariance} to exact symplectic manifolds with negative end the symplectization of the unit cotangent bundle of the three sphere. 

We use notation as in Section \ref{ssec : SFT stretch}. Let $X_\infty^+$ denote an exact symplectic manifold with negative end of the form $(-\infty,0]\times ST^\ast S^3$, where the symplectic form is $d(e^t\alpha)$ for the contact form $\alpha$ induced by the standard round metric on $S^3$. Then Reeb orbits of $\alpha$ are in natural 1-1 correspondence with geodesics and the Conley-Zehender index of a Reeb orbit equals the Morse index of the corresponding geodesic. This means that $\alpha$ has Bott degenerate Reeb flow with Bott manifolds the Grassmannian of oriented 2-planes in $\R^4$, i.e., $S^2\times S^2$, of index $\ge 2$.

Let $L=L_1\cup\dots\cup L_m$ be an asymptotically cylindrical Lagrangian, where $L_j$ topologically is $S^1\times\R^2$, where the negative end of each $L_j$ is empty and where the positive end is a Legendrian $S^1\times S^{1}$.

\begin{lemma}\label{l:curveswnegends} 
Let $X^+_\infty$ and $L$ be as above and assume that $\mathcal{J}$ is a space of almost complex structures that are compatible with the symplectic form and $\R$-invariant in $(-\infty,-T]\times Y$, $T>0$, of the negative end and such that every curve in a basic homology class is somewhere injective. Then Theorem \ref{thm:somewhereinjective} holds for $\mathcal{J}$; skein valued curve counts for $(X_{\infty}^{+},L)$ can be defined exactly as in Definition \ref{eq:invtdef}, and Theorem \ref{thm:invariance} holds.  
\end{lemma}

\begin{proof}
By assumption $J$-holomorphic curves are somewhere injective and Lemma \ref{l:basicsurjectivity} gives transversality for the Cauchy-Riemann operator $\bar\partial_J$ by varying $J$ in $\mathcal{J}$. Then, by SFT compactness, Theorem \ref{thm:SFTlimit}, when there is a negative end, moduli spaces have additional boundary corresponding to holomorphic buildings with levels joined at Reeb orbits. In the present case, a level of such a building in $(X_+^\infty,L)$ is either closed or has negative interior punctures at Reeb orbits. However, since the Reeb orbits have minimal index $2$ the dimension of such a curve is at most $-2$. Since $\mathcal{J}$ gives transversality we find that curves with negative asymptotics do not appear in 0- and 1-parameter families in general position. It follows that the degeneracies in 1-parameter families are exactly as in the case without negative end which implies the result.
\end{proof}

\begin{remark}\label{r : acs good enough}
In Lemma \ref{l:curveswnegends}, if $L=L_1$ of topology $S^\times\R^2$, then Lemma \ref{l : somwhere injectivity exact} shows that space of all almost complex structures cylindrical in $(-\infty,0]\times ST^\ast S^3$ ensures the required somewhere injectivity. Also, if $X^+_\infty=\R\times ST^\ast S^3$ and $L=L_1\cup\dots\cup L_m$ then almost complex structures for which the assumptions of Lemma \ref{l : somewhere injective non-exact} hold and which are cylindrical in $(-\infty,0]\times ST^\ast S^3$ ensures the required somewhere injectivity.
\end{remark}

Let $L=L_1\cup\dots\cup L_m\subset T^\ast S^3$ be an asymptotically cylindrical Lagrangian, where $L_j$ topologically is $S^1\times\R^2$ such that $L\cap S^3=\emptyset$. 
Let  $\beta\in H_2(T^\ast S^3, L\cup S^3)$ be a basic homology class, see Definition \ref{def : basic class}. The inclusion $T^\ast S^3\setminus S^3\to T^\ast S^3$ induces a map $H_2(T^\ast S^3\setminus S^3,L)\to H_2(T^\ast S^3, L\cup S^3)$ with kernel $H_2(T^\ast S^3\setminus S^3)\approx Z$. Note that any class in $H_2(T^\ast S^3\setminus S^3,L)$ that maps to $\beta$ is a basic class for $L\subset T^\ast S^3\setminus S^3$. Picking a $4$-chain for $S^3\subset T^\ast S^3$ gives a splitting $H_2(ST^\ast S^3\setminus S^3,L)=H_2(ST^\ast S^3,L)\oplus \Z\cdot [S^2]$, $[S^2]\in H_2(ST^\ast S^3\setminus S^3)$ is the class of the fiber sphere. We will think of $T^\ast S^3\setminus S^3$ symplectically as the symplectization $\R\times ST^\ast S^3$. 

If $L$ has one component, let $\mathcal{J}_0$ be the space of complex structures on $T^\ast S^3$ and if $L$ has more thean one component then let $\mathcal{J}_0$ is the space of almost complex structures such that the assumptions of Lemma \ref{l : somewhere injective non-exact} holds for $L$. Then if we apply SFT stretching in around the boundary of a neighborhood $N(r)$ of $S^3$ we get almost complex structures in $\R\times ST^\ast S^3$ that satifies the somewhere injectivity hypothesis of Lemma \ref{l:curveswnegends}. We use such almost complex structures for the curve counts in the following result. 

\begin{theorem} \label{thm:deleteS3}  
Let $L$ be a Lagrangian brane (see Definition \ref{brane}) in $T^* S^3 \setminus S^3$ and let
$\beta \in H_2(T^*S^3, L\cup S^3)$ be a basic homology class as above. As explained above, any class in $H_2(\R\times ST^\ast S^3,L)$ that maps to $\beta$ is also basic. We have the following:   
\begin{equation}
	Z_{T^*S^3, S^3 \cup L,\beta}(a, a_L, z) =  \sum_{k} Z_{\R\times ST^\ast S^3,  L, \beta+k[S^2]}(a_L, z)Q^k\big|_{Q = a^2}\otimes[\varnothing_{S^{3}}]
	\in \Sk(S^3)
	\end{equation}
\end{theorem}

\begin{proof}
	It suffices to establish this equality order by order in $z$ and thus we may fix area and Euler characteristic bounds. Write $N(r)$ for a radius $r$-neighborhood of $S^{3}\subset T^{\ast}S^{3}$. 
	
	Consider the curves contributing to $Z_{T^*S^3, S^3 \cup L}$. Let $\epsilon>0$ be sufficiently small so that $L$ lies outside $N(10\epsilon)$. We apply SFT-stretching near $\partial N(5\epsilon)$. By Lemma \ref{l:curveswnegends}  for sufficiently large stretching parameter $\rho$ we find an almost complex structure $J_{\rho}$ such that all curves with boundary on $L$ lies outside $N(5\epsilon)$. Furthermore, all closed $J_{\rho}$-holomorphic curves have zero symplectic area and are therefore constant.  
	
	That is, we may deform the complex structure for $T^*S^3$ so that all curves leave a neighborhood of the $S^3$. 
	Now the curves contributing to the invariant are literally the same for $(T^*S^3, S^3 \cup L)$ and $(\R\times ST^* S^3, L)$ (see Lemma \ref{l:curveswnegends} for the latter count); it remains only to examine how we count them.  In the former case we must account
	for the intersection with the 4-chain; in the latter case, for the class in $H_2(T^* S^3 \setminus S^3)$.  
	These match under the substitution $Q = a^2$, since 
	generating homology class in $H_2(T^* S^3 \setminus S^3)$ is dual to half of the 4-chain, 
	viewed as an element of $H_4^{BM}(T^* S^3 \setminus S^3)$.   
\end{proof}

\begin{figure}
	\centering
	\includegraphics[width=.5\linewidth]{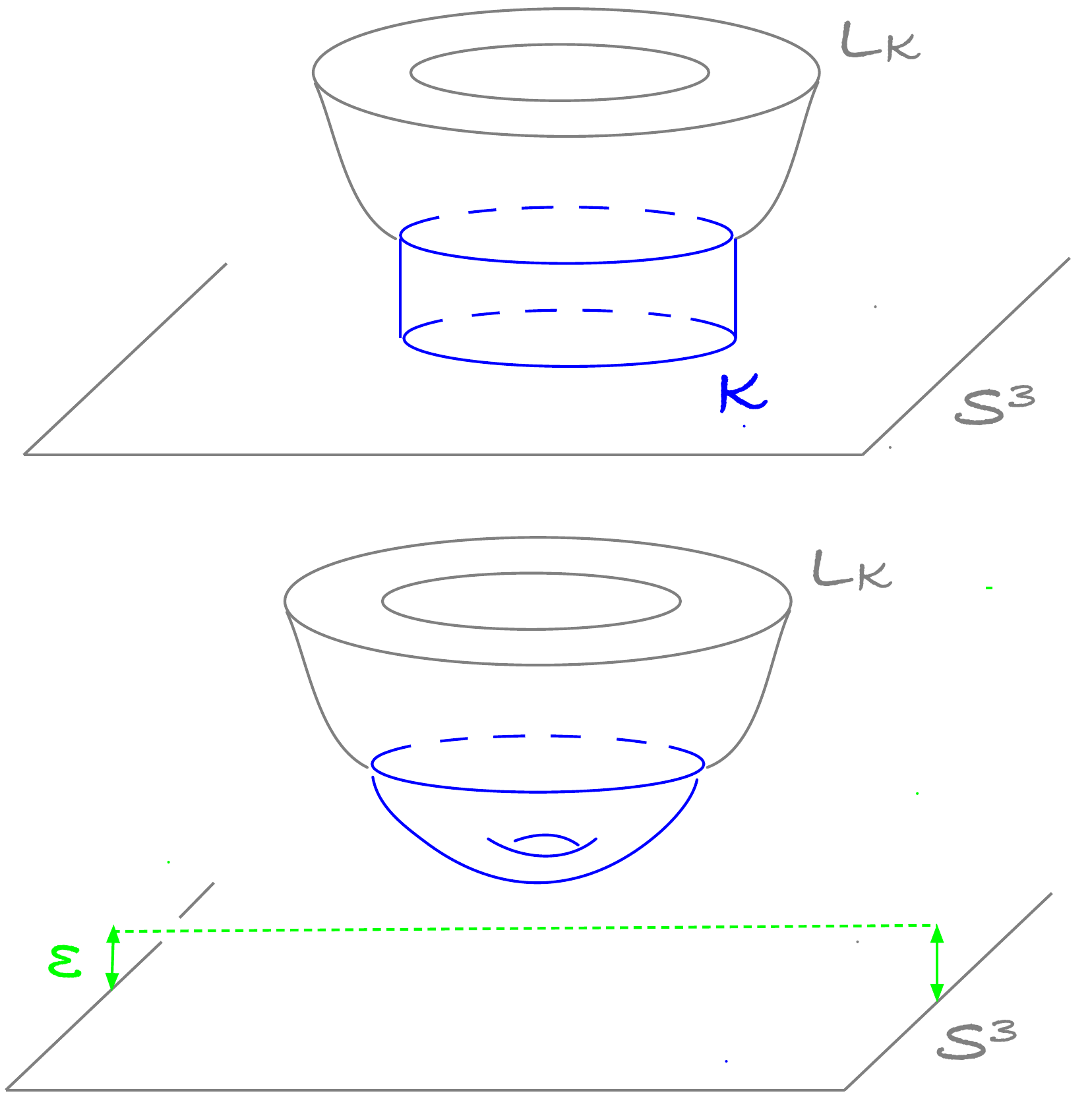}
	\caption{Top picture shows the unique annulus with boundary on $L_{K}$ and along $K\subset S^{3}$, before SFT-stretching. As we SFT-stretch, the boundary on $S^{3}$ undergoes skein moves until, as illustrated in the bottom picture, all curves left an $\epsilon$-neighborhood of $S^{3}$.}
\end{figure}

\subsection{Relating the symplectization of $ST^\ast S^3$ with the resolved conifold}
Given a Lagrangian conormal brane $L_K\subset T^* S^3 \setminus S^3\approx\R\times ST^\ast S^3$, we may consider it as a brane also in the resolved conifold $X$ as follows. The resolved conifold is asymptotic to the symplectization of $ST^\ast S^3$ so that, outside some neighborhood of the central $\C\P^1$, $X$ is symplectimorphic to $[0,\infty)\times ST^\ast S^3$ with symplectic form $d (e^t\alpha)+\omega_0$ for the standard contact form on $S^3$, where $\omega_0$ is bounded with size corresponding to the area of $\C\P^1$. Then, if the area of $\C\P^1$ is sufficiently small then as in Lemma \ref{l : omega0 small}, $L_K\subset [T,\infty)\times ST^\ast S^3$ admits an isotopy $L_K^s$ inside a neighborhood of the $0$-section in $T^\ast L_K$ such that $L_K=L_K^0$ and that $L_K^s$ is Lagrangian with respect to $d(e^t\alpha)+(1-s)\omega_0$. When we consider $L_K$ as a Lagrangian in $X$ we mean $L_K^1$. 

Note that the 4-chain of $L_K\subset \R\times ST^\ast S^3$ gives a $4$-chain for $L_K\subset X$ which is disjoint from a neighborhood $N(\C \P^1)$ of the central $\C \P^1$ in $X$.  

We next consider curve counts for $L\subset X$ which has the properties of $L_K$ above. Were we to attempt to count all possibly disconnected holomorphic curves in basic homology classes for $L \subset X$, there would be no way to exclude disconnected components giving multiple covers of the central $\C\P^1$ and the general bare curve count requires adequate perturbations as developed in \cite{bare}, see Section \ref{sec: beyond}. Here we bypass these difficulties by just restricting to counting curves of which each component has boundary; we denote the resulting count $Z'$ and show next that it is well defined and invariant.

Let $L=L_1\cup\dots\cup L_m\subset X$ and each $L_j$ is topologically $S^1\times\R^2$. Assume that $L\subset [0,\infty)\times ST^\ast S^3$ and that there exists a neighborhood $N(\C\P^1)$ disjoint from the $4$-chain of $L$. Let $\mathcal{J}'$ denote the space of almost complex structures on $X$ that meets the conditions of Lemma \ref{l : somewhere injective non-exact} and which furthermore are standard in some neighborhood of $\C\P^1$. 

\begin{lemma}\label{reducedopen}
Let $(X,L)$ and $\Jj'$ be as above. For $\C\P^1$ of sufficently small area, $\mathcal{J}'$ is non-empty and if $\mathcal{M}'(J)$ is the moduli space of curves in a basic homology class of $(X,L)$, see \eqref{eq : basic non-exact}, for which every connected component has nonempty boundary, then Theorem \ref{thm:somewhereinjective} holds with $\mathcal{J}$ replaced by $\mathcal{J}'$, and $\mathcal{M}$ replaced by  $\mathcal{M}'$, and Theorem \ref{thm:invariance} holds for the corresponding
\begin{equation}\label{eq:fakereduced}
Z_{X,L;J_t}' = 1 + 
\sum_{(u,\Sigma)\in\M'(X,L)} 
w(u) \cdot
z^{-\chi(u)} \cdot  Q^{u_*[\Sigma]} \cdot  a^{u \OpenHopf L} \cdot 
\langle\partial u\rangle \ \in \  \Sk(L)[[Q]],
\end{equation}
which means that $Z_{X,L;J_t}'$ is independent of $t\in[0,1]$.
\end{lemma}

\begin{proof}
To see that $\mathcal{J}'$ is non-empty, note that if the area of $\C\P^1$ is sufficiently small then then an $\R$-invariant almost complex structure in $[0,\infty)\times ST^\ast S^3$ compatible with $d(e^t\alpha)$ is also compatible with $\omega=d(e^t\alpha)+\omega_0$. It is also straightforward to check that one can interpolate between the standard almost complex structure near $\C\P^1$ and an $\R$-invariant complex structure in $[0,\infty)\times ST^\ast ST^\ast S^3$ over a region $[-T,0]\times ST^\ast S^3$. This shows that $\mathcal{J}'$ is non-empty.

We next show that any closed curve is contained in a small neighborhood of $\C\P^1$. Note that in the region $[0,\infty)\times ST^\ast S^3$ the complex structure is $\R$-invariant. Here the $t$-coordinate is pluri-subharmonic and cannot have a local maximum. It follows that closed curves cannot enter this region and we get the desired neighborhood $N(\C\P^1)$. 

Lemma \ref{l : somewhere injective non-exact} and Lemma \ref{l:basicsurjectivity} then shows that $\mathcal{J}'$ gives transversality for curves in basic homology classes for which all components have non-empty boundary. The result then follows by a word by word repetition of the proof of Theorem \ref{thm:invariance}, replacing $Z_{X,L,d,J_t}$ there by $Z_{X,L,J_t}'$. 
\end{proof}

\begin{theorem} \label{thm:conifold} 
For any basic homology class $\beta$:
\begin{equation}\label{eq:conifoldcurves}
	Z_{\R\times ST^*S^3,  L,\beta}(Q, a_L, z) = Z'_{X,L,\beta}(Q, a_L, z)
	\end{equation}
\end{theorem}

\begin{proof}
We use notation as above. Take a complex structures on $X$ that is compatible with $d(e^t\alpha)$ in $[0,\infty)\times ST^\ast S^3$ and use an almost complex structure on $\R\times ST^\ast S^3$ that agree on the correspnding region. Apply SFT-stretching in both spaces in $[0,1]\times ST^\ast S^3$.
If the limit has some level with negative end then the outer piece is asymptotic to a collection of Reeb orbits and hence has dimenson $\le -2$. It follows that for suffiently large stretching all curves in both spaces lie in $[0,\infty)\times ST^\ast S^3$. These curves give all the contributions to both $Z_{\R\times ST^*S^3, L,\beta}(Q, a_L, z)$ and $Z'_{X,L,\beta}(Q, a_L, z)$.   
\end{proof}

\section{Enumerative meaning of the HOMFLYPT polynomial} \label{sec:proofthmbasic}

Let $K=K_{1}\cup\dots\cup K_{m} \subset S^3$ be a link, and let $L_{K}$ its conormal Lagrangian, shifted off the $0$-section. When $m>1$ we assume that the shift is sufficiently small so that Lemma \ref{l : nested nbhds for links} holds.

\subsection{A convenient choice of 4-chain}\label{ssec : nice 4 chain}
Since the conormal is disjoint from a cotangent fiber, it is null-homologous in $T^{\ast}S^{3}\setminus S^{3}$, hence bounds a 4-chain there. Changing the 4-chain affects the invariants by monomial factors and monomial changes of variable.

We will now make a particular choice of 4-chain for $S^3$ in order to avoid framing ambiguities in the statements of theorems.

Fix a tubular neighborhood $N\approx S^1\times D^2$. Fix a metric on $S^3$ that is the flat product metric in $N_j$ and consider the standard complex structure in a neighborhood of the 0-section in $T^\ast N$. In these coordinates the shifted $L_K$ is simply $S^1\times \{\epsilon\}\times\R^2$. 

Let $u$ denote the basic annuli stretching between $L_K$ and $S^3$ which in local coordinates near $K$ have connected components $S^1\times[0,\epsilon]\times\{0\}$. Fix $\xi$ a  non-zero vector field on $S^{3}$ which is nowhere tangent to the link $K$. Consider the $4$-chain $C_{\xi}$ which is the union of the positive and negative half rays determined by $\xi$ and oriented in such a way that $\partial C_{\xi}=2[S^{3}]$. Since $\xi$ and the tangent vector $\tau$ to $K$ are everywhere linearly independent $C_{\xi}$ is disjoint from the interior of the basic annulus $u$ stretching between $L_K$ and $S^3$ and thus 
\begin{equation}\label{eq: I = 0}
I(u,C_\xi)=0,
\end{equation} 
where $I$ is as in Lemma \ref{l : intersections without shifting}. 

In order to calculate $u\OpenHopf S^3$ it remains to calculate $\partial u\OpenHopf \gamma$. The interior of $C_\xi$ is disjoint from $S^3$. Consider the intersection $\gamma=L_{K}\cap C_\xi$. We first look at the intersection 
$\gamma'=L_{K}'\cap C\subset T^{\ast} S^{3}$ where $L_{K}'$ denotes the unshifted conormal. It is easy to see that $\gamma'=K\cup_{p} \xi(p)^{\ast}_{\pm}$, where $p$ denotes the points where the tangent to $K$ is perpendicular to $\xi$ and $\xi^{\ast}_{\pm}$ denotes the positive and negative half-lines in the fiber of the conormal of $K$ at the point $p$ in direction of the co-vector dual to $\xi$. For the shifted conormal the intersection $\gamma$ is a smoothing of the curve $\gamma'$ that is disjoint from the central copy of $K$. We see in particular that $\gamma$ has (ideal) boundary and constructing a bounding chain correspond to picking paths at infinity closing $\gamma$ up. Changing the closing path by a suitably oriented meridian changes $\partial u\OpenHopf \gamma$ by $\pm 1$. It follows that there is a choice such that
\begin{equation}\label{eq: lk gamma = 0}
\partial u\OpenHopf \gamma = 0.
\end{equation}
Lemma \ref{l : intersections without shifting}, \eqref{eq: I = 0}, and \eqref{eq: lk gamma = 0} implies that for $C_\xi$ there is a choice of bounding chain for $\gamma$ such that
\begin{equation}\label{eq : u cdot C = 0}
u\OpenHopf S^3 = 0,
\end{equation}
where $u$ is the basic annulus stretching between $L_K$ and $S^3$.

\subsection{The count in $S^3$}

We compute $Z_{T^*S^3, L_K \cup S^3}$. This invariant takes values 
in a tensor product of skein modules of $S^3$ and solid tori $S^{1}\times\R^{2}$.  Recall the skein of $S^3$ is just the free rank one module over the coefficient ring, generated by the class of the empty knot $\varnothing$.  This assertion is essentially equivalent to the existence of the HOMFLYPT polynomial $\langle K \rangle_{S^3} \in \Z[a^\pm, z^\pm]$ of links $K$: 
$$
K = \langle K \rangle_{S^3} \, \cdot \, \varnothing \in \Sk(S^3).
$$

For the solid torus $S^{1}\times\R^{2}$, we recall from \cite{Turaev} that $\Sk(S^{1}\times\R^{2})$ is a free polynomial algebra over $\Z[a^\pm, z^\pm]$ with generators indexed by the integers.  
$$
\Sk(S^{1}\times\R^{2}) \cong \Z[a^\pm, z^\pm][\ldots, \ell_{-2}, \ell_{-1}, \ell_{0}, \ell_{1}, \ell_2, \ldots].
$$ 
The algebra can be viewed as a $\Z$-graded, with $\deg(\ell_i) = i$; the grading records 
the class in $H_1(S^{1}\times\R^{2})$ determined by the skein representative.  We write $\Sk^+(S^{1}\times\R^{2})$ and $\Sk^-(S^{1}\times\R^{2})$ for the
subalgebras on the positive and negative generators; $\Sk(S^{1}\times\R^{2}) =   \Sk^+(S^{1}\times\R^{2}) \otimes \Sk^-(S^{1}\times\R^{2})$. Then $\Sk^+(S^{1}\times\R^{2})$ is a free $\Z[a^\pm, z^\pm]$ module on a basis naturally indexed by partitions.  
In fact there is a natural identification of $\Sk^+(S^{1}\times\R^{2})$ with the ring of symmetric functions, 
but we will not need it here. 

In particular, if we fix a basis for the skein of the solid torus, e.g.~monomials in the $\ell_i$,
then we may extract the corresponding coefficients of $Z_{T^*S^3, S^3 \cup L_K}$; said coefficient
will take values in $\Sk(S^3)$. Here we will focus just on the simplest term. We fix the generator $\ell=\ell_{1}$ represented by $S^{1}\times\{0\}$ and write  $1^k \in \Z^k = H_1(L_K) =
H_2(T^*S^3, S^3 \cup L_K)$ for the homology class which is the sum of these generators for all components.

Fix 4-chain $C_\xi$ and bounding chain for $\gamma$ as in Section \ref{ssec : nice 4 chain}. 

Note that $H_2(T^\ast S^3,L_K)=H_1(L_K)$ and that the sum of positive generators of $H_1(L)$, written multiplicatively as $1^n$, is a basic homology class. 

\begin{theorem} \label{thm:HOMFLY}
We have the following equality in 
$\Sk(S^3 \cup L_K)$: 
\begin{equation} \label{stupid homfly}
Z_{T^\ast S^3, S^3 \cup L_K, 1^m}
= K \otimes \ell^{\otimes m}\end{equation} 
\end{theorem}

\begin{proof} 
Let $N=\bigcup_j N_j$ be a tubular neighborhood of $K$, where each $N_j$ is diffeomorphic to $S^1\times D^2$. Identify $T^\ast N_j$ with $S^1\times D^2\times \R^3$, where $S^3$ corresponds to $S^1\times D^2$ and $L_{K_j}$ to $S^1\times \{\delta\}\times\R^2$, where $\delta$ is the shift of $L_{K_j}$. Consider $S^1\times D^2\times \R^3$ as a subset of $\C^\ast\times \C^2$ and equip it with the standard integrable complex structure $J_0$. Let $J$ be a complex structure on $T^\ast S^3$ that agrees with $J_0$ in a neighborhood of the 0-section in $T^\ast N_j$ for each $j$. 

There is an obvious holomorphic cylinder stretching between $L_{K_j}$ and $S^3$, given in local coordinates by $S^{1}\times[-\epsilon,0]\times\{0\}$. We claim that for $\epsilon>0$ sufficiently small, the union $A=A_1\cup\dots\cup A_m$ of such cylinders are the unique holomorphic curves in the basic homology class $1^m$. 

By monotonicity, it is clear that any curve in a basic class must line inside $T^\ast N$ for all suffciently small $\epsilon>0$. Inside $N$ the complex structure is split. Projecting to the $T^\ast D^2\subset \C^2$-factor we get a holomorphic curve with boundary on $D^2\cup i\R^2$ which must be constant by the maximum principle. It follows that $A$ is the only curve in the class $1^m$. 

We next check  that $A$ is transversely cut out. The linearization of the $\bar{\partial}_{J_0}$-operator at $A_j$ is simply the standard $\bar{\partial}$-operator with boundary condition in $\R\times\R^{2}$ along one boundary component and $\R\times i\R^{2}$ along the other, there is a constant solution in the $\R$-direction (the linearization of rotations along $A_j$) and no solutions in the other directions. It follows that the curve is transversely cut out.           
We conclude that the coefficient of $\ell \otimes \ell \otimes \cdots \otimes \ell$ is just the count of $A=A_1\cup\dots\cup A_m$. The cylinders $A_j$ have Euler characteristic zero and, by our choice of $C_\xi$ and bounding chain for $\gamma$, $u\OpenHopf C_\xi=0$, see \eqref{eq : u cdot C = 0}. Hence the count is just $1^m$ times the boundary in $S^3$, which is the original link $K$ itself, $\langle K\rangle\in\Sk(S^{3})$. 
\end{proof}

\subsection{Proof of Theorem \ref{thm:basic}}
Applying Theorem \ref{thm:deleteS3} to Theorem \ref{thm:HOMFLY} gives 
the following equality in $\Sk(L_K)$: 
\begin{equation} \label{delete S3 equation} 
\sum_{d} 
Z_{\R\times ST^\ast S^3, L_K, 1^d\oplus d}
=  \langle K \rangle_{S^3}  \cdot \ell^{\otimes n}.
\end{equation}
Equation \eqref{eq:basic} then follows by applying Theorem \ref{thm:conifold}.

We turn to the assertion that there are almost complex structures for which the boundaries of all curves contributing to $Z_{X,L_K, 1^d}$ have boundaries framed isotopic to the link of central curves $\ell^m$ in $L_K$. We produce such almost complex structures by SFT-stretching around $L_K$. The details are as follows. 

Consider curves in a fixed fixed homology class. The area is then bounded and as discussed in Section \ref{ssec:SFTlimitnoncompactL}. Consider a metrics on each $L_{K_j}$ as in Example \ref{ex : stretch around solid torus}. By invariance of the curve count we can deform the almost complex structure in a small neighborhood of $L_K$ to the almost complex structure determined by the metric. Then as explained in Section \ref{ssec:SFTlimitnoncompactL}, SFT-stretching inside this small neighborhood around $L_K$, the boundary of any curve must have negative asymptotic at a Reeb orbit in $ST^\ast L_{K_j}$ for the contact form on $ST^\ast L_K$ compatible with the metric and such Reeb orbits are all multiples of the lift of the basic geodesic, i.e., the central curve $\ell$ in the solid torus. 
Further, the $a$-degree of the curve is given by counting intersections with a fiber that does not meet $L_K$. Since the area of the curve $u$ is 
$$
\deg(u)\cdot\mathrm{area}({\C\P^1})+\int_{\partial u}\tau,
$$
where $\tau=\sum_j d\theta_j$ for $d\theta_j$ on $L_{K_j}\approx S^1\times\R^2$ the standard form along the $S^1$-factor, we find for generic shift, that by assumption on primitive homology class, the only possibility is a single Reeb orbit in each conormal component. 

Consider now the lower level in the SFT-limit after stretching, i.e., the moduli space of curves in $T^\ast L_{K_j}$ with positive asymptotic a single Reeb orbit. For the almost complex structure corresponding to the metric on $K_j$, a straightforward area argument shows that there is only one curve with boundary on the unique non-degenerate geodesic and that it is transversely cut out. Thus, for sufficiently large stretching the boundary of any curve in a basic homology class is arbitrarily $C^1$-close to that of the standard cylinder, i.e., to $\ell$. The theorem follows. 
\qed

\begin{remark}
The curve counts on the left hand sides of Equations \ref{stupid homfly} and \ref{delete S3 equation} can be reinterpreted as counts of curves without boundaries as follows. Consider SFT-stretching around $L_{K}$ as in the last argument of the proof above. Then for sufficiently large stretching curves which have boundary on the zero section are in 1-1 correspondence with curves asymptotic to Reeb orbit $\gamma$ that corresponds to the geodesic which is a positive homology generator (by adding the basic cylinders over this geodesic). Since this SFT-stretching is a modification of the almost complex structure near $L_{K}$, the counts on the right hand side of \eqref{stupid homfly} and \eqref{delete S3 equation} can be identified with the count of punctured curves in $T^{\ast}S^{3}\setminus (S^{3}\cup L_{K})$. Here $\ell^{\otimes n}$ should be substituted by $\gamma^{\otimes n}$, i.e., the curves counted have negative punctures where it is asymptotic to the simple positive Reeb orbits in the components of $L_{K}$.  

However, the derivation of the (rather nontrivial) \eqref{delete S3 equation} from the (trivial seeming) \eqref{stupid homfly} used crucially the invariance of the skein-valued curve counting. 
\end{remark}

\section{Bare curves beyond basic classes} \label{sec: beyond}
In this article we have restricted  to basic homology classes, where somewhere-injectivity holds for essentially topological reasons, and so transversality for bare holomorphic curves can be achieved by changing the almost complex structure. In more general cases, transversality is more involved, but perturbations have been constructed in \cite{bare} to show that the main result holds.  In this section we give a brief description  from this more general point of view.

\subsection{General skein valued holomorphic curve counts}
In \cite{bare} we establish the analogue of Theorem \ref{thm:somewhereinjective}, where bare $J$-holomorphic curve for $J\in \mathcal{J}$ should are replaced by bare holomorphic curves $\bar\partial_J(u,S)=\lambda$ for $\lambda$ with properties as in \cite[Theorem 1.1]{bare}; the moduli spaces are constructed using the `polyfold' formalism \cite{HWZ, HWZ-GW}, and  are weighted branched manifolds (so give invariants in $\mathbb{Q}$ rather than $\mathbb{Z}$), but otherwise function as desired. 
We may then repeate verbatim the proof of
Theorem \ref{thm:invariance}, and deduce (for all $d$) the invariance of
\begin{equation}\label{eq:invpartition}
Z_{X, L, \lambda} := 1 + \sum_{d > 0} 
Z_{X,L,d;\lambda}
\cdot Q^d  \ \in \ \Sk(L) \otimes \Q[[H_2(X, L)]]
\end{equation}

As in Theorem \ref{thm:HOMFLY} we find also in the perturbed case that for small shift all holomorphic curves stretching between $L_K$ and $S^3$ in $T^\ast S^3$ map to the basic annuli. In the more general setup these curves can also be branched multiple covers of the basic annulus and it is more difficult to read off the contribution of such multi-covered annuli in the skein. We found in \cite{ekholm-shende-unknot}, using curves with Reeb chord asymptotics at infinity and recursion arguments, the contributions of such cylinders that we then used in \cite{ekholm-shende-colored} to establish the general version of the Ooguri-Vafa conjecture (general version of Theorem \ref{thm:basic}).

\begin{remark}
A key point of skein valued curve counting is the restriction to bare curves, i.e., curves without area zero components. The more conventional approach to counting holomorphic curves is to perturb the Cauchy-Riemann equation to $\partial u = \epsilon$ also on curves with stable domain but zero symplectic area, 
thus ensuring that no curve will literally map to a point, and then count the resulting solutions.  

Let us explain why such an approach would not directly
yield the skein relations as above.  Consider a 1-parameter family of maps $u_t\colon S_t \to X$ defined over $[0, \epsilon)$, which say at $t > 0$ are embeddings of smooth curves, while $S_0$ is a nodal curve with components some $S_+$ and an annulus $A$, glued at a boundary point, with $A$ having zero symplectic area.
In a setup where we have perturbed the constant curves, we may expect to find such degenerations.  

In order to see the skein relation as above, we would want to know that there was another family where the central fiber is a map from the normalization $S_+ \sqcup A$.  However, $A$ is symplectic area zero and unstable, and maps are not considered in the standard approaches to Gromov-Witten theory.  

Said differently, in order to see the skein relations from a more standard approach to Gromov-Witten theory, it appears that one must confront, or in an organized way avoid, the contributions of unstable maps. 
\end{remark}

\subsection{Closed and open invariants}
In Theorem \ref{thm:basic} we restrict the size of the central $\C\P^1$ in the resloved conifold in order to separate closed and open curves. In general no such separation is possible and with the perturbations in \cite{bare} it is not necessary. We give a brief discussion.

Consider first the empty Lagrangian and assume that we have a perturbation scheme for bare curves. Then $Z_X$ is the count of bare closed disconnected curves. (We show in \cite{bare} that the (exponentiated) usual Gromov-Witten invariants $Z_{\mathrm{GW}}$ can be expressed as $Z_{\mathrm{GW}}(g_s)=Z_{X}|_{z=e^{g_s/2}-e^{-g_s/2}}$.)

For non-empty Lagrangian $L$, one can then consider the a reduced invariant 
$$
Z'_{X,L}:=Z_{X, L} / Z_X;
$$
it is independent of choices because both numerator and 
denominator are.  However, $Z'_{X,L}$ does not generally have 
a direct and invariant enumerative interpretation: it may happen that
closed curves have non-vanishing intersection number
with the 4-chain for $L$, and hence contribute differently to the numerator and denominator. If closed curves do not intersect the 4-chain of $L$, as in our setup for the resolved conifold, then the interpretation of the reduced invariant as curves with components having non-empty boundary on $L$, does make sense.

\begin{remark}
In the absence of some geometric condition on closed curves as in Lemma \ref{reducedopen},
contributions of closed curves which intersect the 4-chain cannot be invariantly separated from
those of open curves with contractible 
boundary.  For example, suppose that we have found generic perturbation data such that there are no curves with boundary on $L$ and that no closed curve intersects the 4-chain. Then $Z'_{X,L}=Z_{X, L} / Z_X$, which counts only actual curves with boundary of which there none, satisfies $Z_{X, L}' = 1$. Now deform the data, e.g.~the 4-chain, in such a way that there is an instance when the 4-chain becomes tangent to one of the closed curves. After this instance that closed curve intersects the 4-chain in two points with opposite intersection signs. After further deformation one of these intersection points moves to the boundary $L$ of the 4-chain and at an elliptic crossing a new holomorphic curve with boundary on $L$ is born.  After this moment there is a unique nondegenerate curve with boundary on $L$.  Nevertheless, its contribution
is canceled by that of closed curves, and it remains
the case that $Z_{X,L}'=Z_{X, L} / Z_X = 1$. 
\end{remark}

\appendix

\section{Properties of holmorphic curves in basic classes} \label{somewhere injective review} 
In this section we review properties of holomorphic curves.

\subsection{Finite area curves in manifolds with cylindrical ends}\label{ssec : finite area curves}
In this section we show that finite area holomorphic curves in symplectic manifolds with cylindrical ends are confined to compact subsets. Let $X$ be a symplectic manifold and $L\subset X$ a Lagrangian. Assume that $(X,L)$ is asymptotically cylindrical. That is, outside a compact $(X,L)$ is diffeomorphic to $([0,\infty)\times Y,[0,\infty)\times\Lambda)$, where $Y$ is a contact manifold with contact form $\alpha$ and $\Lambda\subset Y$ a Legendrian submanifold, and the symplectic form $\omega$ on $X$ satisfies $\omega=d(e^t\alpha)+\omega_0$, for $t<T$, where $t$ is a coordinate on $[0,\infty)$, $\omega_0$ is bounded, and $L$ is the graph  over $\Lambda\times[0,\infty)$ of a bounded closed $1$-form. 

We note that if $L_K$ is the shifted conormal of a link $X=T^\ast S^3$ and $L=L_K\cup S^3$ then $(X,L)$ is asymptotically cylindrical, and that the same holds for the resolved conifold, $X=\mathcal{O}(-1)\oplus\mathcal{O}(-1)\to\C\P^1$ and $L=L_K$.

\begin{lemma}\label{l : curves not at infinity}
Let $(X,L)$ be asympotically cylindrical and let $J$ be an almost complex structure unifomly compatible with $\omega$, i.e., $\omega(v,Jv)\ge C\|v\|$ for a product Riemannian metric at infinity. Let $\gamma$ be any fixed homology class in $H_2(X,L)$. Then there exists $T>0$ such that no bare $J$-holomorphic curve in homology class $\gamma$ intersects $[T,\infty)\times Y$.
\end{lemma}

\begin{proof}
By monotonicity, there is a constant $K>0$ such that if a non-constant holomorphic curve intersects $\{t\ge T\}$ then it has area $>K T$. For sufficiently large $T$, $KT>\int_\gamma\omega$. The lemma follows.
\end{proof}

\subsection{General Fredholm properties and transversality for somewhere injective curves}\label{ssec : gen Fredholm injective}
In this section we recall basic properties of Fredholm maps in general position and that for somewhere injective holomorphic curves, general position can be achieved by varying the almost complex structure.
Let $X$ be an asymptotically cylindrical symplectic manifold and $L\subset X$ an asymptotically cylindrical Lagrangian in $X$. Let $\mathcal{J}$ denote the space of almost complex structures compatible with $\omega$. 

Let $\mathcal{V}$ be a configuration space of maps $u\colon(S,\partial S)\to (X,L)$. We will model this on Sobolev spaces with three derivatives in $L^{2}$ below so that derivatives are continuous. More concretely this means that $\mathcal{V}$ is a bundle of Sobolev spaces where we use weights that limit to exponential weights in cylinders and strips near the boundary of the space of curves. When studying evaluations we will sometimes consider $\mathcal{V}$ as a bundle over the product of the space of curves and some jet-bundle over the target space. We consider non-linear Fredholm problems $F\colon \mathcal{V}\times\mathcal{J} \to \mathcal{W}$, where $\mathcal{W}$ is a bundle of Sobolev spaces of complex anti-linear bundle maps $TS\to TX$. We will sometimes also multiply source and target by auxiliary finite dimensional spaces, $F\colon M\times\mathcal{V}\times{J}\to \mathcal{W}\times K$. The Fredholm problems have the form
\[ 
F(u,J,m) = (\bar\partial_{J}(u,S), k(u,S,m)),
\]
For fixed $J\in \mathcal{J}$ the linearization $d_{(u,S,m)}F$ of $F$ is a Fredholm operator of index
\[ 
\ind(d_{(u,S,m)}F) = \ind(\bar\partial_{J}) + \dim(M) -\dim(K).
\] 
This equality also holds at the level of index bundles in the sense that the index bundle of $F$ is given by adding the bundle difference $TM\ominus TK$ to the index bundle of $\bar\partial_{J}$. In particular, orientations on the index bundle of $\bar\partial_{J}$, on $M$, and on $K$ together induce an orientation on the index bundle of $F$.   

Write $dF$ for the full linearization of $F$, $dF\cdot\delta(u,S,m,J) =d_{(u,S,m)}F\cdot\delta(u,S,m)+d_{J}F\cdot\delta J$.

In this notation, that $\mathcal{J}$ gives transversality for the Cauchy-Riemann operator $\bar\partial_J$ can be stated as $(\delta u,\delta J)\mapsto d_{(u,S)}F\cdot \delta(u,S)+d_{J}F\cdot\delta J\in T\mathcal{W}$ is surjective. We recall standard transversality results for somewhere injective holomorphic curves. 

\begin{lemma}\label{l:basicsurjectivity}
Let $\gamma\in H_2(X,L)$ be a homology class such that every $J$-holomorphic curve in class $\gamma$ is somewhere injective then $\mathcal{J}$ gives transversality for the Cauchy-Riemann operator for curves in class $\gamma$.   

More generally, for the corresponding map $F\colon M\times\mathcal{V}\times \mathcal{J}\to\mathcal{W}\times K$,
if $\bar\partial_J(u,J)=0$ and $dk_{(u,S,m)}$ is surjective then $dF$ is surjective at $(u,S,m,J)$.  
\end{lemma}

\begin{proof}
 Transversality for Fredholm problems at somewhere injective curves is well-known. We recall the argument. If $u$ is a $J$-holomorphic curve then the linearization $L_{(u,S)}\bar\partial_{J}$, see \eqref{eq : linearization C-R operator}, of the $\bar\partial_{J}$-operator at $u$ is an elliptic operator. We consider the full linearization $L_{(u,S)}\bar\partial_{J}+d_{J}F$, where the second term corresponds to variations of $J\in\mathcal{J}$. If the full linearization is not surjective then by a partial integration argument a co-kernel element gives a solution $v$ of the dual elliptic equation which is orthogonal to all variations $d_{J}F\cdot\delta J$. By somewhere injectivity this means $v$ vanishes on the open set $V$ where $u$ is injective, and by unique continuation for the dual operator this implies $v$ vanishes identically. It follows that the full linearization is surjective.

 The second statement is an immediate consequence.
\end{proof}

Together with a general position argument,
Lemma \ref{l:basicsurjectivity} implies that if $\Delta$ is a $d$-dimensional manifold and $b\colon \Delta\to \mathcal{J}$ is a smooth map then after arbitrarily small homotopy, 
\begin{equation}\label{eq:fiberedacs}
\M(\Delta,k_{0})=\{(u,S,\delta)\colon u\in\M(b(\delta)),\; k(u,S)=k_{0}\}
\end{equation} 
is a transversely cut out manifold of dimension $\ind(d_{(u,S)}F)$. 

\begin{remark}
The Fredholm theory described above applies directly to maps $u\colon(S,\partial S)\to(X,L)$ with stable domains. Stable maps with unstable domains (disks and annuli) must first be stabilized. We use somewhere injectivity 
see Lemmas \ref{l : somwhere injectivity exact} and \ref{l : somwhere injectivity exact} to stabilize domains by adding sufficiently many marked points in local hypersurfaces in regions where the maps are injective gives stable domains and we use Sobolev spaces of maps with two derivatives in $L^{2}$ on these domains to convert the moduli problem for unstable domains to an equivalent moduli problem of maps from stable domains.
\end{remark}

\subsection{Basic orientations}\label{sec : basic orientations}   
To orient the space of solutions to the equation $F(u,s)=0$ in the case when the finite dimensional spaces added to the source and target are oriented we must orient the index bundle of $\bar\partial_{J}$. In the case of closed holomorphic curves this is straightforward since the index bundle is a complex bundle. As shown in \cite{FOOO} in the case of maps from the disk into a symplectic manifold $X$ with Lagrangian boundary condition $L\subset X$, the situation in the open case is different and involves the following relative spin condition on the Lagrangian: the second Stiefel Whitney class of the tangent bundle $TL$ to $L$, $w_2(L)\in H^{2}(L;\Z_2)$ should be the restriction of a class $\mathrm{st}\in H^{2}(X)$. In \cite[Theorem 8.1.1]{FOOO} it is shown that the index bundle $\ind(\bar\partial_{J})$ over the space of maps from the disk is orientable if $L$ is relatively spin and that a choice of a relative spin structure on $L$ determines an orientation.  

In the case under consideration here, orientable Lagrangians in Calabi-Yau threefolds, the relative spin condition is trivially met: an orientable $3$-manifold has trivial tangent bundle. Note that an orientation and a spin structure on $L$ is the same thing as a trivialization of $TL$ up to homotopy. The generalization of the orientation results from \cite{FOOO} is straightforward. We give a brief discussion. Let $\ind_{\chi,h}(\bar\partial_{J})$ denote the index bundle over the space of maps $u\colon (S,\partial S)\to (X,L)$ where $S$ is a connected Riemann surface of Euler characteristic $\chi$ with $h$ boundary components.

\begin{lemma}\label{l : basic orientations}
The index bundle $\ind_{\chi,h}(\bar\partial_{J})$ is orientable and the choice of an orientation and spin structure of $L$ determines an orientation. 
\end{lemma}    

\begin{proof}
We follow \cite[Chapter 8]{FOOO}. The linearized operator $L\bar\partial_{J}$ over $S$ can be homotoped to a linearized operator $L'\bar\partial_{J}$ over a closed surface $S'$ with $h$ disks $D_1,\dots, D_h$ attached, where the bundle and boundary condition are trivialized on each $D_k$ to $(\C^{3},\R^{3})$ and the operator agrees with the standard $\bar\partial$-operator. (This is where the trivialization of $TL$ is used.) Now $\ind(L'\bar\partial_J)$ is oriented as a complex bundle and $\ind(\bar\partial)=\det(\R^{3})$ over each disk $D_k$. To obtain $\ind(L\bar\partial_J)$ we need to glue the bundles at the points where $D_k$ are attached to $S'$. This uses a standard linear gluing argument, see e.g., \cite[Lemma 7.2]{ES}. Introduce small positive exponential weights and explicit solutions at the marked points viewed as punctures and glue. Here the twist parameter in the gluing extends over the disk and taking this automorphism into account we find that the index space $\ind'(\bar\pa)$ of the disks have dimension $2=3-1$. We then have 
\[
\ind_{\chi,h}(\bar\partial_J) =  \ind(L'\bar\partial_J) \otimes \bigotimes_{k=1}^{h} (\ind'(\bar\partial)\otimes \det(\C^3)),
\]
where the last factors are the gluing data at each puncture: require the values of sections to agree where glued and take out the twist automorphism.

The first factor is complex and factors in the last tensor product are even dimensional, so the orientation is independent on the ordering of boundary components. The lemma follows.  
\end{proof}

\subsection{Jet extensions and evaluation maps}\label{ssec:jetev}
With the basic Fredholm properties established we now describe how to show that generic holomorphic curves have the geometric properties needed to define the skein valued Gromov-Witten invariant. The ideas used are standard. We use evaluation maps of various kind and add finite dimensional spaces to the target and source of our maps. 

We describe the setup for showing that the boundary of our holomorphic curves give framed links in the Lagrangian (that thus give elements in the framed skein module). The mapping spaces we use are Sobolev spaces of maps $u\colon(\Sigma,\partial\Sigma)\to(X,L)$ with three derivatives in $L^{2}$. The smoothness assumption means that the restriction $\partial u\colon\partial\Sigma\to L$ has $\frac{5}{2}$ derivatives in $L^{2}$ and, in particular, the derivative $\frac{d}{ds}(\partial u)$ is continuous and we consider the $1$-jet evaluation along the boundary
\[ 
j^{1}(\partial u)\colon \partial\Sigma\to J^{1}(\partial\Sigma,L).
\]
In the language above this means that we add the boundary of the surface to the source space of the Fredholm operator $s\in \partial \Sigma$,  and the $1$-jet space of maps $\partial\Sigma\to L$ to the target space, $t\in J^{1}(\partial\Sigma,L)$.  

\begin{lemma}\label{l:immonbdry}
For generic 1-parameter families in $\mathcal{J}$ all bare holomorphic curves are immersions on the boundary.
\end{lemma}

\begin{proof}
Formal maps of vanishing differential forms a codimension $3$ subvariety of $J^{1}(\partial\Sigma,L)$. Thus adding a boundary marked point and the $1$-jet extension to our Fredholm problem as described above, \eqref{eq:fiberedacs} implies that maps with vanishing derivative somewhere along the boundary does not appear in generic 1-parameter families (and at isolated points in generic 2-parameter families).  
\end{proof}

\begin{lemma}\label{l:genframing} 
For generic $J\in\mathcal{J}$ the boundary of any bare holomorphic curve is nowhere tangent to $\xi$. 	
For generic 1-parameter families in $\mathcal{J}$ there is a finite set of points where the tangent vector to the boundary curve of a bare holomorphic curve is tangent to the reference vector field $\xi$. In a neighborhood of any such instances the 1-parameter family is conjugate to the standard tangency family of Lemma \ref{l:tangency}. 
\end{lemma}

\begin{proof}
	Outside the zeros $Z(\xi)$ of the vector field $\xi$, $\xi$ determines a codimension two subvariety $N'_{\xi}$ of $J^{1}(\partial\Sigma,L)$ where the formal differential has image in the line spanned by $\xi$. Let $N_{\xi}\subset J^{1}(\partial\Sigma,L)$ denote the union of the fibers over the zeros of $\xi$ and $N'_{\xi}$. As above, extending the Fredholm problem with the 1-jet extension of the evaluation map $j^{1}(\partial u)$, we find the following. In generic 0-parameter families the tangent vector of $\partial u$ is everywhere linearly independent from $\xi$. In generic 1-parameter families there are isolated instances where the 1-jet extension intersects $N_{\xi}$ transversely in its smooth top dimensional stratum. Such instances correspond to the versal deformation of a tangency with $\xi$, see the local model in the proof of Lemma \ref{l:tangency}.   
\end{proof}

In order to prove that boundaries are embedded we consider the multi-jet extension
\[ 
j^{0}\times j^{0}(\partial u)\colon \partial\Sigma\times\partial\Sigma\to
J^{0}(\partial\Sigma\times\partial\Sigma,L\times L),
\]
adding $\partial\Sigma\times\partial\Sigma$ to the source and $J^{0}(\partial\Sigma\times\partial\Sigma,L\times L)$ to the target.

\begin{lemma}\label{l:genemb}
	For generic $J$ all bare holomorphic curves in $\M(J)$ are embeddings when restricted to the boundary.
	For generic 1-parameter families in $\mathcal{J}$ there is a finite sets where the boundary curves have double points and the 1-parameter family gives a versal deformation of the double point. 
\end{lemma}

\begin{proof}
Note that the diagonal $\Delta_{L}\subset L\times L$ gives a codimension $3$ subvariety of $J^{0}(\partial\Sigma\times\partial\Sigma,L\times L)$ and that the immersion condition implies that there is a neighborhood $U\subset\partial\Sigma\times\partial\Sigma$ such that $(j^{0}\times j^{0})^{-1}(U\setminus\Delta_{\partial\Sigma})=\varnothing$. Transversality of this evaluation map than implies that $(j^{0}\times j^{0})^{-1}(\Delta_{L})\setminus \Delta_{\partial\Sigma}$ is empty for generic $0$-parameter families, and consists of isolated instances of generic crossings corresponding to the versal deformation of the intersection point where the tanget vectors are linearly independent.   
\end{proof}

Finally, we consider the interior of the curve in a similar way. 

\begin{lemma}\label{l:gen4chain}
	For generic $J$ all bare holomorphic curves are embeddings transverse to the 4-chain $C$.
	For generic 1-parameter families in $\mathcal{J}$ this holds except at isolated instances where some holomorphic curve is quadratically tangent to the 4-chain or some curve intersect $L$ at an interior point. Near such instances the 1-parameter family is a versal deformation.
\end{lemma}

\begin{proof}
The proof is similar to the arguments above. We  add an interior marked point to the source and the 1-jet space $J^{1}(\Sigma,X)$ to the target to find that solutions are immersions and  have transversality properties with respect to $C$ and $L$ as stated. To see the embedding property we then add $\Sigma\times\Sigma$ to the source and $J^{0}(\Sigma\times\Sigma,X\times X)$ to the target.
\end{proof}

\subsection{Orientation twisting and different skein relations}
Our proof of the skein relation implicitly fixes choices of the orientation of the Lagrangian and a spin structure. It is easy to check that the skein relation is invariant under these choices. For example for the orientation of the Lagrangian, $a$ should be replaced by $a^{-1}$ and crossing signs change which means that the left hand sides of the skein relations change sign. The mod $2$ number of boundary components differ between the left and right hand side in the skein equations and that by definition of the Fukaya orientation there is then an additional orientation sign difference between the left and right hand side and thus the skein relations are preserved.

One orientation choice  does affect the skein relation: the choice of orientation of the complex plane itself, which induces the orientation on the moduli space of curves. Here we first note that this reversal of orientation does \emph{not} change the orientation input from gluing a disk in the Fukaya orientation: the change of orientation of the disk changes its boundary orientation but then the orientation of the quotient of the orientation changes complex gluing parameter by the orientation changed rotation along the boundary remains unchanged. For the closed curve component the orientation change is $(-1)^{1-g}$. Now at an elliptic boundary crossing the genus $g$ stays constant and the number of boundary components change by $+1$ and there is no sign difference. In the case of hyperbolic boundary splitting there are two cases, when the crossing branches belong to the same boundary component the number of boundary components increase and the genus remains the same which means no sign change. However, in the case when the crossing branches belong to different components the number of boundary components decrease and the genus increases by $1$ which introduces a sign. Thus when the orientation of the complex plane underlying the space of stable domains is changes we get new skein relations, the elliptic case is the same but the hyperbolic version splits into two different relations that differ by a sign in the right hand side depending on whether the number of boundary components increase $(+)$ or decrease $(-)$.          

\section{Regularity at marked points and punctures}
In the gluing results relating crossing and nodal 1-parameter families of curves we introduce marked points and then consider the marked points as punctures. Near marked points our mapping spaces are usualk Sobolev spaces but near punctures we have half infinite strips (boundary punctures) or cylinders interior punctures. In this section we review basic relations between Sobolev norms in this set up. The results are formulated for cylinders, the case of strips follows by doubling.

We compare the regularity requirements for Sobolev spaces $H^{m}$ and $H^{m,\delta}$, see Sections \ref{sssec : hc} and \ref{sssec : ec}. Let $S$ be a Riemann surface and $p\in S$. 
Let $s+it\in[0,\infty)\times S^1$ and let $z=e^{-2\pi(s+it)}$ be coordinates in a neighborhood $D$ of $p\in S$. 
Consider a map $u\colon D\setminus\{p\}\to\R^{2n}$ 
and let $u_{\infty}\colon [0,\infty)\times S^1\to\R^{2n}$ 
and $u_{0}\colon D^{\circ}\to\R^{2n}$ be the corresponding maps in local coordinates. 
Take the weight $\delta\in(0,2\pi)$. We have the following basic relations between regularity requirements. 

\begin{lemma}\label{l:markedvspuncturenorms}
The integral norms are related as follows:
\begin{align}\label{eq:ref pull-back L^2} 
\int_{D}|u_{0}|^{2} \,dxdy  \ &= \ \int_{[0,\infty)\times I} |u_{\infty}|^{2} e^{-4\pi s} \,dsdt,\\\label{eq:push forward L^2}
\int_{[0,\infty)\times I}|u_{\infty}|^{2}\, dsdt \ &= \ \int_{D} |z|^{-2}|u_{0}|^{2} \,dxdy.
\end{align}
and
\begin{align}\label{eq:ref pull-back L^2 on derivatives}
\int_{D}|D^{(k)}u_{0}|^{2} \,dxdy \ &\asymp \ \sum_{j=1}^{k}\int_{[0,\infty)\times I} |D^{(j)}u_{\infty}|^{2} e^{4\pi(j-1)s}\,dsdt,\\\notag 
\int_{[0,\infty)\times I}|D^{(k)}u_{\infty}|^{2}\,dsdt \ &\asymp \ 
\sum_{j=1}^{k}\int_{D} |z|^{j-2}|Du_{0}^{(j)}|^{2} \,dxdy, 
\end{align}
where for positive functions $a,b$, we write $a \asymp b$ to mean that there exists constants $c,C>0$ such that $c a(x)< b(x)< Ca(x)$ for all $x$.
\end{lemma}

\begin{lemma}\label{l:compareL2}
Assume that $u_{\infty}$ and $Du_{\infty}$ both lie in $L^{2}_{\delta}$ then the following hold.
\begin{itemize}
\item[$(i)$] $u_{0}$ and $Du_{0}$ both lie in $L^{2}$.
\item[$(ii)$] For any $|\alpha|\ge 2$, if $D^{\beta}u_{0}$ lies in $L^{2}$ for $|\beta|\le|\alpha|$ then $D^{\alpha}u_{\infty}$ lies in $L^{2}_{\delta}$. 
\end{itemize}	
\end{lemma}

\begin{proof}
We have 
\begin{alignat*}{2}
d \log z &= \frac{dz}{z} = - 2 \pi (ds + i dt),\qquad
z \partial_{z} &&= - \frac{1}{2\pi} \left(\partial_{s} 
- i \partial_{t} \right),\\
\partial_{s} &= 
 - 2 \pi \left( z\partial_{z}  + \bar z \partial_{\bar z}\right),\qquad\qquad
\partial_{t} &&= - 2 \pi i \left( z \partial_{z}  - \bar z \partial_{\bar z} \right),
\end{alignat*}
Also $|z| = e^{-2 \pi s}$ and so $e^{\delta s} = |z|^{-\delta/2\pi}$. Thus, 
$$
\int_{[0,\infty)\times I}|u_{\infty}|^{2}e^{2\delta s} \,dsdt \ \sim \
\int_{|z| < 1} |u_{0}|^2 |z|^{-2\left(\frac{\delta}{2\pi} \right)}  |z|^{-2} \,dz d \bar{z}
$$
and, if $|\alpha|>0$,
$$
\int_{[0,\infty)\times I}|D^{\alpha}_{s,t} u_{\infty}|^{2}e^{2\delta s} \,dsdt \ \sim \
\sum_{|\beta|=1}^{|\alpha|}
\int_{|z| < 1} |z|^{2 |\beta|} |D^{\beta}_{z, \bar{z}} u_{0}|^2 |z|^{-2\left(\frac{\delta}{2\pi} \right)}  |z|^{-2} \,dz d \bar{z},  
$$
where $\sim$ means up to multiplication by constant.
The result follows.
\end{proof}

Using the Sobolev embedding theorem we get more information. 
\begin{lemma}\label{l:H3toH3delta}
Let $m\ge 3$. If $u_{0}$ has $m$ derivatives in $L^{2}$ then $u_{\infty}$ has $m$ derivatives in $L_{\delta}^{2}$.
\end{lemma}

\begin{proof}
By Lemma \ref{l:compareL2} $(ii)$, it suffices to estimate the $L^{2}_{\delta}$-norm of $u_{\infty}$ and $Du_{\infty}$. By Sobolev embedding, $u_{0}$ lies in $C^1(D)$ and we have $|Du_{0}| < C$ for some constant $C$. Therefore, $|u_{0}(z)| < C|z|$ and thus, 
$$\int |u_{\infty}(s+it)|^2 e^{2 \delta s}\, ds dt \ < \ \int C e^{- 4 \pi s} e^{2 \delta s}\, ds dt \ = \ 
C \int e^{2 (\delta - 2\pi)s}\, ds dt  \ < \ \infty $$
and 
$$ 
\int \left|Du_{\infty} \right|^2 e^{2 \delta s}\,  ds dt \ \sim \  
\int \left(\left |z \partial_{z} u_{0} \right| +  \left|\bar z \partial_{\bar z} u_{0}\right| \right) ^2 
e^{2 \delta s}\,  ds dt  \ \sim \ \int e^{-4\pi s} e^{2 \delta s}\, ds dt \ < \ \infty.$$
\end{proof}

\bibliographystyle{plain}
\bibliography{skeinrefs}

\end{document}